\documentclass[11pt,oneside,a4paper]{amsart}

\usepackage{pgfplots}
\pgfplotsset{compat=1.15}
\usetikzlibrary{arrows}
\usepackage{amsthm,amssymb,latexsym,color,amsmath,mathrsfs}
\usepackage[all]{xy}
\usepackage{pb-diagram}
\usepackage{tikz}
\usepackage{pstricks-add}
\usepackage{pst-func}
\usepackage{color}
\usepackage[english]{babel}
\usepackage{verbatim}
\usepackage{soul}
\usepackage{csquotes}
\usepackage[backend=biber,style=alphabetic]{biblatex}
\newcommand{\Z}{{\mathbb Z}}

\newcommand{\R}{\mathbb{R}}

\newcommand{\p}[1]{{\mathbb{P}^{#1}}}

\newcommand{\op}[1]{{\mathcal O}_{\mathbb{P}^{#1}}}

\newcommand{\ox}{{\mathcal O}_{X}}

\newcommand{\ch}{\operatorname{ch}}

\newcommand{\hd}{\operatorname{hd}}

\newcommand{\stab}{\operatorname{Stab}}
\newcommand{\sing}{\operatorname{Sing}}
\newcommand{\supp}{\operatorname{Supp}}

\newcommand{\cala}{{\mathcal A}}
\newcommand{\calb}{{\mathcal B}}

\newcommand{\calf}{{\mathcal F}}
\newcommand{\calg}{{\mathcal G}}
\newcommand{\calh}{{\mathcal H}}
\newcommand{\cali}{{\mathcal I}}

\newcommand{\calm}{{\mathcal M}}

\newcommand{\calo}{{\mathcal O}}
\newcommand{\calp}{{\mathcal P}}

\newcommand{\calt}{{\mathcal T}}

\newcommand{\dbx}{D^{b}(X)}
\newcommand{\coh}{\mathcal{C}oh}

\newcommand{\cohb}{\mathcal{B}^{\beta}}
\newcommand{\cohab}{\mathcal{A}^{\beta,\alpha}}

\newcommand{\labs}{\lambda_{\beta,\alpha,s}}
\newcommand{\free}[1]{\mathcal{F}_{#1}}
\newcommand{\tors}[1]{\mathcal{T}_{#1}}

\newcommand{\obeta}{{\overline{\beta}}}
\newcommand{\oalpha}{\overline{\alpha}}

\newcommand{\inhom}{{\mathcal H}{\it om}}
\newcommand{\inext}{{\mathcal E}{\it xt}}

\newcommand{\Ext}{\operatorname{Ext}}
\newcommand{\Hom}{\operatorname{Hom}}

\DeclareMathOperator{\coker}{coker}
\DeclareMathOperator{\im}{im}

\DeclareMathOperator{\rk}{{rk}}

\DeclareMathOperator{\Pic}{{Pic}}

\newcommand{\lra}{\longrightarrow}
\newcommand{\into}{\hookrightarrow}
\newcommand{\onto}{\twoheadrightarrow}

\newcommand{\Ap}{\mathcal{A}^p}
\newcommand{\Ac}{\mathcal{A}}
\newcommand{\Tc}{\mathcal{T}}
\newcommand{\Fc}{\mathcal{F}}
\newcommand{\wt}{\widetilde}
\newcommand{\wh}{\widehat}
\newcommand{\ol}{\overline}
\newcommand{\dimension}{{\mathrm{dim}\,}}
\newcommand{\image}{{\mathrm{im}\,}}

\newtheorem{theorem}{Theorem}[section]
\newtheorem{mthm}{Main Theorem}
\newtheorem{proposition}[theorem]{Proposition}

\newtheorem{lemma}[theorem]{Lemma}
\newtheorem{corollary}[theorem]{Corollary}

\theoremstyle{definition}
\newtheorem{remark}[theorem]{Remark}
\newtheorem{example}[theorem]{Example}
\newtheorem{definition}[theorem]{{\bf Definition}}
\newtheorem{construction}[theorem]{Construction}

\usepackage{yfonts}

\title[Higher DT/PT]{Higher rank DT/PT wall-crossing in Bridgeland stability}

\author{Marcos Jardim}
\address{IMECC - UNICAMP \\ Departamento de Matemática \\
Rua Sérgio  Buarque de Holanda, 651\\ 13083-970 Campinas-SP, Brazil}
\email{jardim@ime.unicamp.br}

\author{Jason Lo}
\address{Department of Mathematics \\
California State University, Northridge\\
18111 Nordhoff Street\\
Northridge CA 91330 \\
USA}
\email{jason.lo@csun.edu}

\author{Antony Maciocia}
\address{School or Mathematics and Maxwell Institute for Mathematical Sciences\\ University of Edinburgh\\ Peter Guthrie Tait Road\\ Edinburgh\\ EH9 3FD\\ UK}
\email{A.Maciocia@ed.ac.uk}

\author{Cristian Martinez}
\address{School of Engineering, Science and Technology \\Universidad del Rosario\\
Carrera 6 No. 12C-16, 111711, Bogotá, Colombia}
\email{cristianm.martinez@urosario.edu.co}

\begin{document}

\definecolor{uuuuuu}{rgb}{0.26666666666666666,0.26666666666666666,0.26666666666666666}
\definecolor{qqqqff}{rgb}{0,0,1}
\definecolor{ffqqqq}{rgb}{1,0,0}

\begin{abstract}

We prove that the Gieseker moduli space of stable sheaves on a smooth projective threefold $X$ of Picard rank 1 is separated from the moduli space of PT stable objects by a single wall in the space of Bridgeland stability conditions on $X$, thus realizing the higher rank DT/PT correspondence as a wall-crossing phenomenon in the space of Bridgeland stability conditions. In addition, we also show that only finitely many walls pass through the upper $(\beta,\alpha)$-plane parametrizing geometric Bridgeland stability conditions on $X$ which destabilize Gieseker stable sheaves, PT stable objects or their duals when $\alpha>\alpha_0$.
\end{abstract}
\maketitle

\tableofcontents


\section{Introduction}

Let $X$ be a smooth projective threefold of Picard number one for which the generalized Bogomolov-Gieseker inequality holds. The Donaldson--Thomas (DT) invariants give virtual counts of curves $C$ (or equivalently, their ideal sheaves $I_C$) on $X$, while Pandharipande--Thomas (PT) invariants give virtual counts of pairs $(F,s)$ consisting of a 1-dimensional sheaf $F$ and a section $s\in H^0(F)$ whose cokernel is 0-dimensional, so-called \textit{PT stable pairs}. In \cite{pandharipande2009curve}, it was conjectured that there is an equivalence between the DT and the PT invariants, in the sense that the generating series for these invariants can be converted to each other via correspondence formulas; this equivalence came to be known in the literature as the \textit{DT/PT correspondence}.

When $X$ is a Calabi--Yau threefold, the Euler characteristic version of the correspondence was proved by Toda \cite{toda2010curve} under a technical assumption on the local structures of the relevant moduli stacks (the assumption was later removed in \cite{toda2020hall}), while the full conjecture (using Behrend functions) was eventually proved by Bridgeland in \cite{bridgeland2011hall}. Independently, Stoppa--Thomas \cite{stoppa2011hilbert} also proved the same version of the DT/PT correspondence, notably noticing that there is, for an arbitrary threefold $X$, a GIT wall-crossing relating the Hilbert scheme of curves on $X$ and the moduli space of PT stable pairs.   

Note that ideal sheaves of curves are rank 1 slope stable sheaves, while PT stable pairs are objects in the derived category that are equivalent to rank 1  stable objects in the sense of \textit{PT stability} as defined by Bayer in \cite{BayerPBSC}. Besides, counting invariants can be defined for higher-rank stable sheaves on a smooth projective Calabi--Yau threefold and for higher-rank PT stable objects \cite{Lo2,toda2020hall}. It is then natural to ask whether the DT/PT correspondence also holds in \emph{higher ranks}.  Indeed, under a coprime assumption on rank and degree, this \textit{higher-rank DT/PT correspondence} was proved by Toda in \cite{toda2020hall}, thus generalising the original (rank 1) DT/PT correspondence.

It has long been suggested, however, that the DT/PT correspondence could be a manifestation of a wall-crossing phenomenon in the space of Bridgeland stability conditions on the threefold $X$ (e.g.\ see \cite[Section 3.3]{pandharipande2009curve}), i.e.\ there should be a wall in the Bridgeland stability manifold separating the moduli of ideal sheaves and the moduli of PT stable pairs. In fact, in \cite{bridgeland2011hall}, Bridgeland shows that for the rank 1 case, both DT and PT invariants count quotients of $\mathcal{O}_X$ in two different hearts related by a tilt. In \cite{BayerPBSC}, where the concept of polynomial stability conditions was introduced, Bayer showed that there are two standard polynomial stability conditions called DT stability and PT stability that are separated by a wall in the space of polynomial stability conditions \cite[Proposition 6.1.1]{BayerPBSC}, and where the rank-one DT stable objects (respectively\ PT stable objects) are precisely the ideal sheaves of curves (respectively\ PT stable pairs).  Therefore, if one accepts the intuition that polynomial stability conditions correspond to \textit{limits} of Bridgeland stability conditions, then Bayer's result gave one step towards realising DT/PT correspondence as a wall-crossing in the complex manifold ${\rm Stab}(X)$ of Bridgeland stability conditions.

In this paper, we complete the picture for a smooth projective threefold of Picard rank 1 by producing a wall in ${\rm Stab}(X)$ where, for any rank, we have the moduli of PT stable objects on one side of the wall and the moduli of stable sheaves on the other side.

More precisely, let $X$ be a smooth projective threefold of Picard rank 1 for which geometric stability conditions $(\cala^{\beta,\alpha},Z_{\beta,\alpha,s})$ as constructed by Bayer, Macrì, and Toda in \cite{BMT1,BMT} are known to exist; at the moment of writing, the list includes Fano threefolds \cite{M-p3,S-q3,Li}, abelian threefolds \cite{MP,BMS}, quintic threefolds \cite{Li2}, general weighted hypersurfaces in either of the weighted projective spaces $\mathbb{P}(1,1,1,1,2)$ or $\mathbb{P}(1,1,1,1,4)$ \cite{KosDoubleTriple}, and Calabi--Yau threefolds obtained as a complete intersection of quadratic and quartic hypersurfaces in $\mathbb{P}^5$ \cite{LiuBGICY3}. Fix a Chern character vector
$$ v=(v_0,v_1,v_2,v_3) \in \mathbb{Z}\times \mathbb{Z}\times \frac{1}{2}\mathbb{Z}\times \frac{1}{6}\mathbb{Z}, $$
and let ${\mathcal{G}(v)}$, ${\mathcal{M}_{\beta,\alpha,s}(v)}$ and ${\mathcal{T}(v)}$ denote the moduli spaces of Gieseker semistable sheaves, Bridgeland semistable objects and PT semistable objects, respectively, with class $v$. Finally, for any fixed $s\in\R^+$, set
$$ \Theta_v:=\{(\beta,\alpha)\in\R\times\R^+ ~|~\ {\rm Im}\big(Z_{\beta,\alpha,s}(v)\big)=0\} $$ 
regarded as a curve in the $(\beta,\alpha)$ plane parametrizing geometric stability conditions.

In addition, let $E$ be a torsion-free sheaf on $X$ and set $Q_E:=E^{**}/E$; let $Z_E$ be the maximal 0-dimensional subsheaf of $Q_E$, and let $E'$ be the kernel of the composed epimorphism $E^{**}\onto Q_E\onto Q_E/Z_E$; see Section \ref{sec:sheaves} for further details. 

We prove:

\begin{mthm}\label{mthm1}{\rm [\textbf{= Theorem \ref{thm:wall-xing}}]}
    Let $v$ be a Chern character vector with $v_0>0$ and such that $v_0$ and $v_1$ are relatively prime. Fix $s>1/3$ and let $(\beta_v,\alpha_v)\in\Theta_v$ be the last tilt wall for $v$ along $\Theta_v$. For any point $(\obeta,\oalpha)\in\Theta_v$ with $\obeta<\beta_v$ the Bridgeland wall-crossing, via restriction to the locus of PT stable objects, produces a diagram
    $$
    \begin{diagram}
     \node{\mathcal{G}(-v)}\arrow{se,b}{\gamma}\node{}\node{\mathcal{T}(v)}\arrow{sw,b}{\tau}\\
     \node{}\node{\mathcal{M}_{\obeta,\oalpha,s}(v)}\node{}
    \end{diagram}
    $$
with morphisms given by
$$ \gamma(E) = \big[E'[1]\oplus Z_E\big]  ~~{\rm and}~~ \tau(A) = \big[\calh^{-1}(A)[1]\oplus \calh^{0}(A)\big]. $$
\end{mthm}

We highlight that the morphism $\gamma$ above can be regarded as analogous to the so-called \textit{Gieseker-to-Uhlenbeck map} for surfaces originally described by Li in \cite{Li-Jun} and recently recast in terms of a morphism between Bridgeland moduli spaces by Tajakka in \cite{Tajakka}. Indeed, let $E$ be a Gieseker semistable sheaf on a surface $S$; Tajakka's version of the Gieseker-to-Uhlenbeck map goes from the Gieseker moduli space $\mathcal{G}(v)$ to the Bridgeland moduli $\mathcal{M}_{\beta,\alpha}(v)$ space for any $\alpha>0$ and $\beta=v_1/v_0$, taking $E$ to (the S-equivalence class of) the object $E^{**}\oplus(E^{**}/E)[-1]$, which is a Bridgeland semistable object at the vertical wall given by $\ch_1^\beta(E)=0$. Our proposal for the Gieseker-to-Uhlenbeck map for threefolds is precisely the map $\gamma$ above.

It is also worth noticing that Pavel and Tajakka recently claimed that for any smooth projective threefold $X$, the moduli space $\calt(v)$ of PT-stable objects on $X$ is a projective scheme whenever $v_0$ and $v_1$ are coprime, see \cite[Theorem 1.1]{Pavel-Tajakka}.  The moduli space $\calt(v)$ was originally constructed in \cite{Lo1, Lo2} as a universally closed algebraic stack of finite type.  In the absence of strictly semistable objects, $\calt(v)$ admits a proper coarse moduli space.  Using methods on determinantal line bundles that originated in the works of Le Potier \cite{le1992fibre} and Li \cite{Li-Jun} and further developed in Tajakka \cite{Tajakka}, Pavel and Tajakka constructed an ample line bundle on the coarse moduli space of $\calt(v)$ under a coprime assumption on $v_0, v_1$, thereby proving it is a projective moduli space \cite{Pavel-Tajakka}.  Since we prove in this article that PT stability \emph{is} a Bridgeland stability condition (see Theorem \ref{t:PTisBridgeland} below), our results yield a class of examples of projective moduli spaces of Bridgeland stable objects on threefolds.  It would be interesting to check whether the  Bridgeland moduli spaces $\mathcal{M}_{\obeta,\oalpha,s}(v)$ can also be shown to be projective.

Implicit in the statement of Main Theorem \ref{mthm1} is the existence of two special stability chambers in ${\rm Stab}(X)$ separated by the wall $\Theta_v$: the \textit{Gieseker chamber}, for which there was already evidence in \cite{JM19}, in which the Bridgeland moduli space coincides with the Gieseker moduli space of semistable sheaves, and the \textit{PT chamber}, in which the Bridgeland moduli space coincides with the moduli space of PT semistable objects. 

Moreover, for a fixed $s\geq 1/3$ and $\beta\ll 0$, PT stable objects of Chern character $v$ can be found not only near $\Theta_v$ but also for $\alpha\gg 0$, as they satisfy the conditions of \cite[Theorem 3.1]{JMM}. This would be consistent with an absence of walls for $v$. This motivated us to complete the paper by looking at how the geometry of the walls corresponding to the PT stable objects is constrained and deduce that there can only be finitely many such walls. More precisely, we show:

\begin{mthm}\label{mthm2}{\rm [\textbf{= Theorem \ref{t:finite}}]}
    For any $s\geq1/3$, $\alpha_0>0$ and any Chern character $v$ with $\Delta(v)\geq0$ and $v_0>0$, there are only finitely many actual $\lambda$-walls passing through the upper $(\beta,\alpha)$ plane (with  $\alpha>\alpha_0$), which destabilize Gieseker stable objects, PT stable objects or their duals.
\end{mthm}

The results in this paper, especially Main Theorem \ref{mthm1}, Main Theorem \ref{mthm2}, and Theorem \ref{t:PTisBridgeland}, lead to the following picture of the Gieseker and PT chambers in the $(\beta,\alpha)$-plane.

\begin{figure}[h] \centering
\begin{tikzpicture}[line cap=round,line join=round,>=triangle 45,x=1cm,y=1cm]
\draw[->] (-8.5,0) -- (0.5,0);
\draw[->] (0,-0.2) -- (0,7.2);
\begin{scope}
\clip(-8.5,-0.2) rectangle (1,7.2);
\draw [line width=2pt,color=ffqqqq] (-2.5,0) circle (1.8027756377319946cm);
\draw [samples=50,domain=-0.99:0.99,rotate around={0:(-0.5,0)},xshift=-0.5cm,yshift=0cm,line width=2pt,color=qqqqff] plot ({0.8660254037844386*(-1-(\x)^2)/(1-(\x)^2)},{0.8660254037844386*(-2)*(\x)/(1-(\x)^2)});
\draw (0.2,0.1) node[anchor=south west] {$\beta$};
\draw (0.0864859446932387,6.553798607869315) node[anchor=north west] {$\alpha$};
\draw [color=qqqqff](-5.8,4.5) node[anchor=north west] {$\Theta_v$};
\draw [line width=1.5pt] (-2.5,1.8) .. controls (-3.5,2.3) and (-4.5,2.3) .. (-5.5,1.6) .. controls (-7.1,1.1) and (-6.8,0.3) .. (-6.9,0); 
\draw[line width=1.5pt] (-2.5,1.8) .. controls (-3.0,2.5) .. (-3.2,3.6) .. controls (-5.4,5.1) and (-5.4, 5.7) .. (-2.7,5.8)  .. controls (-1.8,6.3) and (-1.7,6.6) .. (-1.5,7.2);
\begin{scriptsize}
\draw (-8.390040124688505,3.6230499630301054) node[anchor=north west] {\parbox{3 cm}{Gieseker chamber}};
\draw (-6,6.5) node[anchor=north west] {PT chamber};
\draw (-2.5,2.3) node[anchor=north west] {$(\beta_v,\alpha_v)$};
\draw [fill=uuuuuu] (-2.5,1.8027756377319948) circle (2pt);
\end{scriptsize}
\end{scope}
\end{tikzpicture}
\caption{The picture illustrates the Gieseker and PT chambers in the $(\beta,\alpha)$-plane for a given Chern character $v$. The red curve is the largest tilt wall for $v$, with the point of intersection with the blue $\Theta_v$ curve marked with $(\beta_v,\alpha_v)$, following the notation of Main Theorem \ref{mthm1}. The black curve represents the first actual wall that destabilizes Gieseker semistable sheaves and PT semistable objects, as described in Main Theorem \ref{mthm2}; this is a piecewise smooth curve, each continuous segment is an arc of a cubic or quartic curve.}
    \label{fig:intro}
\end{figure}
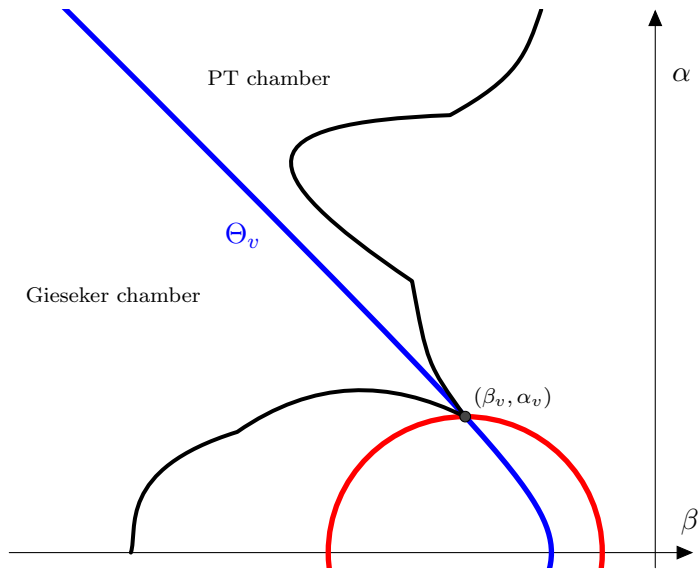

The paper is organized as follows. We start by setting up notation and definitions in Section \ref{sec:BG}. In particular, we briefly recall the Bayer--Macrì--Toda construction of Bridgeland stability conditions on threefolds \cite{BMT} in Section \ref{sec:Bsc}, revise some relevant results from our previous papers, notably \cite{JM19,JMM}, about asymptotic stability in Section \ref{sec:vertical}, and present Bayer's polynomial stability conditions \cite{BayerPBSC} in Section \ref{sec:polynomial}.

In Section \ref{sec:triples} we introduce \textit{stable triples} $(B,P,\eta)$ consisting of a torsion-free sheaf $B$, a 1-dimensional sheaf $P$ and $\eta\in\Ext^2(P,B)$ satisfying certain conditions (see Definition \ref{defn:stable-triple} for a precise statement) as a higher rank generalization of PT stable pairs, see Lemma \ref{lem:rk1ST-PTpairsequiv}; several other properties of these objects are established in this section.

In Section \ref{sec:PT}, we study the relations among PT-stability (in the sense of Bayer \cite{BayerPBSC}), the large volume limit (in the sense of Bayer--Macrì--Toda \cite{BMT1}), vertical stability (as defined in  \cite{JMM}), and stable triples (as defined in Section \ref{sec:triples}).  To start with, we define \textit{$\sigma_{\infty, B}$-stability}, which generalizes the large volume limit.  Then we show that vertical stability coincides with $\sigma_{\infty,B}$-stability  (Theorem \ref{thm:AG53-39-1}), and that $\sigma_{\infty,B}$-stability with \textit{sufficiently negative $B$} coincides with PT-stability (Theorem \ref{thm:PTislimtstab-eff}).  Then, in Propositions \ref{pro:AG53-92-1} and \ref{pro:AG53-92-2}, we show the implications
\begin{gather}\label{eq:equiv-PT-stabtri-intro}
\text{ stable triple } \Rightarrow \text{ PT-stable object}  \Rightarrow \text{ PT-semistable object } 
\Rightarrow \text{ semistable triple}
\end{gather}
thus completing our web of relations among PT-stability, the large volume limit, vertical stability, and stable triples. From these implications, we see that (semi)stable triples give us approximations of PT-(semi)stable objects via characterizations of their cohomology sheaves.  In particular, since stable triples coincide with semistable triples under the coprime assumption on rank and degree, all four notions in \eqref{eq:equiv-PT-stabtri-intro} coincide under the coprime assumption.

The brief Section \ref{sec:stab-theta} is dedicated to characterizing the Brigdeland moduli space $\mathcal{M}_{\obeta,\oalpha}(v)$ when $(\obeta,\oalpha)\in\Theta_v$, see Proposition \ref{prop:moduli in theta}. The proof of Main Theorem \ref{mthm1} is completed in Section \ref{sec:honest-wall} by choosing a convenient heart from the slicing corresponding to the stability conditions $(\mathcal{A}_{\beta,\alpha},Z_{\beta,\alpha,s})$, which allows us to work with the same Chern character on both sides of the wall $\Theta_v$. We also provide a detailed study of three concrete cases to exemplify how the wall $\Theta_v$ could be collapsing, fake, or an honest wall.

Finally, Main Theorem \ref{mthm2} is proved in Section \ref{sec:y} extending the arguments in \cite{JM19}.

\subsection*{Acknowledgments}
MJ is supported by the CNPQ grant number 302889/2018-3 and the FAPESP Thematic Project 2018/21391-1. JL was partially supported by NSF grant DMS-2100906.  We thank the ICMS for the Research-in-Groups programme \textit{Bridgeland stability on three-dimensional varieties} held from 19 February to 1 March 2024.


\section{Background material}\label{sec:BG}

\subsection{Notation}\label{sec:notation}

In what follows, $X$ will denote a smooth projective threefold of Picard rank 1 with ample generator $H$ for which stability conditions are known to exist, e.g., $X$ is Fano \cite{M-p3,S-q3,Li}, or abelian \cite{MP,BMS}, or a quintic threefold \cite{Li2}, or a general weighted hypersurface in either of the weighted projective spaces $\mathbb{P}(1,1,1,1,2)$ or $\mathbb{P}(1,1,1,1,4)$ \cite{KosDoubleTriple}, or a Calabi-–Yau threefold obtained as complete intersection of quadratic and quartic hypersurfaces in $\mathbb{P}^5$ \cite{LiuBGICY3}. Other than specified, we will use the following notation:
\begin{itemize}
\item The Chern character vector of an object $A\in \dbx$ is the vector 
\[
\ch(A)=(\ch_0(A)H^3,\ch_1(A)H^2,\ch_2(A)H,\ch_3(A))\in \mathbb{Z}\times \mathbb{Z}\times \frac{1}{2}\mathbb{Z}\times \frac{1}{6}\mathbb{Z}.
\]
By abuse of notation we will simply write $\ch_i(A)$ for $\ch_i(A)H^{3-i}$.
\item For $\beta\in \mathbb{R}$ and $A\in \dbx$, the twisted Chern characters of $A$ are
\begin{align*}
\ch^{\beta}_0(A)&=\ch_0(A),\\
\ch^{\beta}_1(A)&=\ch_1(A)-\beta\ch_0(A),\\
\ch^{\beta}_2(A)&=\ch_2(A)-\beta\ch_1(A)+\frac{\beta^2}{2}\ch_0(A),\\
\ch^{\beta}_3(A)&=\ch_3(A)-\beta\ch_2(A)+\frac{\beta^2}{2}\ch_1(A)-\frac{\beta^3}{6}\ch_0(A),
\end{align*}
i.e., $\ch^{\beta}(A):= \exp(-\beta H)\cdot\ch(A)\in \mathbb{R}^4$.
\item $\delta_{ij}^{\beta}(F,A) := \ch_i^{\beta}(F)\ch_j^{\beta}(A)-\ch_j^{\beta}(F)\ch_i^{\beta}(A)$. We suppress $\beta$ from the notation when $\beta=0$.
\item $\rho_{\beta,\alpha}(A) = \ch_2^{\beta}(A) -\alpha^2\ch_0^{\beta}(A)/2$.

\item Let $\mathcal{A}$ be the heart of a bounded t-structure on $\dbx$. For an object $A\in \dbx$, we denote by $\mathcal{H}^{i}_{\mathcal{A}}(A)$
the cohomology objects of $A$ with respect to the t-structure generated by $\mathcal{A}$. We suppress $\mathcal{A}$ from the notation in the case of the standard t-structure. 

\item For a line bundle $L$, we will occasionally write $n\cdot L:=L^{\oplus n}$.
\item We denote by $\coh^{\leq k}(X)$ the sub-category of $\coh(X)$ consisting of sheaves supported in dimension $\leq k$, and denote by  $\coh^{\geq k}(X)$  the sub-category of sheaves supported in dimension $\geq k$ that do not have any subsheaves in $\coh^{\leq k-1}(X)$. 
\item Following the notation in \cite[p.26]{HL}, we will write $\coh_{k,l}(X)$ to denote the quotient category $\coh^{\leq k}(X)/\coh^{\leq l-1}(X)$.
\item For an object  $A\in \dbx$, we denote $A^\vee:=\mathcal{H}om_{\dbx}(A,\ox)$; when $E$ is a sheaf of pure codimension $c$, we denote $E^*:=\calh^c(E^\vee)[c]=\inext^c(E,\ox)$. 
\end{itemize}


\subsection{Stability for sheaves} \label{sec:sheaves}

Given a coherent sheaf $E$ on $X$, let $P_E(t)$ denote its Hilbert polynomial with respect to a fixed polarization on $X$; let $p_E(t):=P_E(t)/\rk(E)$ be the reduced Hilbert polynomial. Recall that $E$ is said to be Gieseker (semi)stable if and only if every proper, non-trivial subsheaf $F\into E$ satisfies the inequality $p_F(t)<(\le)p_E(t)$, where $\le$ indicates the lexicographic order on $\mathbb{Q}[t]$. 

Now let $\bar{P}_E(t)$ be the polynomial obtained by eliminating the term on $t^0$ from $P_E(t)$; we will call this the \textit{truncated Hilbert polynomial}. We will say that a torsion-free sheaf $E$ is \textit{2-Gieseker (semi)stable} if and only if every proper, non-trivial subsheaf $F\into E$ satisfies the inequality $\bar p_F(t)<(\le)\bar p_E(t)$ involving the reduced, truncated Hilbert polynomials for $F$ and $E$. This is equivalent to the notion of stability in the category $\coh_{3,1}(X)$ in the sense of \cite[Definition 1.6.3]{HL}; in \cite[Section 2.5]{JM19}, this same notion was called $\mu_{\le2}$-stability.

Finally, a torsion-free sheaf $E$ is \textit{$\mu$-(semi)stable} if and only if every proper, non-trivial subsheaf $F\into E$ satisfies the inequality $\mu(F)<(\le)\mu(E)$, where $\mu$ indicates the usual slope for a torsion-free sheaf.

The following chain of implications is easy to check:
$$ \mu\textrm{-stable} ~~\Rightarrow~~ \textrm{2-Gieseker stable} ~~\Rightarrow~~ \textrm{Gieseker stable} ~~\Rightarrow $$
$$ \Rightarrow~~ \textrm{Gieseker semistable} ~~\Rightarrow~~ \textrm{2-Gieseker semistable} ~~\Rightarrow~~ \mu\textrm{-semistable}. $$
Note that if $\gcd(\ch_0(E),\ch_1(E))=1$, then every $\mu$-semistable sheaf is actually $\mu$-stable, and all these notions of stability coincide.

If $E$ is torsion-free, then we have the following canonical short exact sequence:
\begin{equation} \label{std sqc}
 0 \to E \to E^{**} \to Q_E \to 0,~~ Q_E:=E^{**}/E,
\end{equation}
where $E^*:=\inhom(E,\ox)$. Note that $\dim Q_E\le1$; letting $Z_E$ be the maximal 0-dimensional subsheaf of $Q_E$, define $T_E:=Q_E/Z_E$, which is a pure 1-dimensional sheaf. Let $E'$ be the kernel of the composed epimorphism $ E^{**} \onto Q_E \onto T_E$; it fits into the following short exact sequence in $\coh(X)$:
\begin{equation} \label{e' sqc}
 0 \to E \to E' \to Z_E \to 0 .
\end{equation}
Clearly $(E')^{**}=E^{**}$, and $(E')^{**}/E'=T_E$. Note that 
\begin{itemize}
\item[(i)] $E$ is $\mu$-(semi)stable if and only if $E'$ is $\mu$-(semi)stable;
\item[(ii)] $E$ is 2-Gieseker (semi)stable if and only if $E'$ is 2-Gieseker (semi)stable.
\end{itemize}

\begin{remark}
In this paper, the moduli space of Gieseker semistable sheaves with fixed Chern character $v$ will be denoted by $\calg(v)$.

A \emph{reflexive} sheaf will always be torsion-free by default.
\end{remark}


\subsection{Technical lemmas}

For an ample class $\omega$ on a smooth projective threefold $X$, We write $\wh{\mu}$ to denote the slope function $\wh{\mu}:= \ch_3/\omega\ch_2$ on $\coh^{\leq 1}(X)$.

\begin{proposition}\label{prop:AG53-46-1}
Let $X$ be a smooth projective variety with an ample class $\omega$, and $\{E_t\}_{t \in I} $ a bounded set of $\mu_\omega$-semistable  sheaves on $X$.  Then there exists a constant $b \in \mathbb{R}$ such that $\wh{\mu}_{\mathrm{max}}( (E_t)^{\ast\ast}/(E_t)) < b$ for all $t \in I$.
\end{proposition}

\begin{proof}
Suppose $\{E_t\}_{t \in I}$ is a bounded set of  $\mu_\omega$-semistable sheaves on $X$. This means that there is a finite-type scheme $W$ over $\mathbb{C}$ and an object $F \in D^b(X \times W)$ such that each $E_t$ is isomorphic to $F|^L_q$ for some $q \in W$.  To prove our claim, we can assume that $W$ itself is projective over $\mathbb{C}$.  Since derived dual and derived pull-back commute under a morphism of projective schemes, it follows that each $(E_t)^\vee$ occurs as $(F^\vee)|^L_q$ for some $q \in W$.  That is, the set $\{(E_t)^\vee\}_{t \in I}$ is also bounded.  Now, using a flattening stratification $W_1, \cdots, W_m \subseteq W$ such that each cohomology sheaf of $F^\vee$ is flat over each $W_i$, it follows that the set $\{ H^0( (E_t)^\vee ) \}_{t \in I} = \{ (E_t)^\ast\}_{t \in I}$ is also bounded.  Repeating this process, we obtain the boundedness of the set $\{ (E_t)^{\ast\ast}\}_{t \in I}$.

For each $t \in I$, we have an exact triangle in $\dbx$
\[
(E_t)^{\ast\ast} \to (E_t)^{\ast\ast}/E_t \to (E_t)[1].
\]
Since both the sets $\{ E_t\}_{t \in I}$ and $\{ (E_t)^{\ast\ast}\}_{t \in I}$ are bounded, it follows that the set $\{ (E_t)^{\ast\ast}/(E_t) \}_{t \in I}$ is also bounded \cite[Lemma 3.16]{Toda08}.  That is, there is some scheme $V$ of finite type over $\mathbb{C}$ and some object $E \in D^b(X \times V)$ such that, for every $t \in I$, the sheaf $(E_t)^{\ast\ast}/E_t$ is isomorphic in $\dbx$ to the derived pullback $E|_s^L$ for some closed point $s \in V$.  

Now consider the locus of closed points
\[
V_1 = \{ s \in V : E |_s^L \cong (E_t)^{\ast\ast}/E_t \text{ for some }t \in I\}.
\]
Since being a coherent sheaf is an open property in a family of complexes \cite[Lemma 2.1.4]{Lieblich06}, for each $s \in V_1$ there exists an open subscheme $U_s \subseteq V$ containing $s$ such that, for every $s' \in U_s$, the derived pullback $E|^L_{s'}$ is also a sheaf.  Let $U$ be the union of all such $U_s$.  Then $E|_U$ is a $U$-flat family of sheaves on $X \times U$ by \cite[Lemma 3.31]{FMTAG} and, in particular, $E|_U$ itself is a sheaf sitting at degree 0 in $D^b(X \times U)$.

Using  \cite[Theorem 2.3.2]{HL}, we now have a stratification  of $U$ into locally closed subschemes $U_1, \ldots, U_n$ such that, for every $U_i$, the slope $\wh{\mu}_{\mathrm{max}}(E|_s)$ is constant as $s$ ranges over $U_i$.  The proposition now follows.
\end{proof}


The following lemma is well-known to experts, but we include the proof here for ease of reference.

\begin{lemma}\label{lem:CohzeroSerresubs}
Let $X$ be a smooth projective threefold, and $\omega, B$ divisors on $X$ such that $\omega$ is ample.  Then 
\begin{itemize}
    \item[(i)] $\coh^{\leq 0}(X)$ is a Serre subcategory of $\coh^B(X)$.
    \item[(ii)] $\coh^{\leq 0}(X)$ is a Serre subcategory of $\Ac^{B,\omega}$.
\end{itemize}
\end{lemma}

\begin{proof}
(i) Take any $\coh^B(X)$-short exact sequence
\begin{equation}\label{eq:AG53-78-1}
0 \to A \to E \to Q \to 0
\end{equation}
where $E \in \coh^{\leq 0}(X)$. Since $\coh^{\leq 0}(X) \subset \coh^B(X)$, the long exact sequence of sheaves looks like
\[
0 \to \mathcal{H}^{-1}(Q) \to A \overset{f}{\to} E \to \mathcal{H}^0(Q) \to 0.
\]
Since $\coh^{\leq 0}(X)$ is a Serre subcategory of $\coh (X)$, $\image (f)$ lies in $\coh^{\leq 0}(X)$, and so $\mathcal{H}^{-1}(Q),\ A$ have the same $\ch_0$.  If $\ch_0(\mathcal{H}^{-1}(Q))=\ch_0(A)$ is non-zero, then $\mu_{B,\omega}(\mathcal{H}^{-1}(Q))\leq 0$ while $\mu_{B, \omega}(A)>0$, which is impossible, and so $\ch_0(\mathcal{H}^{-1}(Q))=0$, forcing $\mathcal{H}^{-1}(Q)=0$. Hence \eqref{eq:AG53-78-1} is a short exact sequence in $\coh (X)$ and both $A, Q$ lie in $\coh^{\leq 0}(X)$. 

(ii) The argument is completely analogous to that in (i).  Take any $\Ac^{B,\omega}$-short exact sequence 
\begin{equation}\label{eq:AG53-78-2}
0 \to A \to E \to Q \to 0
\end{equation}
where $E \in \coh^{\leq 0}(X)$. Noting $E \in \coh^{\leq 0}(X)\subset \coh^B(X)$, the  $\coh^B(X)$-long exact sequence of cohomology looks like
\[
0 \to \calh^{-1}(Q) \to A \overset{f}{\to} E \to \calh^0(Q) \to 0
\]
where $\calh^i$ denotes the degree-$i$ cohomology with respect to the heart $\coh^B(X)$.  Since $\coh^{\leq 0}(X)$ is a Serre subcategory of $\coh^B(X)$ from (i), $\image (f)$ lies in $\coh^{\leq 0}(X)$, and so $\ch_1^B (\calh^{-1}(Q))=\ch_1^B(A)$.  If $\omega^2\ch_1^B (\calh^{-1}(Q))=\omega^2\ch_1^B(A)$ is non-zero, then $\nu_{B,\omega}(\calh^{-1}(Q))\leq 0$ while $\nu_{B,\omega}(A)>0$, which is impossible since $\calh^{-1}(Q), A$ have the same $\ch_0, \ch_1, \ch_2$.  Hence $\omega^2\ch_1^B(\calh^{-1}(Q))=0$, forcing $\calh^{-1}(Q)$ itself to be zero.  This means \eqref{eq:AG53-78-2} is a short exact sequence in $\coh^B(X)$ and so, by (i), both $A$ and $Q$ lie in $\coh^{\leq 0}(X)$.
\end{proof}

Recall that the homological dimension of a coherent sheaf $E$ is the smallest length of a locally free resolution of $E$; it will be denoted by $\hd(E)$.
The following lemma will also be useful below.

\begin{lemma}\label{lem:hd1equivconds}
Let $X$ be a smooth projective threefold and $E$ a torsion-free sheaf on $X$.  The following are equivalent:
\begin{itemize}
    \item[(i)] $E$ has homological dimension at most 1.
    \item[(ii)] $\Hom (\coh^{\leq 0}(X),E[1])=0$.
    \item[(iii)] $E^{\ast\ast}/E$ is pure 1-dimensional, if non-zero.
\end{itemize}
\end{lemma}

\begin{proof}
The equivalence of (i) and (ii) follows from Lemma \ref{lem:Lo17lem3p3}.  

To see the equivalence between (ii) and (iii), take the double dual of $E$ to form the short exact sequence
\[
0 \to E \to E^{\ast\ast} \to E^{\ast\ast}/E \to 0
\]
and then rotate the corresponding exact triangle in $\dbx$ to obtain the short exact sequence in $\Ac^p$
\begin{equation}\label{eq:AG54-4-1}
  0 \to E^{\ast\ast}/E \to E[1] \to E^{\ast\ast}[1] \to 0.
\end{equation}
Suppose (ii) is true.  Applying  $\Hom (\coh^{\leq 0}(X),-)$ to this short exact sequence  yields $\Hom (\coh^{\leq 0}(X),E^{\ast\ast}/E)=0$, meaning $E^{\ast\ast}/E$ is pure 1-dimensional if non-zero. 

Now suppose (iii) is true.  Applying $\Hom (\coh^{\leq 0}(X),-)$ to \eqref{eq:AG54-4-1} again and noting that $\Hom (\coh^{\leq 0}(X), E^{\ast\ast}[1])=0$ by Lemma \ref{lem:Lo17lem3p3}(ii), we see that (ii) follows.
\end{proof}


\subsection{Bridgeland stability conditions on threefolds} \label{sec:Bsc}

We briefly recall the construction of Bridgeland stability conditions on $X$ due to Bayer, Macrì, and Toda \cite{BMT}. 

For $\beta\in\mathbb{R}$ define the (twisted) Mumford slope by
$$ \mu_{\beta}(E)=\begin{cases}\dfrac{\ch^{\beta}_1(E)}{\ch^{\beta}_0(E)}& \text{if}\ \ch^{\beta}_0(E)\neq 0,\\ +\infty &\text{if}\ \ch^{\beta}_0(E)=0.\end{cases} $$ 
The subcategories 
\begin{align*}
\tors{\beta}&:=\{E\in \coh(X)\ |\ \mu_{\beta}(Q)>0\ \text{for every quotient}\ E\onto Q\},\ \text{and}\\
\free{\beta}&:=\{E\in \coh(X)\ |\ \mu_{\beta}(F)\leq 0\ \text{for every subsheaf} \ F\into E\},
\end{align*}
form a torsion pair on $\coh(X)$. The corresponding tilted category is the extension closure
$$
\coh^{\beta}(X):=\langle \free{\beta}[1],\ \tors{\beta}\rangle.
$$
For every $\alpha>0$ we can define the following \textit{slope} function on $\coh^{\beta}(X)$
$$
\nu_{\beta, \alpha}(E)=\begin{cases}\dfrac{\ch^{\beta}_2(E)-\frac{\alpha^2}{2}\ch^{\beta}_0(E)}{\alpha\ch^{\beta}_1(E)}& \text{if}\ \ch^{\beta}_1(E)\neq 0,\\ +\infty &\text{if}\ \ch^{\beta}_1(E)=0.\end{cases}
$$ 
We refer to $\nu_{\beta,\alpha}$ as the tilt, and to objects in $\coh^{\beta}(X)$ that are semistable with respect to $\nu_{\beta,\alpha}$ as tilt--semistable objects.

\begin{remark}
    As in the case of surfaces, it is known that in the $(\beta,\alpha)$-plane the tilt walls for a fixed Chern character $v$ (with $v_0>0$ and satisfying the Bogomolov inequality $\Delta(v)=v_1^2-2v_0v_2\geq 0$) are bounded (see \cite{S}) and that for a fixed value of $\beta_0\ll 0$ the only objects that are $\nu_{\beta,\alpha}$-semistable for $\alpha\gg 0$ are precisely the $2$-Gieseker semistable sheaves of Chern character $v$ (see for instance \cite[Proposition 14.2]{B08} and \cite[Theorem 5.2]{JM19}).
\end{remark}

Analogously to the case of Mumford slope, the subcategories
\begin{align*}
\mathcal{T}_{\beta,\alpha}&:=\{E\in \coh^{\beta}(X)\ |\ \nu_{\beta,\alpha}(Q)>0\ \text{for every quotient}\ E\onto Q\ \text{in}\ \coh^{\beta}(X)\},\ \text{and}\\
\mathcal{F}_{\beta,\alpha}&:=\{E\in \coh^{\beta}(X)\ |\ \nu_{\beta,\alpha}(F)\leq 0\ \text{for every subobject} \ F\into E\ \text{in}\ \coh^{\beta}(X)\},
\end{align*}
form a torsion pair on $\coh^{\beta}(X)$. The corresponding tilted category is the extension closure
$$
\mathcal{A}^{\beta,\alpha}:=\langle {\mathcal{F}}_{\beta,\alpha}[1],\ \tors{\beta,\alpha}\rangle.
$$
As proven in \cite{BMT,M-p3,BMS,Li}, we know that $\mathcal{A}^{\beta,\alpha}$ supports a 1-parameter family of stability conditions whose central charges are given by
$$
Z_{\beta,\alpha,s}(A)=-\left(\ch^{\beta}_3(A)-\left(s+\frac{1}{6}\right)\alpha^2\ch^{\beta}_1(A)\right)+i\left(\ch^{\beta}_2(A)-\frac{\alpha^2}{2}\ch^{\beta}_0(A)\right),\ s>0.
$$ 
The associated Bridgeland slope is
$$
\lambda_{\beta,\alpha,s}(A):=\frac{\ch^{\beta}_3(A)-\left(s+\frac{1}{6}\right)\alpha^2\ch^{\beta}_1(A)}{\ch^{\beta}_2(A)-\frac{\alpha^2}{2}\ch^{\beta}_0(A)}.
$$
The stability conditions $\sigma_{\beta,\alpha,s}=(Z_{\beta,\alpha,s},\mathcal{A}^{\beta,\alpha})$ are geometric (i.e., skyscraper sheaves $\mathcal{O}_x$ are stable of the same phase), and they satisfy the support property with respect to the quadratic form
$$
Q_{\beta,\alpha, K}(\ch):=(\ch^{\beta})^{t}B_{\alpha,K}\ch^{\beta},\ \text{where}\ B_{\alpha,K}=\begin{pmatrix}0 & 0 & -K\alpha^2& 0\\ 0 & K\alpha^2 &0 &-3\\ -K\alpha^2 & 0 & 4 &0 \\ 0 & -3 & 0 & 0\end{pmatrix},
$$
for every $K\in[1, 6s+1)$ (see \cite{BMT,BMS}). For an object $A\in\mathcal{A}^{\beta,\alpha}$ we write $Q_{\beta,\alpha, K}(A)$ for $Q_{\beta,\alpha, K}(\ch(A))$, and denote the corresponding pairing 
$$
(\ch^{\beta}(F))^{t}B_{\alpha,K}\ch^{\beta}(A)
$$ 
by $Q_{\beta,\alpha, K}(F,A)$.

For a fixed Chern character $v$ with $v_0>0$ and satisfying the Bogomolov inequality $\Delta(v)\geq 0$, we define the curves
\begin{align*}
    \Theta_v&=\{(\beta,\alpha)\colon 2\ch_2^{\beta}(v)-\alpha^2\ch_0(v)=0\}\\
    \Gamma_{v,s}&=\{(\beta,\alpha)\colon 6\ch_3^{\beta}(v)-(6s+1)\ch_1^{\beta}(v)=0\}
\end{align*}

\begin{figure}[ht]\label{fig:theta} \centering
\includegraphics[scale=2]{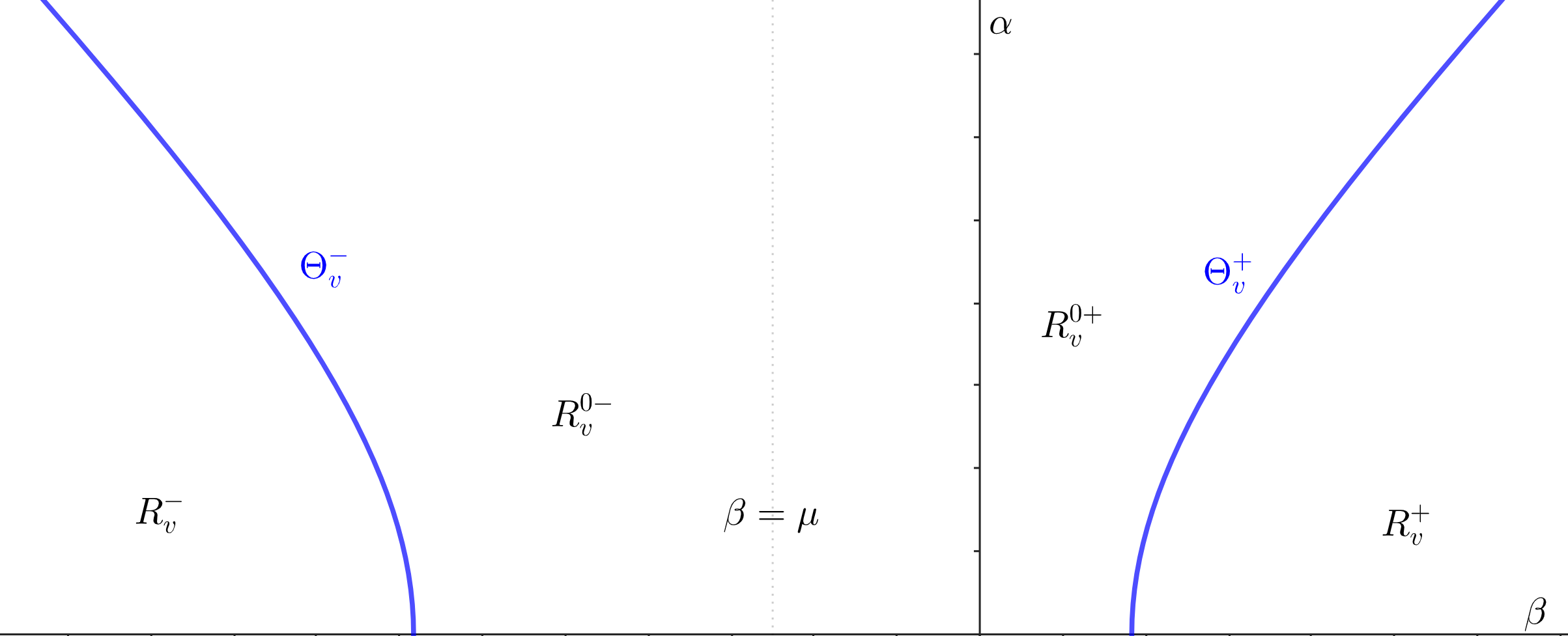}
\caption{The four regions of the plane as defined by the hyperbola $\Theta_v$ and the vertical line $\{\beta=\mu(v)\}$ when $v$ is a numerical Chern character satisfying $\Delta(v)\geq 0$.}
\end{figure}

\begin{proposition}
Let $v=(v_0,v_1,v_2)$ with $v_0>0$ and $\Delta(v)>0$. The set
$$
E(v)=\{e\colon e=\ch_3(E)\ \text{for some 2-Gieseker semistable sheaf } E \text{ with}\ \ch_{\leq 2}(E)=v\}
$$
is bounded above.
\end{proposition}
\begin{proof}
    This is because, under the hypotheses, the actual tilt walls for $v$ are semicircles of bounded radii. Let $(\beta_0,\alpha_0)\in \Theta_v^-$ such that there is no tilt wall intersecting $\Theta_v^-$ at a point $(\beta,\alpha)$ with $\beta\leq \beta_0$. Since $v_0>0$ then any 2-Gieseker semistable sheaf $E$ with $\ch_{\leq 2}(E)=v$ is $\nu_{\beta_0,\alpha_0}$-semistable and moreover $\nu_{\beta_0,\alpha_0}(E)=0$. Thus, $E[1]\in\mathcal{A}^{\beta_0,\alpha_0}$ and $\mathfrak{Im}\big(Z_{\beta_0,\alpha_0,s}(E[1])\big)=0$ for all $s>0$. Therefore,
    $$
    \mathfrak{Re}\big(Z_{\beta_0,\alpha_0,s}(E[1])\big)<0\ \ \text{for all}\ \ s>0.
    $$
    This says that
    $$
    \ch_3^{\beta_0}(E)\leq \frac{\alpha_0^2}{6}\ch_1^{\beta_0}(E),
    $$
    which is equivalent to having
    $$
    \ch_3(E)\leq \frac{\alpha_0^2}{6}(v_1-\beta_0 v_0)+\beta_0 v_2-\frac{\beta_0^2}{2}v_1+\frac{\beta_0^3}{6}v_0.
    $$
\end{proof}

\begin{remark} 
Observe that $E(v)$ is not bounded below: given a 2-Gieseker semistable sheaf $E$, choose a point $p$ away from the singular locus of $E$ and pick an epimorphism $\varphi:E\to\calo_p$ (which exists because $E$ is locally free at $p$); set $E_\varphi:=\ker(\varphi)$ and note that $E_\varphi$ is also 2-Gieseker semistable and $\ch_3(E_\varphi)=\ch_3(E)-1$. Therefore, $\ch_3$ of a 2-Gieseker semistable sheaf can be arbitrarily negative.  
\end{remark}

Following Schmidt \cite{S21}, for a fixed $v=(v_0,v_1,v_2)$, we define
\begin{equation} \label{eq:epsilon(v)}
\varepsilon_X(v):=\max \big\{ E(v) \big\}. 
\end{equation}

\begin{lemma} \label{lem:2Gsst hd=1}
If $E$ is a Gieseker semistable sheaf with $\ch(E)=(v_0,v_1,v_2,\varepsilon_X(v))$, then $\hd(E)\le 1$.
\end{lemma}
\begin{proof}
For a contradiction, assume that $E^{**}/E$ admits a non-trivial 0-dimensional subsheaf $Z$ so that $\hd(E)>1$; let $E'$ be the sheaf satisfying the short exact sequence in display \eqref{e' sqc}. Note that $\ch_{\le2}(E')=\ch_{\le2}(E)=v$ and
$$ \ch_3(E') = \ch_3(E) + \ch_3(Z) > \varepsilon_X(v); $$
however, as it was observed above, $E'$ is 2-Gieseker semistable because $E$ is Gieseker semistable, so $\ch_3(E')\le \varepsilon_X(v)$, providing the desired contradiction.
\end{proof}


\subsection{Vertical stability} \label{sec:vertical}

Let us now recall the concept of asymptotic stability introduced in \cite[Definition 7.1]{JM19}. Let $\gamma\colon [0,\infty)\rightarrow \mathbb{H}$ be an unbounded path. An object $A\in\dbx$ is \emph{asymptotically $\labs$-(semi)stable along $\gamma$} if the following two conditions hold for a given $s>0$: 
\begin{itemize}
\item[(i)] there is $t_0>0$ such that $A\in\mathcal{A}^{\gamma(t)}$ for every $t>t_0$;
\item[(ii)]  there is $t_1>t_0$ such that for all $t>t_1$, all sub-objects $F\into A$ in $\mathcal{A}^{\gamma(t)}$ satisfy $\lambda_{\gamma(t),s}(F) < ~(\le)~ \lambda_{\gamma(t),s}(A)$. 
\end{itemize}
Since $\mathcal{A}^{\beta,\alpha}$ is a full subcategory of $\dbx$ we can always consider any fixed non-trivial morphism $F\to A$ in $\dbx$ which becomes an injection in $\mathcal{A}^{\beta,\alpha}$ as $(\beta,\alpha)$ vary. 
In \cite{JM19} the authors studied the asymptotics along the so-called \textit{$\Theta^\pm$-curves}, that is, unbounded curves $\gamma(t)=(\alpha(t),\beta(t))$ such that $\lim_{t\to\infty}\beta(t)=\pm\infty$ and
\[ \lim_{t\to\infty} \left|\frac{\dot\alpha(t)}{\dot\beta(t)}\right| > 1. \]
The key results from \cite{JM19} (Theorem 7.19 and Theorem 7.23, to be precise) can be summarized as follows.

\begin{theorem} \label{thm:asymp-h+}
Let $v$ be a numerical Chern character satisfying $v_0\ne0$ and $Q^{\rm tilt}(v)\ge0$; fix $s\geq1/3$. 
\begin{enumerate}
\item If $\gamma$ is a $\Theta^-$-curve, then an object $E\in \dbx$ with $\ch(E)=v$ is asymptotically $\lambda_{\alpha,\beta,s}$-(semi)stable along $\gamma$ if and only if $E$ is a Gieseker (semi)stable sheaf.
\item If $\gamma$ is a $\Theta^+$-curve, then an object $A\in\dbx$ with $\ch(A)=v$ is asymptotically $\labs$-(semi)stable along $\gamma$ if and only if $A^\vee$ is a Gieseker (semi)stable sheaf.
\end{enumerate}
\end{theorem}

In another direction, the special case of asymptotic stability along vertical lines $\{\beta=\obeta\}$ was studied in detail in \cite{JMM}. We will refer to an asymptotic $\lambda_{\beta,\alpha,s}$-(semi)stable object along the vertical curve $\{\beta=\obeta\}$ simply as a \emph{vertical (semi)stable object} at $\obeta$.

In this context of vertical lines, it turns out that once a map $F\to E$ injects, it must inject for all $\alpha\gg 0$ or it never injects for $\alpha\gg 0$.
\begin{lemma}
Fix $\beta$ and let $\alpha_0>0$. Suppose $\phi:F\to E$ is a morphism in $\mathcal{A}^{\beta,\alpha}$ for all $\alpha\geq\alpha_0$ and suppose $\phi$ injects in $\mathcal{A}^{\beta,\alpha_i}$ for an unbounded strictly increasing sequence of $\alpha_i\geq\alpha_0$. Then $\phi$ injects in $\mathcal{A}^{\beta,\alpha}$ for all $\alpha>\alpha_0$.
\end{lemma}
\begin{proof}
Let $C$ be the cone of $\phi$ in $D^b(X)$ and for any $A\in D^b(X)$ let $A^i=\mathcal{H}^i_{\cohb}(A)$. Let $K_\alpha$ and $Q_\alpha$ denote the kernel and cokernel of $\phi$ respectively.  We have $C^0=\mathcal{H}^0_{\cohb}(Q_\alpha)$. But the LHS does not depend on $\alpha$ and since $\phi$ injects in $\mathcal{A}^{\beta,\alpha_0}$ it must vanish. We also have a short exact sequence in $\cohb$ for all $\alpha\geq\alpha_0$:
\[0\to K^0_\alpha\to C^1\to Q^1_\alpha\to 0.\]
But $C^1=K^0_{\alpha_0}$ and so $\nu^-_{\beta,\alpha_0}(C^1)>0$.

Suppose $\ch_0(C^1)\geq0$. Note that 
\[\nu_{\beta,\alpha}(v)=\frac{\ch_2^\beta(v)-\alpha^2\ch_0(v)}{\ch_1^\beta(v)}\]
and $\ch_1^\beta(v)>0$ for any Chern character $v$ of an element of $\cohb$. So $\nu_{\beta,\alpha}(C^1)$ is a decreasing function of $\alpha$. But $\nu^-_{\beta,\alpha_i}(C^1)>0$ and so this is a contradiction as $\alpha_i\to\infty$.

It follows that $\ch_0(C^1)<0$. But then $\nu_{\beta,\alpha}(Q^1_{\alpha})>0$ which implies $Q^1_{\alpha}=0$ and so $\phi$ injects for all $\alpha\geq\alpha_0$. 
\end{proof}
It follows that (ii) in the definition of asymptotic stability above in the case of vertical paths is equivalent to the apparently weaker
\begin{itemize}
    \item[(ii)$^\prime$] there is $t_1>t_0$ such that for any map $\phi:F\to A$ in $\mathcal{A}^{\gamma(t)}$ for all $t>t_0$, if $\phi$ injects for all $t>t_1$ in $\mathcal{A}^{\gamma(t)}$ then $\lambda_{\gamma(t),s}(F) < ~(\le)~ \lambda_{\gamma(t),s}(A)$. 
\end{itemize}
This is the version we will use in Theorem \ref{thm:AG53-39-1} below.

The precise characterization of vertically stable objects, established in \cite{JMM}, is the content of Theorem \ref{vert-3.1} and Theorem \ref{vert-3.2} below.

\begin{theorem}\label{vert-3.1}\cite[Theorem 3.1]{JMM} 
Let $s\geq1/3$. If an object $A\in\dbx$ with $\ch_0(A)\ne0$ satisfies the following conditions
\begin{enumerate}
\item $\calh^p(A)=0$ for $p\ne-1,0$;
\item $\dim\calh^0(A)\le1$, and every quotient sheaf $\calh^0(A)\onto P$ (including $\calh^0(A)$ itself) satisfies
$$ 
\dfrac{\ch_3(P)}{\ch_2(P)} > C_{\obeta,s}(A),\ \text{where}\ C_{\obeta,s}(A):=\dfrac{6s+1}{3}\left( \mu(A)-\obeta \right) + \obeta
$$
whenever $\ch_2(P)\ne0$;
\item $\calh^{-1}(A)$ is $\mu$-stable;
\item $\calh^{-1}(A)$ is either reflexive or $\calh^{-1}(A)^{**}/\calh^{-1}(A)$ has pure dimension 1, and every subsheaf $R\into\calh^{-1}(A)^{**}/\calh^{-1}(A)$ (including $\calh^{-1}(A)^{**}/\calh^{-1}(A)$ itself) satisfies
$$ \dfrac{\ch_3(R)}{\ch_2(R)} < C_{\obeta,s}(A) ; $$ 
\item if $U$ is a sheaf of dimension at most 1 and $u:U\to \calh^0(A)$ is a non-zero morphism that lifts to a monomorphism $\tilde{u}:U\into A$ within $\mathcal{A}^{\obeta,\alpha}$ for every $\alpha\gg0$, then
$$ \dfrac{\ch_3(U)}{\ch_2(U)} < C_{\obeta,s}(A); $$
\end{enumerate}
then $A$ is vertical stable at $\obeta$.
\end{theorem}

Condition (3) in the previous result can be slightly weakened, under some additional hypotheses.

\begin{lemma} \label{lem:alt v3.1}
Fix $s>1/3$, and let $A\in\dbx$ be an object with $\ch_0(A)\ne0$ satisfying conditions (1), (2), (4) and (5) in Theorem \ref{vert-3.1}. If $\calh^{-1}(A)$ is 2-Gieseker stable, then there is $\beta_0'<0$ such that $A$ is vertical stable at $\beta<\beta_0'$.
\end{lemma}
\begin{proof}
Following the proof of \cite[Theorem 3.1]{JMM}, we observe that it is enough to consider sub-objects $F\into A$ within $\cala^{\beta,\alpha}$ for $\alpha\gg0$ such that $\ch_0(A)<\ch_0(F)<0$; taking $\beta<\mu(A)$, \cite[Lemma 3.14]{JMM} implies that $C:=\calh^{-1}(F)$ is a subsheaf of $B:=\calh^{-1}(A)$. Since $B$ is a $\mu$-semistable sheaf, we have that $\mu(F)\le\mu(C)\le\mu(B)=\mu(A)$. When $\mu(F)<\mu(A)$, then $F$ does not destabilize $A$, see \cite[Proof of Theorem 3.1 for $\obeta<\mu(A)$]{JMM}; if $\mu(F)=\mu(A)$, then one must analyze the sign of the limit
$$ \lim_{\alpha\to\infty} \alpha^2\big( \labs(F) - \labs(A) \big) = \dfrac{4}{\ch_0(F)\ch_0(A)} D_{\beta,s}(F,A), $$
where
\begin{eqnarray*} 
D_{\obeta,s}(F,A) & := & \left(s+\dfrac{1}{6}\right)(\mu(A)-\obeta)\delta_{20}^{\obeta}(F,A)-\dfrac{1}{2}\delta_{30}^{\obeta}(F,A) \\
& = & -\left(s-\dfrac{1}{3}\right)\delta_{20}(F,A)\beta + \left(s+\dfrac{1}{6}\right)\mu(A) - \dfrac{1}{2}\delta_{30}(F,A);
\end{eqnarray*}
note that $\ch_0(F)\ch_0(A)>0$ and
$$ \delta_{20}(F,A) = \delta_{20}(C,B)-\ch_2(\calh^0(F))\ch_0(B) \le \delta_{20}(C,B)<0$$
as $\mu(F)=\mu(C)$ forces $\ch_1(\calh^0(F))=0$, so $\dim(\calh^0(F))\le1$. It then follows that $D_{\obeta,s}(F,A)<0$ when 
$$ \beta < \dfrac{-1}{(s-1/3)\delta_{20}(F,A)}\left( \dfrac{1}{2}\delta_{30}(F,A) - \left(s+\dfrac{1}{6}\right)\mu(A) \right). $$ 
Therefore, we take $\beta_0'$ as the minimum between $\mu(A)$ and the right-hand side of the previous inequality to reach the desired conclusion.
\end{proof}

\begin{remark}
Let $B$ be a Gieseker stable sheaf that is not 2-Gieseker stable, so that there exists a 2-Gieseker stable subsheaf $C$ of $B$ with $\delta_{10}(C,B)=\delta_{20}(C,B)=0$ and $\delta_{30}(C,B)<0$; assume in addition that $\mu(B)\ge0$.

Lemma \ref{lem:alt v3.1} tells us that there is $\beta_0$ such that $C[1]$ is vertical stable at any $\beta<\beta_0$. The sheaf monomorphism $C\into B$ yields a nonzero morphism $C[1]\to B[1]$ in $\cala^{\beta,\alpha}$; note that
$$ D_{\obeta,s}(C[1],B[1]) = \left(s+\dfrac{1}{6}\right)\mu(B) -\dfrac{1}{2}\delta_{30}(C,B) > 0. $$
It follows that $B[1]$ can not be vertically stable at any $\beta<\beta_0$. This example shows that Lemma \ref{lem:alt v3.1} is in some sense sharp.
\end{remark}

\begin{theorem}\label{vert-3.2}\cite[Theorem 3.2]{JMM}
Let $v$ be a numerical Chern character satisfying $v_0\ne0$ and the Bogomolov inequality $v_1^2-2v_0v_2\geq 0$. If $A\in\dbx$ is a vertical semistable object at $\obeta$ with $\ch(A)=v$, then:
\begin{enumerate}
\item $\calh^p(A)=0$ for $p\ne-1,0$;
\item $\dim\calh^{0}(A)\le1$, and every sheaf quotient $\calh^{0}(A)\onto P$ (including $\calh^{0}(A)$ itself) satisfies
$$ \dfrac{\ch_3(P)}{\ch_2(P)} \ge C_{\obeta,s}(A) $$
whenever $\ch_2(P)\ne0$;
\item $\calh^{-1}(A)$ is $\mu$-semistable, and every sub-object $F\into A$ with $\mu(F)=\mu(A)$ satisfies $D_{\obeta,s}(F,A)\le0$.
\item $\calh^{-1}(A)$ is either reflexive or $\calh^{-1}(A)^{**}/\calh^{-1}(A)$ has pure dimension 1, and every subsheaf $R\into\calh^{-1}(A)^{**}/\calh^{-1}(A)$ (including $\calh^{-1}(A)^{**}/\calh^{-1}(A)$ itself) satisfies
$$ \dfrac{\ch_3(R)}{\ch_2(R)} \le C_{\obeta,s}(A) ; $$ 
\item if $U$ is a sheaf of dimension at most 1 and $u:U\to \mathcal{H}^0(A)$ is a non-zero morphism that lifts to a monomorphism $\tilde{u}:U\into A$ within $\mathcal{A}^{\obeta,\alpha}$ for every $\alpha\gg0$, then
$$ \dfrac{\ch_3(U)}{\ch_2(U)} \le C_{\obeta,s}(A). $$
\end{enumerate}
\end{theorem}

\begin{lemma}\label{lem:v1=0}
Let $v$ be a numerical Chern character satisfying $v_0\ne0$ and $v_1^2-2v_0v_2\geq 0$. There is $\beta_0<0$ depending on $A$ such that if $A\in\dbx$ is a vertical semistable object at $\beta<\beta_0$ with $\ch(A)=v$, then $\calh^{-1}(A)$ is 2-Gieseker semistable.
\end{lemma}
\begin{proof}
We know from Theorem \ref{vert-3.2} that $B:=\calh^{-1}(A)$ is $\mu$-semistable. Assume, for a contradiction, that $B$ is not 2-Gieseker semistable, and let $C\into B$ be the maximal 2-Gieseker destabilizing subsheaf,
so that $\delta_{10}(C,B)=0$ and $\delta_{20}(C,B)>0$. Since $C$ is 2-Gieseker semistable, $F:=C[1]$ belongs to $\cala^{\beta,\alpha}$ for $\alpha\gg0$ \cite[Lemma 3.8]{JMM}, thus $F$ is a sub-object of $A$ with $\mu(F)=\mu(A)=0$ within $\cala^{\beta,\alpha}$ for $\beta<\mu(C)=\mu(B)$ and $\alpha\gg0$.

Setting $\beta_0:=\min\{\mu(A),\beta_0'\}$ where
\begin{equation} \label{eq:b0'}
\beta_0':=\dfrac{-1}{(s-1/3)\delta_{20}(C,B)}\left( \dfrac{1}{2}\delta_{30}(C,B) - \ch_3(\calh^0(A))\ch_0(C) - \left(s+\dfrac{1}{6}\right)\mu(A) \right),
\end{equation}
note that $D_{\obeta,s}(F,A)>0$ when $\beta<\beta_0$, thus contradicting condition (3) in Theorem \ref{vert-3.2}.
\end{proof}


\subsection{Polynomial stability} \label{sec:polynomial}

Let $X$ be a smooth projective variety.  By \cite[Proposition and Definition 4.1.1]{BMT1}, giving a polynomial stability condition $\sigma$ on $\dbx$ is equivalent to giving a pair $\sigma = (Z, \cala)$ where $\cala$ is the heart of a t-structure on $\dbx$, and $Z: K(\dbx) \to \mathbb{C}[m]$ is a central charge satisfying
\begin{itemize}
       \item[(a)] There is some real number $a$ such that, for every nonzero object $E \in \cala$, we have 
    \[
     Z(E)(m) \in \{ r e^{i \pi \phi} : r \in \mathbb{R}_{>0}, \phi \in (a, a+1]\} \text{\qquad for $m \gg 0$}.
    \]
    \item[(b)] $Z$ satisfies the Harder--Narasimhan (HN) property on $\cala$.
\end{itemize}
Property (a) allows us to define, for every nonzero $E \in \cala$, a polynomial phase function $\phi (E)$ (a function germ) satisfying
$$ Z(E)(m) \in \mathbb{R}_{>0} e^{i\pi \phi (E)(m)} \text{\qquad for $m \gg 0$} $$
and such that $\phi (E)(m) \in (a,a+1]$ for all $m \gg 0$.  This allows us to define a notion of semistability for objects in $\cala$: We say $E \in \Ac$ is $Z$-(semi)stable if, for every nonzero proper subobject $F$ of $E$ in $\cala$, we have $\phi (F) \prec (\preceq) \phi (E)$, i.e.\
\[
\phi (F) < (\leq) \phi (E) \text{\qquad for $m \gg 0$}.
\]
The HN property in (b) is then defined using this notion of $Z$-semistability. Bayer constructed a family of polynomial stability conditions on $\dbx$ for any normal projective variety $X$ \cite{BayerPBSC} which included Gieseker stability.  

Suppose $X$ is a smooth projective threefold and $\omega$ is an ample class on $X$. Consider the heart \footnote{The superscript $p$ is not a parameter, but a relic from Bayer's notation: it originally stands for the perversity function that is needed as an input for constructing a polynomial stability; here we are fixing the perversity function $p(d)=-\lfloor \frac{d}{2}\rfloor$.}
\[
\Ap = \langle \coh^{\geq 2}(X)[1], \coh^{\leq 1}(X)\rangle
\]
in $\dbx$ and the polynomial central charge 
\begin{equation} \label{eq:dtpt charge}
Z(E)(m) = \sum_{i=0}^3 \int_X \rho_i m^i \omega^i \ch(E)
\end{equation}
where  $\rho_i$ ($0\leq i\leq 3$) are nonzero complex numbers such that $\rho_0, \rho_1, -\rho_2, -\rho_3$ all lie on the upper-half complex plane.  Then by \cite[Theorem 3.2.2, Proposition 6.1.1]{BayerPBSC} we have:
\begin{itemize}
\item[(i)] (PT-stability) If the phases of the $\rho_i$ satisfy the ordering
\[
\phi (-\rho_2) > \phi (\rho_0) > \phi (-\rho_3) > \phi (\rho_1),
\]
then $(Z,\Ap)$ is a polynomial stability condition with respect to $(0,1]$, and the rank-$(-1)$ semistable objects of trivial determinant are isomorphic in $\dbx$ to \textit{stable pairs} $[\mathcal{O}_X \to F]$ in the sense of Pandharipande--Thomas \cite{pandharipande2009curve} (see Definition \ref{def:PT} below for a precise definition), with $F$ placed at degree $0$ in the derived category.

\item[(ii)] (DT-stability) If the phases of the $\rho_i$ satisfy the ordering
\[
\phi (-\rho_2) > \phi (-\rho_3) > \phi (\rho_0)  > \phi (\rho_1),
\]
then $(Z,\Ap)$ is a polynomial stability condition with respect to $(0,1]$, and the rank-$(-1)$ semistable objects of trivial determinant are isomorphic in $\dbx$ to shifts by 1 of ideal sheaves of 1-dimensional subschemes of $X$.
\end{itemize}

\begin{figure}
  \caption{Arrangements of $\rho_i$ for DT and PT polynomial stability conditions on a threefold.} \label{fig:DTPTstabvec}
  \centering
  \vspace*{0.2in}
  \scalebox{1}{
\begin{tikzpicture}
    \draw[->] (-2,0) -- (2,0);
    \draw[->] (0,-1) -- (0,2);

     \draw[->,thick] (0,0) -- (1.5,0.7);
     \draw[->,thick] (0,0) -- (1.1,1.4142);
     \draw[->,thick] (0,0) -- (-1, 1.7321);
     \draw[->,thick] (0,0) -- (-1.4,1);

     \filldraw (1.5,0.7)  node[right] {$\rho_1$};
     \filldraw (1.1,1.4142)  node[right] {$\rho_0$};
     \filldraw (-1, 1.7321)  node[left] {$-\rho_3$};
     \filldraw (-1.4,1) node[left] {$-\rho_2$};

     \filldraw (0,-1.2) node[below] {DT};

    \draw[->] (3,0) -- (7,0);
    \draw[->] (5,-1) -- (5,2);

     \draw[->,thick] (5,0) -- (6.4142,1.1);
     \draw[->,thick] (5,0) -- (5.7654,1.8478);
     \draw[->,thick] (5,0) -- (4.7,1.7321);
     \draw[->,thick] (5,0) -- (3.9,1);

     \filldraw (6.4142,1.1)  node[right] {$\rho_1$};
     \filldraw (5.7654,1.8478)  node[right] {$-\rho_3$};
     \filldraw (4.7,1.7321)  node[left] {$\rho_0$};
     \filldraw (3.9,1)  node[left] {$-\rho_2$};
    \filldraw (5,-1.2) node[below] {PT};

\end{tikzpicture}
}
\end{figure}
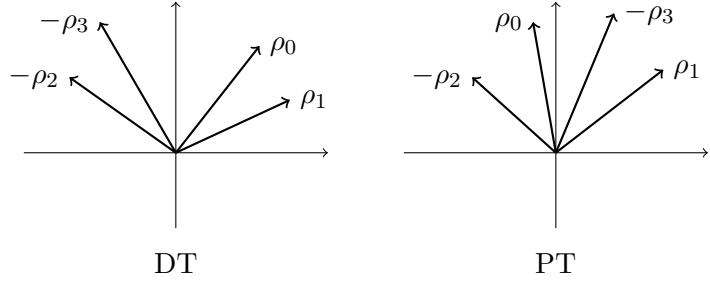

The following definition is motivated by the considerations above.

\begin{definition}
An object $A\in \dbx$ is said to be PT-(semi)stable if it is (semi)stable with respect to a polynomial stability condition $(Z,\Ap)$ where the phases of the coefficients $\rho_i$ as in equation \eqref{eq:dtpt charge} satisfy the ordering $$\phi (-\rho_2) > \phi (\rho_0) > \phi (-\rho_3) > \phi (\rho_1).$$ 
\end{definition}

On the other hand, the large volume limit can also be defined as a polynomial stability condition (see \cite[Section 4]{BayerPBSC} or \cite{Toda1}).  In this article, we will use the definition as given in \cite[Section 6.1]{BMT1}, where the large volume limit is defined as the pair $(Z_{\infty\omega, B}, \Ac)$ where 

\begin{equation*}
Z_{\infty\omega, B}(E) = Z_{m\omega,B}(E) = -\ch_3^B(E) + m^2 \frac{\omega^2}{2}\ch_1^B(E) + im(\omega \ch_2^B(E) - m^2 \frac{\omega^3}{6}\ch_0^B(E)).
\end{equation*}
Note that $(Z_{\infty\omega, B}, \Ap)$ is a polynomial stability condition with respect to $(\tfrac{1}{4}, \tfrac{5}{4})$.


\section{Stable triples} \label{sec:triples} 

On a smooth projective variety $X$, consider the set of triples
\[ \Big\{ (B, P, \xi) : B, P \in \coh (X), \xi \in \Ext^2_{\coh(X)}(P,B) \Big\}.
\]
We aim to show that this set is the object set of a well-defined category.
For the definition of the morphisms between triples, we define $A_\xi:=\mathrm{cone}(\xi)[-1]$, so that we obtain the distinguished triangle in $\dbx$
$$ B[1] \lra A_\xi \lra P \stackrel{\xi}{\lra} B[2]. $$
The key observation now is that the derived category (as well as the homotopy category) is \emph{fully} triangulated, see \cite[Section 3]{AM23}. This means that not only are there distinguished triangles, but there are also \textit{distinguished maps of distinguished triangles} in such a way that they compose. Choices have to be made to make this precise, but for the derived category we can fix this by fixing isomorphisms of the triangles in our object collection to \textit{standard} triangles given by mapping cones. We define a morphism $\psi$ from $(B, P, \xi)$ to $(B', P', \xi')$ to be a distinguished map of triangles
\begin{equation}\label{eq:TXmordef}\begin{split}
\xymatrix{
B[1] \ar[r] \ar[d]^{b[1]} &  A_\xi \ar[r]^{\zeta} \ar[d]^\alpha  & P \ar[r]^\xi \ar[d]^p & B[2] \ar[d]^{b[2]} \\
B'[1] \ar[r]^{\eta'} & A_{\xi'} \ar[r] & P' \ar[r]^{\xi'} & B'[2]
} \end{split} \end{equation}
In particular, the identity morphism on an object $(B,P,\xi)$ is defined to be the diagram
\[ \xymatrix{
B[1] \ar[r] \ar[d]^{1_B[1]} &  A \ar[r] \ar[d]^{1_A}  & P \ar[r]^\xi \ar[d]^{1_P} & B[2] \ar[d]^{1_B[2]} \\
B[1] \ar[r] & A \ar[r] & P \ar[r]^{\xi} & B[2]
} \]
The key observation is that
any two such maps differ by what Neeman in \cite{Neeman91} calls a \emph{lightning strike} (\cite[Axiom 5, Def 3.1]{AM23}). In our context that means that if we have two maps $\psi$ and $\tilde\psi$ of triangles of the form (\ref{eq:TXmordef}) given by $\alpha$ and $\tilde\alpha$ for fixed $b$ and $p$, then $\alpha-\tilde\alpha=\eta' \ell\zeta$ for some map $\ell:P\to B'[1]$.

We will then denote by $Tr(X)$ the category of triples, with objects and morphisms described above, given by distinguished morphisms of triangles. 


Observe that the object $A_\xi$ belongs to the full subcategory $D^{[-1,0]}(X)$ of $\dbx$ consisting of objects $A$ such that $\calh^{p}(A)=0$ for $p\ne-1,0$; however, the assignment $(B, P, \xi)\mapsto A_\xi$ is not functorial. To get around this issue, we consider the full subcategory $Tr^{\text{PT}}(X)$ consisting of triples $(B, P, \xi)$ satisfying the following conditions:
\begin{itemize}
\item[(i)] $B$ is a torsion-free sheaf such that $Q_B=B^{**}/B$ has pure dimension 1;
\item[(ii)] $P$ is a 0-dimensional sheaf.
\end{itemize}
Applying the functor $\Hom(P,\cdot)$ to the standard short sequence \eqref{std sqc} for the sheaf $B$ we obtain the exact sequence
\[ \cdots\longrightarrow \Hom(P,B^{**}/B)\longrightarrow\Ext^1(P,B)\longrightarrow\Ext^1(P,B^{**})\longrightarrow\cdots \]
But $B^{**}/B$ is pure 1-dimension thus $\Hom(P,B^{**}/B)=0$, and $\Ext^1(P,B^{**})=0$ because $B^{**}$ is reflexive and $P$ is 0-dimensional. Hence, $\Hom(P,B'[1])=0$ thus the shifted cone $A_\xi$ is unique up to unique isomorphism and the maps $\alpha$ are uniquely defined.

We then define a functor
$$ \mathfrak{A} : Tr^{\text{PT}}(X) \lra D^{[-1,0]}_{\coh (X)} \quad,\quad \mathfrak{A}(B,P,\xi):=\mathrm{cone}(\xi)[-1]. $$
We denote the image of this functor by $D^{\text{PT}}_{\coh(X)}$, namely the full subcategory of $D^{[-1,0]}_{\coh (X)}$ whose objects $A$ have $\calh^{-1}(A)$ and $\calh^0(A)$ satisfying the conditions in items (i) and (ii) above.

On the other hand, given any object $A \in D^{[-1,0]}_{\coh (X)}$, we have the canonical exact triangle
\[
\mathcal{H}^{-1}(A)[1] \to A \to \mathcal{H}^0(A) \overset{\xi}{\to} \mathcal{H}^{-1}(A)[2]
\]
which then gives a triple $(\mathcal{H}^{-1}(A),\mathcal{H}^0(A),\xi)$ which is an object of the category $Tr(X)$. Given a morphism $\alpha:A\to A'$, the truncation functors with respect to the standard t-structure on $\dbx$ gives a commutative diagram in $\dbx$ 
\[
\xymatrix{
  \mathcal{H}^{-1}(A)[1] \ar[r] \ar[d]^{\mathcal{H}^{-1}(\alpha)[1]} & A \ar[r] \ar[d]^\alpha &  \mathcal{H}^0(A) \ar[r] \ar[d]^{\mathcal{H}^0(\alpha)} &  \mathcal{H}^{-1}(A)[2] \ar[d]{\mathcal{H}^{-1}(\alpha)[2]} \\
  \mathcal{H}^{-1}(A')[1] \ar[r] & A' \ar[r] & \mathcal{H}^0(A') \ar[r] & \mathcal{H}^{-1}(A')[2]
  }
\]
which is well defined on $D^{\text{PT}}_{\coh(X)}$ since $\mathcal{H}^{-1}(A)$ and $\mathcal{H}^{0}(A)$ satisfy the conditions on $B$ and $P$ above. We therefore obtain a morphism between the triples $(\mathcal{H}^{-1}(A),\mathcal{H}^0(A),\xi)$ and $(\mathcal{H}^{-1}(A'),\mathcal{H}^0(A'),\xi')$. In addition, this operation takes an identity morphism in $D^{[-1,0]}_{\coh (X)}$ to an identity morphism in $Tr(X)$ by \cite[Chap.\ IV \S 4, 5. Lemma b)]{GM}, and,
since $\mathcal{H}^{-1},\mathcal{H}^0$ themselves are functors from $\dbx$ to $\coh(X)$, this operation also respects composition of morphisms. Therefore, we define the functor 
$$ \mathfrak{H} : D^{\text{PT}}_{\coh(X)} \lra Tr^{\text{PT}}(X)  \quad,\quad \mathfrak{H}(A):=\big( \calh^{-1}(A),\calh^{0}(A),\xi\big) . $$

\begin{lemma}\label{lem:AG52-29-1}
The functors
\[ \xymatrix{
Tr^{\text{PT}}(X)  \ar@<0.5ex>[r]^{\mathfrak{A}} &  D^{\text{PT}}_{\coh(X)} \ar@<0.5ex>[l]^{\mathfrak{H}}
} \]
induce mutually inverse bijections on the isomorphism classes on both sides.
\end{lemma}

\begin{proof}
The claim will follow if we can show that:
\begin{itemize}
\item[(a)] For any object $t=(B, P, \xi) \in Tr^{\text{PT}}_X$, we have $(\mathfrak{H}\circ\mathfrak{A})(t) \cong t$.
\item[(b)] For any object $A \in D^{\text{PT}}_{\coh (X)}$, we have $(\mathfrak{A}\circ\mathfrak{H})(A) \cong A$ in $\dbx$.
\end{itemize}
Statement (a) follows directly from applying \cite[Chap.\ IV \S 4, 5. Lemma b)]{GM} to $\mathfrak{A}(t)$. To see statement (b), set $\mathfrak{H}(A) = (\mathcal{H}^{-1}(A), \mathcal{H}^0(A), \xi)$.  Then $\mathfrak{A}(\mathfrak{H} (A))$ is simply the cone of $\xi:\mathcal{H}^0(A)\to\mathcal{H}^{-1}(A)[2]$, which, as we have seen above, is isomorphic to $A$ in the derived category $\dbx$, and so (b) follows.
\end{proof}


\subsection{Stable triples}

Let us now introduce one of the main characters of this section.

\begin{definition}\label{defn:stable-triple}
A \emph{(semi)stable triple} $(B,P,\eta)$ consists of
\begin{itemize}
\item[(i)] a 2-Gieseker (semi)stable sheaf $B$ such that $B^{**}/B$ has pure dimension 1, whenever non-trivial;
\item[(ii)] a 0-dimensional sheaf $P$;
\item[(iii)] a morphism $\eta\in\Hom_{\dbx}(P,B[2])$, that is \emph{completely non-trivial}, that is, no subsheaf $U\hookrightarrow P$ factors through $A_\eta\to P$, where $A_\eta:={\rm cone}(\eta)[-1]$.
\end{itemize}
When $B$ is $\mu$-stable, the triple $(B,P,\eta)$ is  said to be a $\mu$-stable triple.
\end{definition}

In other words, a (semi)stable triple is an object $(B,P,\eta)$ in $Tr^{\text{PT}}(X)$ such that $B$ is 2-Gieseker (semi)stable and $\eta$ satisfies the complete nontriviality condition.

\begin{remark}\label{rem:fourtripesequiv}
A torsion-free sheaf $B$ with $\gcd(\ch_0(B),\ch_1(B))=1$ is $\mu$-stable if and only if it is $\mu$-semistable, if and only if it is 2-Gieseker (semi)stable. In this case, for a triple $(B,P,\eta)$, the following  notions all coincide:
\begin{itemize}
\item $(B,P,\eta)$ is a $\mu$-stable triple;
\item $(B,P,\eta)$ is a stable triple;
\item $(B,P,\eta)$ is a semistable triple;
\end{itemize}
\end{remark}

Two stable triples $(B,P,\eta)$ and $(B',P',\eta')$ are isomorphic if there are isomorphisms $\beta:B\to B'$ and $\gamma:P\to P'$ such that the following diagram commutes
$$ \xymatrix{
P \ar[d]^\gamma \ar[r]^\eta & B[2] \ar[d]^{\beta} \\
P' \ar[r]^{\eta'} & B'[2]
} $$
equivalently, $A_\eta\simeq A_{\eta'}$.

The condition in item (iii) is meant to exclude the following situation. Assume that $P=\calo_{q_1}\oplus\calo_{q_2}$ for distinct points $q_1,q_2\in X$, so that 
$$ \Hom(P,B[2]) = \Hom(\calo_{q_1},B[2]) \oplus \Hom(\calo_{q_2},B[2]); $$
then $\eta$ can be expressed as a pair $(\eta_1,\eta_2)$ with $\eta_i\in\Hom(\calo_{q_i},B[2])$. The class $\eta$ is completely non-trivial if and only if both $\eta_1$ and $\eta_2$ are non-trivial.

Clearly, if $B$ is a 2-Gieseker (semi)stable sheaf such that $B^{**}/B$ is pure 1-dimensional, then $(B,0,0)$ is a (semi)stable triple. Similarly, if $P\in\coh^{\leq 0}$, then $(0,P,0)$ is also a stable triple. Below is a simple non-trivial example of a stable triple.

\begin{example}\label{eg:reflx-sh-gives-ST}
Let $F$ be a $\mu$-stable reflexive sheaf, and set $A:=F^\vee[1]$. Setting $B:=\calh^{-1}(A)\simeq F^*$ and $P:=\calh^{0}(A)\simeq\inext^1(F,\ox)$, we have the triangle 
$$ B[1] \lra A \lra P \stackrel{\eta}{\lra} B[2] $$
which defines a morphism $\eta\in\Hom(P,B[2])$ whose cone is precisely $A[1]$; we argue that $(B,P,\eta)$ is a stable triple. Indeed, $B$ also a $\mu$-stable reflexive sheaf, so in particular $B$ is 2-Gieseker semistable and $B^{**}/B$ is trivial; in addition, $P$ is a 0-dimensional sheaf. The fact that $\eta$ is completely non-trivial follows from \cite[Proposition 2.15]{JM19}.
\end{example}

\begin{lemma}\label{lem:sup-sing}
If $(B,P,\eta)$ is a triple in $Tr^{\text{PT}}(X)$ satisfying condition (iii) of Definition \ref{defn:stable-triple}, then $\supp(P)\subset\sing(B)$.
\end{lemma}
\begin{proof}
Set $\supp(P)=\{x_1,\dots,x_l\}\subset X$, so that $P=\bigoplus_i P_i$ with $\supp(P_i)=\{x_i\}$. Note that
$$\Ext^2(P_i,B) \simeq \Ext^1(B,P_i)^* \simeq H^0(\inext^1(B,P_i))^* \text{ for each  } i=1,\dots,k , $$
where the first isomorphism is given by Serre duality, while the second one comes from the fact that $\dim(P)=0$.

If any of the points $x_i$ is not in the singular locus of the sheaf $B$, then $\inext^1(B,P_i)_{x_i}=\inext^2(B_{x_i},(P_i)_{x_i})=0$ because $B_{x_i}$ is a free module over the local ring $\calo_{X,x_i}$. It follows that $P_i\into P$ factors through $A_\eta$, contradicting the complete non-triviality hypothesis.  
\end{proof}

We define $\ch(B,P,\eta):=\ch(B)-\ch(P)=\ch(A_\eta[1])$. In particular, the \textit{rank} of a stable triple $(B,P,\eta)$ is defined as the rank of $B$.
Let us recall the notion of a stable pair as originally introduced by Pandharipande and Thomas in \cite{pandharipande2009curve}.

\begin{definition} \label{def:PT}
A pair $(F,s)$ consisting of a sheaf $F$ and a section $s\in H^0(F)$ is called a \textit{PT stable pair} if $F$ is pure 1-dimensional and $\coker\{\ox\stackrel{s}{\to}F\}$ is 0-dimensional. Two pairs $(F,s)$ and $(F',s')$ are isomorphic if there is an isomorphism $\phi:F\to F'$ such that $\phi\circ s=s'$.
\end{definition}

We argue below that stable triples generalize PT stable pairs.

\begin{lemma}\label{lem:rk1ST-PTpairsequiv}
There is a 1-1 correspondence between isomorphism classes of rank 1 stable triples with $\ch_1=0$ and isomorphism classes of PT stable pairs.
\end{lemma}

\begin{proof}
Given a PT stable pair $(F,s)$, note that $\im(s)=\calo_C$ for some pure 1-dimensional subscheme $C\subset X$, thus $\ker(s)=\cali_C$, the corresponding ideal sheaf; the sheaf $P:=\coker(s)$ is 0-dimensional by hypothesis. we obtain the exact sequence
$$ 0\to \cali_C \to \ox \stackrel{s}{\to} F \to P \to 0, $$
which can be regarded as an extension class $\eta\in\Ext^2_{\coh}(P,\cali_C)=\Hom_{D(X)}(P,\cali_C[2])$; we check that this is completely non-trivial: since $A_\eta={\rm cone}(\eta)[-1]=\{\ox\stackrel{s}{\to}F\}$, a subsheaf of $P$ that factors $A_\eta$ would give be a 0-dimensional subsheaf of $F$, which cannot exist because $F$ is pure 1-dimensional. 

Therefore, the PT stable pair $(F,s)$ yields the rank 1 stable triple $(\cali_C,P,\eta)$ with $\ch_1=0$. If $(F',s')$ is isomorphic to $(F,s)$, then 
\begin{itemize}
    \item $\ker(s')\simeq\ker(s)=\cali_C$;
    \item $\coker(s')\simeq\coker(s)$;
    \item $A_\eta\simeq A_{\eta'}$;
\end{itemize}
it follows that $(\cali_C,P,\eta)$ and $(\cali_C,P',\eta')$ are isomorphic.

Conversely, let $(B,P,\eta)$ be a rank 1 stable triple with $\ch_1=0$; since $B$ is a torsion-free sheaf of rank 1, then $B=\cali_C\otimes L$, for some pure 1-dimensional scheme $C\subset X$ and a line bundle $L\in\Pic^0(X)$. Note that
$$ \Hom_{D(X)}(P,\cali_C\otimes L[2]) \simeq \Hom_{D(X)}(P\otimes L^{-1},\cali_C[2]) \simeq \Hom_{D(X)}(P,\cali_C[2]) $$
since $P\otimes L^{-1}\simeq P$. Therefore, we can just disregard the twisting line bundle $L$ and consider instead a triple of the form $(\cali_C,P,\eta)$.

Next, note that 
$$ \Hom_{D(X)}(P,\cali_C[2]) = \Ext^2(P,\cali_C) = \Ext^1(P,\calo_C) $$
since, by Serre duality, $\Ext^i(P,\calo_X)=H^{3-i}(P)^*=0$ for $i\ne3$. This means that $\eta$ can be regarded as an extension of the form
$$ 0 \longrightarrow \calo_C \longrightarrow F \longrightarrow P \longrightarrow 0; $$
it only remains for us to argue that the sheaf $F$ is pure 1-dimensional. Indeed, assume that $Z\hookrightarrow F$ is a 0-dimensional subsheaf; since $\calo_C$ is pure, the composition $Z\hookrightarrow F \onto P$ is a monomorphism, meaning that a subsheaf of $P$ factors through $F$ and thus contradicts the hypothesis of complete non-triviality for $\eta$. The PT stable pair associated with the rank 1 stable triple $(\cali_C,P,\eta)$ will then consist of the sheaf $F$ above, plus the section given by the composition $s:\calo_X\onto\calo_C\hookrightarrow F$, i.e. ${\rm cone}(\eta)[-1]=\{\ox\stackrel{s}{\to}F\}$. If $(\cali_C,P',\eta')$ is isomorphic to $(\cali_C,P,\eta)$, then $A_\eta\simeq A_{\eta'}$, meaning that the corresponding PT stable pairs are isomorphic as well.
\end{proof}

The following alternative formulations of the complete non-triviality condition, namely Definition \ref{defn:stable-triple} item (iii), will be useful later on.  

\begin{lemma}\label{eq:STcond3equivs}
Suppose $(B, P,\eta)$ is a triple in $Tr(X)^{{\text PT}}$
and set $A_\eta:={\rm cone}(\eta)[-1]$. The following four conditions are equivalent:
\begin{itemize}
\item[(iii-a)] No nonzero subsheaf $U \subseteq P$ factors through the canonical map $p : A_\eta \to P$ in $\dbx$.
\item[(iii-b)] No nonzero morphism $U \to P$ in $\dbx$, with $U \in \coh^{\leq 0}(X)$, factors through $p$.
\item[(iii-c)] For any $U \in \coh^{\leq 0}(X)$, the induced map $\overline{p} : \Hom (U, A_\eta) \to \Hom (U, P)$ has $\image (\overline{p})=0$.
\item[(iii-d)] For any $U \in \coh^{\leq 0}(X)$, we have $\Hom_{\dbx}(U,A_\eta)=0$.
\end{itemize}
\end{lemma}

\begin{proof}
Suppose  $(B,P,\eta)$ satisfies (i), (ii), and (iii).  The equivalence of (iii-b) and (iii-c) follows from the definition of $\overline{p}$, while (iii-c) clearly implies (iii-a).  

(iii-a) $\Rightarrow$ (iii-c): Take any morphism $j : U \to P$ where $U \in \coh^{\leq 0}(X)$, and assume  that $j$ factors as 
\[
U \overset{\wt{j}}{\to} A_\eta \overset{p}{\to} P
\]
for some $\wt{j}$ and where $p$ is the canonical map.  Our goal is to show that $j$ must be the zero map.

In order to apply assumption (iii-a), consider the factorisation of $j$ in $\Ac^p$
\begin{equation}\label{AG52-14-1}
  U \overset{j_1}{\longrightarrow} \ol{U} \overset{j_2}{\longrightarrow} P
\end{equation}
where $\ol{U}$ is the image of $j$ in $\Ap$, and $j_1, j_2$ are surjective and injective in $\Ap$, respectively.  Note that \eqref{AG52-14-1} is also a short exact sequence in $\coh^{\leq 0}(X)$ since $\coh^{\leq 0}(X)$ is  a Serre subcategory of $\Ap$.  We will also write $K$ to denote the kernel of $j_1$ in $\Ap$, so that we have
\begin{equation}\label{eq:AG52-14-2}
0 \to K \overset{k}{\longrightarrow} U \overset{j_1}{\longrightarrow} \ol{U} \to 0
\end{equation}
which is a short exact sequence in both $\Ap$ and $\coh^{\leq 0} (X)$.

Consider the  diagram 
\[
\xymatrix{
\Hom (\ol{U}, B[1]) \ar[r] \ar[d] & \Hom (\ol{U}, A_\eta) \ar[r]^{p^{\ol{U}}} \ar[d]^{(j_1)_A} & \Hom (\ol{U},P) \ar[d]^{(j_1)_P} \\
\Hom (U, B[1]) \ar[r] \ar[d] & \Hom (U, A_\eta) \ar[r]^{p^U} \ar[d]^{k_A} & \Hom (U,P) \ar[d]^{k_P} \\
\Hom (K, B[1]) \ar[r] & \Hom (K, A_\eta) \ar[r]^{p^K} & \Hom (K, P)
}
\]
where the columns are obtained by applying the functors  $\Hom (-, B[1]), \Hom (-,A_\eta)$ and $\Hom (-,P)$ to \eqref{eq:AG52-14-2}, and the rows are obtained by applying $\Hom (\ol{U},-), \Hom (U,-)$ and $\Hom (K,-)$ to the canonical exact triangle
\[
B[1] \to A_\eta \to P \to B[2].
\]
The commutativity of the diagram and the exactness of all the rows and columns follow directly from the construction.  Also, all the terms in the leftmost column are zero by Lemma \ref{lem:hd1equivconds}.

Now, we know $j=j_2\circ j_1$, i.e.\ $j=((j_1)_P)(j_2)$.  If we can show that $j_2$ lies in the image of $p^{\ol{U}}$, then it means $j_2$ factors through $p$.  Since $j_2$ is also an injection in $\coh (X)$,  assumption (iii-a)  implies that $j_2$ must be zero, i.e.\ $j$ must be zero, completing the proof.

Since $A_\eta$ and $P$ both lie in the heart $\Ap$, the maps $p^{\ol{U}}, (j_1)_P, (j_1)_A, p^U$ are all injective.  Recalling that $j=(p^U)(\wt{j})$, the claim that $j_2$ lies in the image of $p^{\ol{U}}$ will follow if we can show that $\wt{j}$ lies in the image of $(j_1)_A$ or, equivalently, $\wt{j}$ lies in the kernel of $k_A$.

Furthermore, since $p^K$ is injective, it suffices to show $(p^K \circ k_A)(\wt{j})=0$, i.e.\ $p\wt{j}k=0$.  Now, 
\[
  p\wt{j}k=jk=j_2j_1k=0
\]
since $j_1k=0$, and we are done.

(iii-d) $\Rightarrow$ (iii-b): Suppose $U \in \coh^{\leq 0}(X)$ and $j : U \to P$ is a morphism in $\dbx$ that factors through $p$.  Our goal is to show that $j$ must be zero.

We begin with the short exact sequence of sheaves $0 \to B \to B^{\ast\ast} \to B^{\ast\ast}/B \to 0$.  Considering this as an exact triangle and rotating it, we obtain the exact triangle
\[
B^{\ast\ast}/B \to B[1] \to B^{\ast\ast}[1] \to (B^{\ast\ast}/B) [1].
\]
That $B^{\ast\ast}/B$ is reflexive implies $\Hom (U, (B^{\ast\ast}/B)[1])\cong \Ext^1 (U,B^{\ast\ast}/B)=0$ \cite[Lemma 3.3(ii)]{Lo17}, while $\Hom (U,B^{\ast\ast}/B)=0$ by (i).  Therefore, $\Hom (U, B[1])=0$.  Applying $\Hom (U,-)$ to the exact triangle $B[1]\to A_\eta \to P \to B[2]$ then shows the induced map 
\[
 \overline{p} : \Hom (U,A_\eta) \to \Hom (U,P)
\]
is an injection.  Our assumption that $j$ factors through $p$ means $j$ lies in the image of $\overline{p}$, while (iii-d) implies $\Hom (U,A_\eta)=0$.  Hence $j=0$ as wanted.

(iii-b) $\Rightarrow$ (iii-d): Suppose (iii-b) holds.  Take any $U \in \coh^{\leq 0}(X)$ and any $g \in \Hom (U, A_\eta)$.  Now $pg$ is a morphism $U \to P$ that factors through $p$ and so, by (iii-d), we have $pg=0$.  The same argument as in the previous paragraph shows $\overline{p}$ is injective.  Since $0=pg=\overline{p}(g)$, it follows that $g=0$.
\end{proof}

\begin{remark}
By \cite[Example 5.2(a)]{Lo17}, condition (iii-d) above allows us to think of stable triples as objects lying in the category $\mathcal{E}_0$ defined in \cite[5.3]{Lo17}; it then follows from \cite[Theorem 5.9]{Lo17} that we can write the subcategory of stable triples in $\dbx$ as a disjoint union indexed by surjections of the form $\mathcal{E}xt^1(A,\mathcal{O}_X) \to Q$ in $\coh (X)$ where $A$ is a coherent sheaf of homological dimension at most 1, and $Q$ is a 0-dimensional sheaf.
\end{remark}


\subsection{Standard representation of triples}
Given a triple $(B,P,\eta)$ in $Tr(X)^{{\text PT}}$, the associated object $A_\eta$ can be represented by a two-step complex $F\xrightarrow{\phi} T$ for some coherent sheaves $F$ and $T$ so that $\ker(\phi)=B$ and $\coker(\phi)=P$. Are there some standard choices for $F$ and $T$, like in the rank 1 case (in which $F=\ox$ and $T$ is pure 1-dimensional, as in the proof of Lemma \ref{lem:rk1ST-PTpairsequiv})? The answer is positive at least under the following vanishing condition.

\begin{lemma} \label{lem:presentation}
Let $(B,P,\eta)$ be a triple in $Tr(X)^{{\text PT}}$.
If $\Ext^2(P,B^{**})=0$, then there is a pure 1-dimensional sheaf $T$ and a morphism $\phi\in\Hom(B^{**},T)$ such that $A_\eta\simeq\{B^{**}\xrightarrow{\phi} T\}$.
\end{lemma}

Note that the hypothesis holds when $B^{**}$ is locally free, or, more generally, when $\supp(P)\cap\supp(\inext^1(B^{**}),\ox)=\emptyset$.

\begin{proof}
First, note that if $F$ is a reflexive sheaf and $\dim(P)=0$, then $\Ext^1(P,F)=0$. Indeed, recall that $F$ admits a locally free resolution of the form $0\to L_1 \to L_0 \to F$. Applying the functor $\Hom(P,\cdot)$ to this sequence, we obtain
$$ \Ext^1(P,L_0) \longrightarrow \Ext^1(P,F) \longrightarrow \Ext^2(P,L_1); $$
since $Ext^q(P,L_i)\simeq H^0\big(\inext^q(P,\ox)\otimes L_i\big)=0$ when $q<3$ because $P$ is 0-dimensional, we obtain the desired vanishing.

Now apply the functor $\Hom(P,\cdot)$ to the sequence $0\to B\to B^{**}\to Q_B\to 0$ to conclude that $\Ext^2(P,B)\simeq\Ext^1(P,Q_B)$ since $\Ext^2(P,B^{**})=0$. This means that $\eta\in\Ext^2(P,B)$ corresponds to an extension class $\tilde\eta\in\Ext^1(P,Q_B)$, providing an exact sequence $0\to Q_B \to T \to P\to 0$; the fact that $\eta$ is completely non-trivial implies that $T$ is pure 1-dimensional, since $Q_B$ is pure 1-dimensional and any 0-dimensional subsheaf of $T$ would also be a subsheaf of $P$.

Therefore, the composition $\phi:B^{**}\onto Q_B \hookrightarrow T$, has the desired properties.
\end{proof}

Conversely, let $(E,T,\phi)$ be a triple consisting of a torsion-free sheaf $E$ with $\hd(E)\le1$, a pure 1-dimensional sheaf $T$, and a morphism $\phi\in\Hom(E,T)$ such that $\coker(\phi)$ is 0-dimensional. Similar triples (with other conditions on the sheaves $E$ and $T$ and on the morphism $\phi$) have long appeared in the literature under the nomenclature of \textit{holomorphic triples}, see for instance \cite{BGPG,RH}, and can be considered as a generalization of PT stable pairs, compare with Definition \ref{def:PT} above.

Considering $A:=\{E\stackrel{\phi}{\to}T\}$ as an object in $\dbx$, set $B:=\calh^{-1}(A)\simeq\ker(\phi)$ and $P:=\calh^0(A)\simeq\coker(\phi)$; in addition, the exact sequence
$$ 0\lra B \lra E \stackrel{\phi}{\lra} T \lra P \lra 0 $$
defines a class $\eta\in\Ext^2(P,B)$. We argue that $(B,P,\eta)$ is an object in $Tr^{{\text PT}}(X)$. 

Indeed, $\dim(P)=0$ by hypothesis. Let $I:=\im(\varphi)$ and note that, since $I\hookrightarrow T$, $I$ is also pure 1-dimensional; we also have the commutative diagram
$$ \xymatrix{
& 0\ar[d] & 0\ar[d] & & \\
0\ar[r] & B \ar[r]\ar[d] & E \ar[r]\ar[d] & I \ar[r] & 0 \\
 & B^{**} \ar[r]^{\sim}\ar[d] &  E^{**} \ar[d] & & \\
 & Q_B \ar[r]\ar[d] & Q_E \ar[d] & & \\
 & 0 & 0 & &
} $$
Therefore, we obtain the short exact sequence $0\to I \to Q_B \to Q_E \to0$; since both $I$ and $Q_E$ are pure 1-dimensional, then $Q_B$ is also pure 1-dimensional. 

Note, in addition, that $\eta$ is completely nontrivial: if $U$ is a subsheaf of $P$, then it cannot lift to $A$ precisely because $T$ is pure 1-dimensional. 

Finally, since $B^*\simeq E^*$ if $E$ is $\mu$-stable, then $B$ is also $\mu$-stable and $(B,P,\eta)$ is a stable triple.


\subsection{Constructing  stable triples}

In \cite{Lo17}, the second author describes a method for constructing all \textit{PT stable objects}, which are slightly different from stable triples but coincide with them under a coprime condition on degree and rank (see Section \ref{sec:PT} for precise comparisons between them). 
In this subsection, we describe how the constructions in \cite{Lo17} can be used to construct all stable triples.  The approach in this section is different from that of Lemma \ref{lem:presentation}, in that we do not try to obtain a standard representation of stable triples.  We note that attempts at obtaining standard representations of PT stable objects can be found in works including \cite[Section 3.2]{toda2020hall} and \cite[Section 6.8]{Lo17}. We  recall:

\begin{lemma}\cite[Lemma 3.3]{Lo17}\label{lem:Lo17lem3p3}
Let $X$ be a smooth projective threefold.  Suppose $E \in \coh^{\geq 2}(X)$.  Then
\begin{itemize}
    \item[(i)] $E$  has homological dimension 1 if and only if $\Hom_{\dbx}(\coh^{\leq 0}(X), E[1])=0$.
    \item[(ii)] If $E$ is torsion-free, then $E$ is reflexive if and only if $\Hom_{\dbx}(\coh^{\leq 1}(X),E[1])=0$.
\end{itemize}
\end{lemma}

To state Proposition \ref{prop:enum-stabletriples-1} below, we need some notation and constructions from \cite{Lo17}.  Consider the categories
\begin{align*}
    \mathcal{E} &= \{ E \in D^{[-1,0]}_{\coh (X)} : \mathrm{hd}(H^{-1}(E)) \leq 1, H^0(E) \in \coh(X)^{\leq 0}(X) \}, \\
    \mathcal{E}_0 &= \{ E \in \mathcal{E} : \Hom_{\dbx} (\coh^{\leq 0}(X), E)=0 \}.
\end{align*}
Additionally, let $\mathcal{M}or(\coh(X))$ denote the category where the objects are morphisms $A \to B$ in $\coh(X)$, and morphisms are commutative diagrams.  Then let $\mathcal{M}or(\coh(X))/\thicksim$ denote the set of isomorphism classes in $\mathcal{M}or (\coh (X))$.  Let us also write $C_{3,1}^s$ (respectively\ $C_{3,1}^{ss}$) to denote the full subcategory of $\coh (X)$ consisting of 2-Gieseker stable (respectively\ 2-Gieseker semistable) sheaves.

\begin{construction}\label{constr:enum-stabtriple-2}
Suppose $E$ is an object in $\mathcal{E}$.  We have the canonical exact triangle in $\dbx$
\[
\calh^{-1}(E)[1] \to E \to \calh^0(E) \overset{w}{\to} \calh^{-1}(E)[2].
\]
Applying $(-)^\vee [2]$ gives the exact triangle
\[
\calh^0(E)^\vee [2] \to E^\vee [2] \to \calh^{-1}(E)^\vee [1] \overset{w^\vee [3]}{\longrightarrow} \calh^0(E)^\vee [3]
\]
where $\calh^{-1}(E)^\vee [1] \in D^{[-1,0]}_{\coh (X)}$ and $\calh^0(E)^\vee [3] \in \coh^{\leq 0}(X)$.  We denote the morphism $w^\vee [3]$ in $\dbx$ by $\widetilde{F}(E)$.
\end{construction}

\begin{construction}\label{constr:pre-stabtriple}
Suppose $B$ is a torsion-free sheaf on $X$ of homological dimension at most 1.  We have the canonical map in $\dbx$
\[
  B^\vee [1] \overset{c}{\to}  \calh^0(B^\vee [1]) = \mathcal{E}xt^1 (B,\calo_X).
\]
For any 0-dimensional sheaf quotient
\[
\mathcal{E}xt^1 (B,\calo_X) \overset{r}{\onto} Q,
\]
we can take the mapping cone $C$ of the composition $rc$ to form the exact triangle
\[
  B^\vee [1] \overset{rc}{\to} Q \to C \to B^\vee [2].
\]
Applying the composite functor $-^\vee \circ [-3]$ then gives the exact triangle
\begin{equation}\label{eq:enum-triples-1}
  B[1] \to C^\vee [3] \to Q^\vee [3] \overset{\eta}{\to} B[2]
\end{equation}
where $Q^\vee [3]$ is a 0-dimensional sheaf sitting in degree 0 in $\dbx$.
\end{construction}

Putting together Constructions \ref{constr:enum-stabtriple-2} and \ref{constr:pre-stabtriple}, we obtain the following characterization of semistable triples.

\begin{proposition}\label{prop:enum-stabletriples-1}
Let $X$ be a smooth projective threefold.  
\begin{enumerate}
    \item[(i)] Suppose $B$ is a 2-Gieseker (semi)stable sheaf on $X$ with homological dimension at most 1,  $r : \mathcal{E}xt^1 (B,\calo_X) \to Q$ is an epimorphism of sheaves where $Q$ is supported in dimension 0, and $\eta$ is as in Construction \ref{constr:pre-stabtriple}.  Then $(B,Q^\ast, \eta)$ is a (semi)stable triple.
    \item[(ii)]  Suppose $(B, P, \eta)$ is a (semi)stable triple.  Let $A_\eta = \mathrm{cone}(\eta)[-1]$.  Then $\calh^0(\widetilde{F}(A_\eta)) : \mathcal{E}xt^1 (B,\calo_X) \to P^\ast$ is an epimorphism of sheaves.
\end{enumerate}
\end{proposition}

\begin{proof}
(i) By Lemma \ref{lem:hd1equivconds}, the quotient $B^{\ast\ast}/B$ is pure of dimension 1.  By \cite[Lemma 5.7]{Lo17} and its proof, Construction \ref{constr:pre-stabtriple} gives  an object $E \in\mathcal{E}_0$ such that the morphism of sheaves $H^0(\widetilde{F}(G))$ is isomorphic to $r$ in $\mathcal{M}or(\coh(X))$; moreover, $E$ fits in an exact triangle
\[
B [1] \to E \to Q^\ast \overset{\eta}{\to} B[2].
\]
Since $E$ lies in $\mathcal{E}_0$, we have $\Hom (\coh^{\leq 0}(X), E)=0$.  Note that   $E$ is an object of $Tr(X)^{{\text PT}}$.  By Lemma \ref{eq:STcond3equivs}, the extension class $\eta \in \Ext^2(Q^\ast, B)$ satisfies complete nontriviality, and so the triple $(B, Q^\ast, \eta)$ is a (semi)stable triple.  

(ii) By Lemma \ref{eq:STcond3equivs}, we have $\Hom (\coh^{\leq 0}(X), A_\eta)=0$. Since $B^{\ast\ast}/B$ is pure 1-dimensional, we have $\mathrm{hd}(B)\leq 1$ by Lemma \ref{lem:hd1equivconds}.  Hence $A_\eta \in \mathcal{E}_0$.  That $H^0(\widetilde{F}(A_\eta))$ is a surjective morphism of sheaves then follows from \cite[Lemma 5.1]{Lo17}.
\end{proof}

\begin{remark}\label{rem:AG54-17-1}
Proposition \ref{prop:enum-stabletriples-1} says that there is a surjective function from the set $(B,P,\eta)$ of all stable (respectively\ semistable) triples to the set of pairs $(B,r)$ in $C_{3,1}^s \times (\mathcal{M}or(\coh(X))/\thicksim)$ (respectively\ $C_{3,1}^{ss} \times (\mathcal{M}or(\coh(X))/\thicksim$)) satisfying
\begin{itemize}
\item[(a)] $B$ is a 2-Gieseker stable (respectively\ semistable) sheaf such that $B^{\ast\ast}/B$ is pure 1-dimensional;
\item[(b)]  $r : \mathcal{E}xt^1 (B,\calo_X) \to Q$ is a surjective morphism in $\coh (X)$ with $Q \in \coh^{\leq 0}(X)$.
\end{itemize}
In particular, part (i) of Proposition \ref{prop:enum-stabletriples-1} says we can construct (semi)stable triples by starting with a pair $(B,r)$ of the form above and applying Construction \ref{constr:pre-stabtriple}.  By \cite[6.2.3]{Lo17}, we obtain all the (semi)stable triples this way.
\end{remark}


\section{PT stability} \label{sec:PT}

In this section, we establish connections among the following on a threefold:
\begin{itemize}
    \item[(i)] PT-stability;
    \item[(ii)] the large volume limit;
    \item[(iii)] vertical stability;
    \item[(iv)] stable triples.
\end{itemize}

Recall that approximations of vertical stable objects $E$ in terms of the cohomology sheaves $\mathcal{H}^i(E)$ were stated in Theorem \ref{vert-3.1} and Theorem \ref{vert-3.2}.  The definition of stable triples in Section \ref{sec:triples} gave a more concise way to think of these approximations.

In this section, we take a different approach. Instead of studying vertical stability by characterizing its cohomology sheaves with respect to the standard t-structure, we study vertical stability itself as a stability condition.  We first show that vertical (semi)stability coincides with $\sigma_{\infty,B}$-(semi)stability (Theorem \ref{thm:AG53-39-1}).  Here, $\sigma_{\infty,B}$ is a class of polynomial stability conditions that generalizes the \textit{large volume limit} as defined in \cite[Section 6.1]{BMT1}.    

Then, we show that when the $B$-field is sufficiently negative, $\sigma_{\infty,B}$-(semi)stability coincides with PT-(semi)stability (Theorem \ref{thm:PTislimtstab-eff}).  This gives us the  equivalences of polynomial stability conditions
\[
\text{PT-stability}\,\, \Leftrightarrow \,\, \text{$\sigma_{\infty, B}$-stability for $B \ll 0$} \,\, \Leftrightarrow \,\, \text{vertical stability for $B\ll 0$}
\]
as well as their analogues if we replace all instances of \textit{stability} with \textit{semistability}. Since these are equivalences of stability conditions, for \emph{any} fixed Chern character, they yield isomorphisms of the corresponding moduli spaces of stable objects with that Chern character.

As a result of the above equivalences, Theorem \ref{vert-3.1} and Theorem \ref{vert-3.2} can be thought of as characterisations of PT-(semi)stable objects $E$ in terms of their cohomology sheaves $\calh^i(E)$; these are more refined characterisations than those originally stated in \cite{Lo2}.    

In Proposition \ref{prop:PTequivSTnSST}, we also prove that PT-stable objects coincide with stable triples when we impose a coprime assumption on rank and degree.  Without the coprime assumption, we show in Propositions \ref{pro:AG53-92-1} and \ref{pro:AG53-92-2}, that when $s>0$, we have the implications
\begin{align*}
\text{ stable triple } &\Rightarrow \text{ PT-stable object ($\Leftrightarrow$ vertical stable object for $\beta <\beta_0$)}  \\
&\Rightarrow \text{ PT-semistable object ($\Leftrightarrow$ vertical semistable object for $\beta <\beta_0$)} \\
&\Rightarrow \text{ semistable triple}
\end{align*}
where $\beta_0$ depends only on the Chern character of the object.  These equivalences explain why many results on vertical stable objects and stable triples resemble those of PT-stable objects.

\begin{remark}
Toda proved in \cite[Theorem 4.7]{Toda1} that the moduli of rank-one stable objects with respect to the large volume limit coincide with the moduli of rank-one PT-stable objects (i.e.\ the moduli of PT stable pairs) when the $B$-field is sufficiently negative.  Theorem \ref{thm:PTislimtstab-eff} can be thought of as a higher-rank analogue of \cite[Theorem 4.7]{Toda1}.
\end{remark}


\subsection{Properties  of PT-semistable objects}

\begin{proposition}\cite[3.1 and Proposition 2.24]{Lo2} \label{prop:tech-0}
Let $X$ be a smooth projective threefold. If $E$ is a PT-semistable object in $\Ap$ of nonzero rank, then 
\begin{itemize}
    \item[(i)] $\mathcal{H}^{-1}(E)$ is a 2-Gieseker semistable torsion-free sheaf.
    \item[(ii)] $\mathcal{H}^0(E) \in \coh^{\leq 0}(X)$.
    \item[(iii)] $\Hom_{\dbx}(\coh^{\leq 0}(X), E)=0.$
\end{itemize}
If $E$ is an object of $\Ap$ with nonzero rank, and its rank and degree  are coprime, then the following are equivalent:
\begin{itemize}
    \item $E$ is PT-stable.
    \item $E$ is PT-semistable. 
    \item $E$ satisifes (i), (ii), and (iii).
\end{itemize}
\end{proposition}
In particular, under the coprime assumption, (i) is equivalent to 
\begin{itemize}
    \item[(i')] $\mathcal{H}^{-1}(E)$ is a $\mu$-stable torsion-free sheaf.
\end{itemize}

\begin{lemma}\label{lem:tech-1}
Let $X$ be a smooth projective threefold, and let $F$ be a torsion-free sheaf on $X$. Then for any $T \in \coh^{\leq 1}(X)$, we have
\[
\Hom (T, F^{\ast\ast}/F) \cong \Hom (T, F[1]).
\]
\end{lemma}

\begin{proof}
From the short exact sequence of sheaves
\[
  0 \to F\to F^{\ast\ast} \to F^{\ast\ast}/F \to 0
\]
we obtain the short exact sequence in $\Ap$
\[
  0 \to F^{\ast\ast}/F \to F[1] \to F^{\ast\ast}[1] \to 0.
\]
Since $F^{\ast\ast}$ is torsion-free and reflexive, we have $\Hom (\coh^{\leq 1}(X), F^{\ast\ast}[1])=0$ by \cite[Lemma 4.20]{CL}.  The result then follows.
\end{proof}

\begin{lemma}\label{lem:tech-2}
If $E\in\Ap$ is a PT-semistable object of nonzero rank, then $\hd\big(\mathcal{H}^{-1}(E)\big)\le1$.
\end{lemma}

\begin{proof}
Since $E$ is PT-semistable in $\Ap$ of nonzero rank, Proposition \ref{prop:tech-0} tells us that $\mathcal{H}^{-1}(E)$ is a semistable torsion-free sheaf and $\Hom (\coh^{\leq 0}(X), E)=0$.  Since  $\coh^{\leq 0}(X)$ is a Serre subcategory of $\Ap$, from the canonical short exact sequence in $\Ap$
\[
0 \to \mathcal{H}^{-1}(E)[1] \to E \to \mathcal{H}^0(E) \to 0
\]
 we get $\Hom (\coh^{\leq 0}(X), \mathcal{H}^{-1}(E)[1])=0$.  Then, from the $\Ap$-short exact sequence
\[
0 \to \mathcal{H}^{-1}(E)^{\ast\ast}/\mathcal{H}^{-1}(E) \to \mathcal{H}^{-1}(E)[1] \to \mathcal{H}^{-1}(E)^{\ast\ast}[1] \to 0,
\]
we obtain $\Hom (\coh^{\leq 0}(X), \mathcal{H}^{-1}(E)^{\ast\ast}/\mathcal{H}^{-1}(E))=0$, i.e.\ $\mathcal{H}^{-1}(E)^{\ast\ast}/\mathcal{H}^{-1}(E)$ is a pure 1-dimensional sheaf when it is nonzero.  The result now follows from Lemma \ref{lem:hd1equivconds}.
\end{proof}


\subsection{Vertical stability and the large volume limit}\label{sec:lambdastabnLVL}

Recall that the large volume limit (see \cite[Section 4]{BayerPBSC} or \cite{Toda1}) can be considered as a polynomial stability condition.  In   \cite[Section 6.1]{BMT1}, the large volume limit is defined as the polynomial stability condition given by the pair  $(Z_{\infty \omega, B}, \mathcal{A}^p)$ where
\begin{equation}\label{eq:limstabcentralcharge}
Z_{\infty\omega, B} = Z_{m\omega,B} = -\ch_3^B + m^2 \frac{\omega^2}{2}\ch_1^B + im(\omega \ch_2^B - m^2 \frac{\omega^3}{6}\ch_0^B).
\end{equation}
Note that $(Z_{\infty \omega, B}, \mathcal{A}^p)$ is a polynomial stability condition with respect to the interval $(\tfrac{1}{4}, \tfrac{5}{4})$, i.e.\ for every $0 \neq E \in \Ap$, we have $\phi (E)(m) \in (\tfrac{1}{4}, \tfrac{5}{4})$ for $m \gg 0$.
Let $\mathcal{Q}^p$ denote the slicing for this polynomial stability condition, and define  $\mathcal{C}^p = \mathcal{Q}^p(0,1]$ as in \cite[Definition 6.1.1]{BMT1}.

Recall the central charge $Z_{\beta, \alpha, s}$ from \cite{JMM}
\begin{equation}\label{eq:JMMcentralcharge}
Z_{\beta, \alpha, s} = -\ch_3^\beta + \left( s+\frac{1}{6}\right)\alpha^2 \ch_1^\beta + i\left(\ch_2^\beta - \frac{\alpha^2}{2}\ch_0^\beta\right)
\end{equation}
when $X$ has Picard rank 1 with $H$ as an ample generator.  

If we set
\[
\alpha = \frac{m}{\sqrt{3}}  \text{ and } s= \frac{4}{3},
\]
then  $\tfrac{\alpha^2}{2}=\tfrac{m^2}{6}$ and $(s+\tfrac{1}{6})\alpha^2=\tfrac{m^2}{2}$, in which case the central charge \eqref{eq:limstabcentralcharge} for the large volume limit coincides with  \eqref{eq:JMMcentralcharge}   up to a $\mathrm{GL}^+\!(2,\mathbb{R})$-action with $\omega = H$ and $B=\beta H$. This suggests that when $s=\tfrac{4}{3}$, vertical stability and the large volume limit may coincide.  This is indeed the case, and we will make this precise for any $s>\frac{1}{3}$ in the next subsection.


\subsection{A generalization of the large volume limit, $\sigma_{\infty, B}$}  \label{sec:LVLgeneralization} In this section, we show that on a smooth projective threefold $X$, vertical stability is a polynomial stability condition in the sense of Bayer.  We begin by defining a class of polynomial stability conditions that contains the large volume limit.

Suppose $f, g \in \mathbb{R}[m]$ are monomials of degree 2 with positive coefficients, and let us write 
\[
f = f_2m^2, \text{\quad} g=g_2m^2
\]
where $f_2, g_2 \in \mathbb{R}_{>0}$.  Consider the central charge
\begin{equation}\label{eqn:AG53-34-3}
 Z = -\ch_3^B + f\tfrac{\omega^2}{2}\ch_1^B + i m\left(\omega \ch_2^B - g \tfrac{\omega^3}{6} \ch_0^B\right)
\end{equation}
as well as the modification
\begin{equation} \label{eq:AG53-34-2b}
Z' =   \begin{pmatrix} 1 & 0 \\ 0 & 1/m \end{pmatrix} Z 
=   -\ch_3^B + f\tfrac{\omega^2}{2}\ch_1^B + i \left(\omega \ch_2^B - g \tfrac{\omega^3}{6} \ch_0^B\right).
\end{equation}

We first show that $Z'$, or equivalently $Z$, is the central charge of a polynomial stability condition.

\begin{lemma}\label{lem:AG53-43-1}
Let $X$ be a smooth projective threefold, and $\omega, B$ divisors on $X$ with $\omega$ ample.  Then $(Z,\Ac^p)$ is a polynomial stability condition.
\end{lemma}

In the proof of this lemma, we will make use of a torsion pair $(\Ac^p_1, \Ac^p_{1/2})$ in $\Ac^p$ (see \cite[Lemmas 2.16, 2.17, 2.19]{Toda08} or \cite[Lemma 4.2.2]{BMT1}) given by
\begin{align*}
    \Ac^p_1 &= \langle F[1], k(x) : F \text{ is pure 2-dimensional}, x \in X \rangle \\
    \Ac^p_{1/2} &= \{ E \in \Ac^p : \Hom (\Ac^p_{\geq 1}, E)=0 \}.
\end{align*}

\begin{proof}
The argument is essentially the same as that of \cite[Lemma 4.2.3]{BMT1}.   For any object $E \in \Ac_1^p$, we have $\dimension E = 2$ or $0$, and so $\phi_Z(E) \to 1$.  For any object $E \in \Ac_{1/2}^p$, we have $\dimension E = 3$ or $1$, in which case $\phi_Z(E) \to \tfrac{1}{2}$.  Hence $Z$ is a polynomial stability function on $\Ac^p$.  Then, since we have $\phi (E') \succ \phi (E'')$ for any nonzero objects $E' \in \Ac_1^p, E'' \in \Ac_{1/2}^p$, and both $\Ac_1^p, \Ac_{1/2}^p$ are subcategories of $\Ac^p$ of finite length  \cite[Lemma 4.2.2]{BMT1}, it follows from  \cite[Lemma 4.1.2]{BMT1}  that $Z$ has the HN property on $\Ac^p$.
\end{proof}

\begin{definition}[$\sigma_{\infty, B}$-stability]
Let $X$ be a smooth projective threefold, and $\omega, B$ divisors on $X$ with $\omega$ ample.  Suppose $f, g \in\mathbb{R}[m]$ are monomials of degree 2 with positive coefficients, and let 
\begin{equation}
 Z = -\ch_3^B + f\tfrac{\omega^2}{2}\ch_1^B + i m\left(\omega \ch_2^B - g \tfrac{\omega^3}{6} \ch_0^B\right)
\end{equation}
be as in \eqref{eqn:AG53-34-3}.  We define $\sigma_{\infty,B}$ to be the pair $(Z, \Ac^p)$, a polynomial stability condition with respect to the interval $(\tfrac{1}{4}, \tfrac{5}{4})$.
\end{definition}

Since $Z'$ differs from $Z$ via a stretch along the vertical direction, it follows that $(Z', \Ac^p)$ is also a polynomial stability condition, which we will denote by $\sigma'_{\infty, B}$.  Note that $\sigma'_{\infty, B}$ is no longer a polynomial stability with respect to the interval $(\tfrac{1}{4}, \tfrac{5}{4})$.  It also follows that an object $E \in \dbx$ is $\sigma_{\infty, B}$-(semi)stable if and only if it is $\sigma'_{\infty, B}$-(semi)stable.  Let $R, R'$ denote the slicings of $\sigma_{\infty, B}, \sigma'_{\infty, B}$, respectively, so that $\Ac^p = R(\tfrac{1}{4}, \tfrac{5}{4})$.

To show that $\sigma'_{\infty, B}$ coincides with vertical stability,  it is easier to use a slightly different heart from $\Ac^p$.  To this end, we recall a torsion pair that was used in studying the large volume limit in \cite[Section 6]{BMT1}.

Given a smooth projective threefold $X$ and divisors $\omega, B$ on $X$ where $\omega$ is ample, consider the slope function $\widehat{\mu}_{\omega, B}$ on $\coh^{\leq 2}(X)$ given by, for any nonzero $E \in \coh^{\leq 2}(X)$,
\[
\widehat{\mu}_{\omega, B} (E) = \begin{cases} \dfrac{\omega \ch_2^B(E)}{\omega^2 \ch_1^B(E)} &\text{ if $\omega^2 \ch_1^B(E)\neq 0$} \\
\infty &\text{ if $\omega^2 \ch_1^B(E)= 0$} \end{cases}.
\]
Now, define the torsion pair $(\Tc^p, \Fc^p)$ in $\coh (X)$, where 
\begin{align*}
    \Tc^p &= \{ E \in \coh^{\leq 2}(X) : \widehat{\mu}_{\omega, B; \, \mathrm{min}}(E)>0 \} \\
    \Fc^p &= \{ E \in \coh (X): \Hom (\Tc^p, E)=0 \}.
\end{align*}

\begin{lemma}\label{lem:AG53-43-2}
 $R (0,1] = \langle \Fc^p [1], \Tc^p\rangle$.
\end{lemma}

\begin{proof}
We follow the argument of \cite[Lemma 6.1.2]{BMT1} with a slight modification at the end.  Note that since $R(0,1]$ and $\langle \Fc^p [1], \Tc^p\rangle$ are both hearts of bounded t-structures, it suffices to show that $\langle \Fc^p [1], \Tc^p\rangle$ is contained in $R(0,1]$.  We do this by considering various types of objects in $\Fc^p[1]$ and $\Tc^p$.  To this end, note that the same argument as in  \cite[Lemma 2.27]{Toda1} shows that an object $E \in \Ac^p$ is $Z$-semistable if and only if it lies in $\Ac_i^p$ for $i=1$ or $\tfrac{1}{2}$ and $E$ is $Z$-semistable in $\Ac_i^p$.

First, suppose $E \in \coh^{\leq 2}(X)$ is a 2-dimensional $\wh{\mu}_{\omega, B}$-stable with $\wh{\mu}_{\omega, B}(E)>0$. Then $E$ must be pure 2-dimensional and so $E[1] \in \Ac^p_1$.  Note that $\mathfrak{Im}\big(Z(E[1])\big)=m\omega \ch_2^B(E[1])<0$ for all $m>0$, and so $\phi_{Z'}(E[1]) \succ 1$.  Now consider any short exact sequence in $\Ac^p_1$
\[
0 \to A \to E[1] \to Q \to 0
\]
where $A, Q\neq 0$.  If $A$ is supported in dimension 0, then $\phi_Z (A) = 1 \prec \phi_Z (E[1])$.  If $A$ is supported in dimension 2, then so is $Q$, and the  $\wh{\mu}_{\omega, B}$-stability of $E$ ensures $\phi_Z (A) \prec \phi_Z (Q)$ since $\mathcal{H}^0(A)$ is supported in dimension 0.  Thus $E$ is a $Z$-semistable object in $\Ac_i^p$ and hence is a $Z$-semistable object in $\Ac^p$.  This means  $E[1] \in R(1,2]$ and hence $E \in R(0,1]$.

Next, suppose $E$ is a pure 1-dimensional sheaf on $X$.  Then $E \in \Ac^p_{1/2} \subset \Ac^p$. Since $\coh^{\leq 1}(X)$ is a Serre subcategory of $\Ac^p$, every $Z$-HN factor of $E$ lies in $\coh^{\leq 1}(X)$ and must be supported in dimension 1, giving us  $E \in R(0,1]$.

If $E$ is a 0-dimensional sheaf of length 1, then it is a $Z$-stable object of phase 1 in $\Ac^p$ and so $E \in R(0,1]$.  This finishes showing that $\Tc^p \subset R(0,1]$.

Next, suppose $E$ is a $\wh{\mu}_{\omega, B}$-stable sheaf in $\coh^{\leq 2}(X)$ with $\hat{\mu}_{\omega, B}(E) \leq 0$.  Then $E$ is pure 2-dimensional and $E[1]$ lies in $\Ac^p_1$. The same argument as above shows that $E[1]$ is $Z$-semistable in $\Ac_1^p$ and hence in $\Ac^p$.  Since $\hat{\mu}_{\omega, B}(E) \leq 0$, it follows that $\phi_Z (E[1]) \in (0,1]$ and so $E[1] \in R(0,1]$.

Lastly, suppose $E$ is a torsion-free sheaf on $X$.  In the abelian category $\Ac^p$, notice that $E[1]$ cannot have any subobject $A$ supported in dimension 2.  Therefore, the left-most $Z'$-HN factor of $E[1]$ in $\Ac^p$ cannot be supported in dimension 2.  As a result, all the $Z'$-HN factors of $E[1]$ are supported in dimensions 0, 1, or 3 and thus have $\phi_{Z}$ in $(0,1]$.  Hence $E[1] \in R(0,1]$.  As a result, $\Fc^p[1] \subset R(0,1]$ and the lemma is proved.
\end{proof}

\begin{remark}\label{rmk:sigmaprimePSaltform}
\begin{enumerate}
    \item Our proofs of Lemmas \ref{lem:AG53-43-1} and \ref{lem:AG53-43-2} do not depend on the specific values of $f_2$ and $g_2$ in the definitions of $Z, Z'$.  As a result, the heart $R(0,1]$ is independent of $f_2, g_2$.  Since $Z'$ differs from $Z$
 by a stretch along the vertical direction, we  have $R(0,1] = R'(0,1]$.
    \item The polynomial stability condition $(Z_{\infty \omega, B}, \Ac^p)$ considered in \cite[Section 6]{BMT1}, referred to as the \textit{polynomial stability condition at the large volume limit}, coincides with our polynomial stability condition $\sigma_{\infty, B}=(Z, \Ac^p)$ for $f=g=m^2$.  In their notation, they used $Q^p$ to denote the slicing of $(Z_{\infty \omega, B}, \Ac^p)$, defined $\mathcal{C}^p := Q^p(0,1]$, and  proved that $\mathcal{C}^p = \langle \Fc^p[1], \Tc^p\rangle$.  In light of Lemma \ref{lem:AG53-43-2}, we have
    \[
    \mathcal{C}^p = \langle \Fc^p[1], \Tc^p\rangle = R(0,1] = R'(0,1].
    \]
    \item By the previous remark, the polynomial stability condition $\sigma_{\infty, B}=(Z,\Ac^p)$ (respectively\ $\sigma'_{\infty, B}=(Z', \Ac^p)$) can now be equivalently defined as the pair $(Z,\mathcal{C}^p)$ (respectively\ $(Z', \mathcal{C}^p)$).
\end{enumerate} 
\end{remark}

In preparation of the next theorem, we make a comparison between the notation for Bridgeland stability conditions in \cite{JMM} and that in \cite{BMT1}.

Let $X$ be a smooth projective threefold of Picard rank 1 with $\mathrm{Pic}(X)=\langle  H \rangle$ for $H$ ample.  Bridgeland stability conditions in \cite[Section 2.2]{JMM} are written using  the central charge
\[
Z_{\beta, \alpha, s } = -\ch_3^\beta + \left( s+\tfrac{1}{6}\right) \alpha^2 \ch_1^\beta + i \left(\ch_2^\beta - \tfrac{\alpha^2}{2} \ch_0^\beta \right)
\]
where $\ch_i^\beta \in \mathbb{R}$ is defined by setting
\[
 \ch_i^\beta H^i =\ch_i^{\beta H}.
\]
A quick computation  shows 
\begin{equation}\label{eq:AG53-38-1}
H^3 Z_{\beta, \alpha, s} = -\ch_3^B  + \left( s+\tfrac{1}{6}\right) \alpha^2 \omega^2 \ch_1^B + i \left(\omega\ch_2^B - \tfrac{\alpha^2}{2} \omega^3\ch_0^B \right)
\end{equation}
if we write $\omega = H$ and $B = \beta H$.  The pair $\sigma_{\beta, \alpha, s} = (Z_{\beta, \alpha, s}, \Ac^{\beta, \alpha})$ is a Bridgeland stability condition whenever  $(\beta, \alpha) \in \mathbb{R} \times\mathbb{R}_{>0}$ and $s \in \mathbb{R}_{>0}$.

On the other hand, Bridgeland stability conditions in \cite[Section 3]{BMT1} are written using the  central charge
\[
 Z_{\omega, B} = -\ch_3^B + \tfrac{\omega^2}{2} \ch_1^B + i  \left(\omega \ch_2^B  - \tfrac{\omega^3}{6} \ch_0^B\right).
\]
Whenever $\omega, B$ are divisors on $X$ with $\omega$ ample, the pair $\sigma_{\omega, B} = (Z_{\omega, B}, \Ac_{\omega, B})$ is a Bridgeland stability condition.

We can check that when $\omega = H, B = \beta H$ and $\alpha = m/\sqrt{3}$, we have
\begin{equation}\label{eq:AG53-40-1}
\Ac^{\beta, \alpha } = \Ac_{m\omega, B}.
\end{equation}

Now recall the following lemmas of Bayer--Macrì--Toda on the heart $\mathcal{C}^p$.

\begin{lemma}\cite[Lemma 6.2.1]{BMT1}\label{lem:BMT1Lem621}
Let $X$ be a smooth projective threefold, and $\omega, B$ divisors on $X$ where $\omega$ is ample.  For any object $E \in \dbx$, if $E \in \Ac_{m\omega, B}$ for $m \gg 0$, then $E \in \mathcal{C}^p$.
\end{lemma}

\begin{lemma}\cite[Lemma 6.2.2]{BMT1}\label{lem:BMT1Lem622}
Let $X$ be a smooth projective threefold, and $\omega, B$ divisors on $X$ where $\omega$ is ample.  For any object $E \in \mathcal{C}^p$, we have  $E \in \Ac_{m\omega, B}$ for $m \gg 0$.
\end{lemma}

\begin{theorem}\label{thm:AG53-39-1}
Let $X$ be a smooth projective threefold with $\mathrm{Pic}(X) = \mathbb{Z}H$ for $H$ ample.  Let us write $\omega = H$ and $B = \beta H$ where $\beta \in \mathbb{R}$, and fix 
\begin{equation}\label{eq:AG53-41-1}
f = 2\left( s + \tfrac{1}{6}\right)m^2, \,\, g = 3m^2
\end{equation} 
in the definition of $Z$.   Then for any object $E \in \dbx$, we have 
\[
E \text{ is $\sigma_{\infty,B}$-(semi)stable} \text{ if and only if } E \text{ is vertical (semi)stable}.
\]
That is, vertical stability is a polynomial stability condition in the sense of Bayer \cite{BayerPBSC}.
\end{theorem}

\begin{proof}[Proof of Theorem \ref{thm:AG53-39-1}]
Since an object $E \in \dbx$ is $\sigma_{\infty, B}$-(semi)stable if and only if it is $\sigma'_{\infty, B}$-(semi)stable, we will work with $\sigma'_{\infty,B}$-stability below.  Our choice of $f, g$ as in \ref{eq:AG53-41-1} ensures that  
\begin{equation}\label{eq:cccompare1}
  Z' = H^3 Z_{\beta, m, s},
\end{equation}
i.e.\ $Z'$ coincides with a constant scalar multiple of $Z_{\beta, \alpha, s}$ under the change of variable $m=\alpha$.

Suppose $E \in \mathcal{C}^p$ is a $\sigma'_{\infty,B}$-semistable object.  By Lemma \ref{lem:BMT1Lem622} and \eqref{eq:AG53-40-1}, we have $E \in \Ac^{\beta, \alpha}$ for $\alpha \gg 0$.  Now take any morphism $j : F \to E$ that is injective in $\Ac^{\beta, \alpha}$ for all $\alpha \gg 0$, and let $N = \mathrm{cone}(j)$ in $\dbx$.  The exact triangle in $\dbx$
\[
F \overset{j}{\to} E \to N \to F[1] 
\]
then gives a short exact sequence in $\Ac^{\beta, \alpha}$
\begin{equation}\label{eq:AG53-37-1}
0 \to F \to E \to N \to 0
\end{equation}
for $\alpha \gg 0$.  By Lemma \ref{lem:BMT1Lem621}, it follows that \eqref{eq:AG53-37-1} is a short exact sequence in $\mathcal{C}^p$.  Then $Z'$-semistability of $E$ in $\mathcal{C}^p$ now gives $\phi_{Z'}(F) \preceq \phi_{Z'}(N)$.  Since $Z'=H^3Z_{\beta, m, s}$ from \eqref{eq:cccompare1}, it follows that $\lambda_{\beta, \alpha, s}(F) \leq \lambda_{\beta, \alpha_s}(N)$ for $\alpha \gg 0$, i.e.\ $E$ is vertical semistable.

For the converse, suppose $E \in \dbx$ is a vertical semistable object.  This means that $E \in \Ac^{\beta, \alpha}$ for $\alpha \gg 0$, and so $E \in \mathcal{C}^p$ by Lemma \ref{lem:BMT1Lem621}.  Now take any short exact sequence in $\mathcal{C}^p$
\[
0 \to F \overset{j}{\to} E \to N \to 0.
\]
By Lemma \ref{lem:BMT1Lem622}, the morphism $j$ is an injection in the heart $\Ac^{\beta, \alpha}$ for $\alpha \gg 0$.  The vertical semistability of $E$ gives $\lambda_{\beta, \alpha, s}(F) \leq \lambda_{\beta, \alpha, s}(N)$ for $\alpha \gg 0$.  By \eqref{eq:cccompare1}, this means that $\phi_{Z'}(F) \preceq \phi_{Z'}(N)$.  Thus $E$ is $\sigma'$-semistable.

It is easy to see that the above arguments can be modified to prove the corresponding statement for $\sigma'$-stability and vertical stability.
\end{proof}

Based on Theorem \ref{thm:AG53-39-1} and Remark \ref{rmk:sigmaprimePSaltform}(2), we now know that the class of polynomial stability conditions $\sigma_{\infty, B}$, in which $f, g \in \mathbb{R}[m]$ are parameters, includes both vertical stability and the large volume limit in \cite{BMT1} as examples.  We can therefore think of $\sigma_{\infty, B}$ as a generalization of the large volume limit, and that it includes vertical stability as a special case.


\subsection{PT-stability as a limit of large volume limit}

In this subsection, we show in that if we take the large volume limit $(Z_{\infty \omega, B}, \mathcal{A}^p)$ in Section \ref{sec:lambdastabnLVL} and let $B$  \textit{go to negative infinity} in a particular, linear manner, then we obtain a polynomial stability that is equivalent to PT-stability (Proposition \ref{prop:PTislimitstab}).  Specifically, we take $B=bm\omega$ where $m$ is the indeterminate in polynomial stability, $\omega$ is a fixed ample class, and $b$ is a constant in the interval $\left(-\tfrac{1}{\sqrt{3}}, 0\right)$.  This gives us a description of PT-stability different from that in \cite{BayerPBSC}.

This is also consistent with the picture in \cite[Theorem 4.7]{Toda1} where, for rank-one objects, the moduli space of stable objects with respect to the large volume limit (with the extra factor $\sqrt{\mathrm{td}_X}$ in the central charge) coincides with the moduli of PT stable objects.

We also show that if the constant $b$ lies in an interval other than  $\left(-\tfrac{1}{\sqrt{3}}, 0\right)$, then the central charge $Z_{\infty\omega, B}$ coincides with that of other standard polynomial stability conditions in \cite{BayerPBSC}, including DT stability, dual PT stability, and dual DT stability (Proposition \ref{prop:otherlimitsLVL}).

\begin{proposition}\label{prop:PTislimitstab}
Let $X$ be a smooth projective threefold and  $\omega$ an ample class on $X$.  For any  $b \in \left(-\tfrac{1}{\sqrt{3}},0\right)$, if we set $B=bm\omega$
in the polynomial central charge
\[
Z_{\infty\omega,B} = -\ch_3^B + m^2 \frac{\omega^2}{2}\ch_1^B + i\left(m\omega \ch_2^B - m^3\frac{\omega^3}{6}\ch_0^B\right),
\]
then  $(Z_{\infty \omega, B},\Ap)$ is a PT polynomial stability condition in the sense of Bayer \cite{BayerPBSC}.  Moreover, for any $a \in \mathbb{R}$ such that 
\[
-\tfrac{1}{2}(1-b^2)+ib,\,\, b+i \in \{re^{i\phi \pi} : r \in \mathbb{R}_{>0}, \phi \in (a,a+1]\},
\]
$(Z_{\infty \omega, B},\Ap)$ is a  polynomial stability condition with respect to the interval $(a,a+1]$.
\end{proposition}

\begin{proof}
For a real number $b$, setting  $B=bm\omega$  gives
\begin{align*}
  Z_{\infty \omega, B} = & m^3\omega^3 \left( \left(-\frac{1}{2}b + \frac{1}{6}b^3\right) + i\left( -\frac{1}{6}+\frac{1}{2}b^2\right)\right)\ch_0 \\ & + m^2\omega^2 \left(  \left(\frac{1}{2}-\frac{1}{2}b^2\right)+i\left(-b\right) \right)\ch_1 \\
  & + m\omega \left( b+i \right)\ch_2 + \left( -1\right)\ch_3.
\end{align*}
Let us write $\rho_i$ to denote the coefficient of the $m^i\omega^i\ch_{3-i}$ term above, i.e.\ 
\begin{align*}
 \rho_3&= \left(-\frac{1}{2}b + \frac{1}{6}b^3\right) + i\left( -\frac{1}{6}+\frac{1}{2}b^2\right) \\
 \rho_2 &= \left(\frac{1}{2}-\frac{1}{2}b^2\right)+i\left(-b\right)\\
 \rho_1 &= b+i \\
 \rho_0 &= -1.
\end{align*}
If we can ensure that the complex numbers $-\rho_2, \rho_0, -\rho_3, \rho_1$ all lie in some rotation of the upper-half complex plane and that
\begin{equation}\label{eq:AG52-19-1}
\phi (-\rho_2) > \phi (\rho_0) > \phi (-\rho_3) > \phi (\rho_1),
\end{equation}
then we can write
\[
Z_{\infty \omega, B}(E) = \sum_{i=0}^3 \int_X m^i \omega^i \rho_i \ch (E)
\]
where the $\rho_i \in \mathbb{C}$ would satisfy the configuration for a PT stability function, thus proving the claim.

Note that $\rho_0=-1$ lies on the negative real axis on the complex plane.  We have:
\begin{itemize}
\item $-\rho_2$ lies on the lower-half complex plane if and only if $b<0$.
\item $-\rho_3$ lies above the upper-half complex plane if and only if 
\[
  -\left(-\frac{1}{6}+\frac{1}{2}b^2\right) >0 \Leftrightarrow b \in \left(-\frac{1}{\sqrt{3}}, \frac{1}{\sqrt{3}}\right).
\]
\item $-\rho_3 \in \{\rho_1 re^{i\phi \pi} : r \in\mathbb{R}_{>0}, \phi \in (0,1)\}$ if and only if 
    \[
    \begin{vmatrix} b & 1 \\ \tfrac{1}{2}b-\tfrac{1}{6}b^3 & \tfrac{1}{6}-\tfrac{1}{2}b^2 \end{vmatrix} >0 \Leftrightarrow b^3+b<0 \Leftrightarrow b<0
    \]
\item $-\rho_2 \in \{\rho_1 re^{i\phi \pi} : r \in\mathbb{R}_{>0}, \phi \in (0,1)\}$ if and only if
    \[
      \begin{vmatrix} b & 1 \\ -\frac{1}{2}+\tfrac{1}{2}b^2 & b \end{vmatrix} > 0 \Leftrightarrow \tfrac{1}{2}(b^2+1)>0,
    \]
    which holds for all $b \in \mathbb{R}$.
\end{itemize}
Therefore, when $b \in  (-\tfrac{1}{\sqrt{3}}, 0)$, all the inequalities in \eqref{eq:AG52-19-1} hold, and all of $-\rho_2, \rho_0, -\rho_3, \rho_1$ lie in some rotation of the upper-half complex plane, making $Z_{\infty \omega, B}$ a PT polynomial stability function.  

In \cite[Section 6]{BayerPBSC}, Bayer already established that a PT stability function has the HN property on the heart $\Ap$.  The rest of the proposition then follows.
\end{proof}

Proposition \ref{prop:PTislimitstab} shows that if we deform the $B$-field in the polynomial central charge $Z_{\infty\omega, B}$ at a specific rate, we obtain the stability function of PT stability.  The next proposition shows that if we deform $B$ at different rates, we obtain the stability functions of other polynomial stability conditions.  See Section \ref{sec:polynomial} for the definitions of DT stability, PT stability, and the large volume limit, and  \cite[Section 3]{BayerPBSC} for the definitions of dual stability conditions.

\begin{proposition}\label{prop:otherlimitsLVL}
Let $X$ be a smooth projective threefold, and $\omega$ an ample class on $X$.  Suppose  $B=bm\omega$ for some constant $b \in \mathbb{R}$, and let  $\rho_i$ for $0\leq i \leq 3$ be as in the proof of Proposition \ref{prop:PTislimitstab}.  Then $Z_{\infty \omega, B}$, as a polynomial central charge in the indeterminate $m$, is of the following type  in the sense of Bayer up to a rotation of the complex plane:
\begin{itemize}
\item[(i)] $b \in \left(-\infty, -\tfrac{1}{\sqrt{3}}\right)$:  $Z_{\infty \omega, B}$ is a DT stability function.
\item[(ii)] $b \in \left( -\frac{1}{\sqrt{3}} , 0\right)$:  $Z_{\infty \omega, B}$ is a PT stability function.
\item[(iii)] $b=0$:  $Z_{\infty \omega, B}$ is the polynomial stability function at the large volume limit.
\item[(iv)] $b \in \left( 0, \frac{1}{\sqrt{3}}\right)$:  $Z_{\infty \omega, B}$ is the polynomial stability function for the dual stability to PT stability.
\item[(v)] $b \in \left( \frac{1}{\sqrt{3}}, \infty\right)$:  $Z_{\infty \omega, B}$ is the polynomial stability function for the dual stability to DT stability.
\end{itemize}
\end{proposition}

\begin{proof}
Recall that we defined
\begin{align}
 \rho_3&= \left(-\frac{1}{2}b + \frac{1}{6}b^3\right) + i\left( -\frac{1}{6}+\frac{1}{2}b^2\right) \notag\\
 \rho_2 &= \left(\frac{1}{2}-\frac{1}{2}b^2\right)+i\left(-b\right) \notag\\
 \rho_1 &= b+i \notag\\
 \rho_0 &= -1.  \label{eq:rhoidef}
\end{align}
so that we can write 
\[
Z_{\infty \omega, B}(E) = \sum_{i=0}^3 \int_X m^i \omega^i \rho_i \ch (E).
\]

(i) \textbf{The case $b \in \left(\infty, -\tfrac{1}{\sqrt{3}}\right)$}.  In this case, $\rho_1$ always lies in the 2nd quadrant of the complex plane.  Also, since 
\[
\begin{vmatrix}\rho_1 \\ -\rho_2 \end{vmatrix}=\begin{vmatrix} b & 1 \\ -\frac{1}{2}+\tfrac{1}{2}b^2 & b \end{vmatrix}  = \tfrac{1}{2}(b^2+1) >0 \text{ for all $b \in \mathbb{R}$},
\]
the angle from $\rho_1$ to $-\rho_2$, measured counterclockwise, lies in the range $(0,\pi)$.  Note that
\begin{align}
 \mathfrak{Re} (-\rho_3) &= \tfrac{1}{6}b(\sqrt{3}+b)(\sqrt{3}-b) \\
 \mathfrak{Im} (-\rho_3) &= \frac{1}{6}(1+\sqrt{3}b)(1-\sqrt{3}b) \label{eq:-rho3reim}
\end{align}
and so $\mathfrak{Im} (-\rho_3)<0$ for $b \in \left(\infty, -\tfrac{1}{\sqrt{3}}\right)$, i.e.\ $-\rho_3$ always lies in the lower-half complex plane in case (i).  Lastly, we have
\[
\begin{vmatrix} -\rho_3 \\ -\rho_2\end{vmatrix} = \begin{vmatrix} \tfrac{1}{2}b-\tfrac{1}{6}b^3 & \tfrac{1}{6}-\tfrac{1}{2}b^2 \\ 
-\frac{1}{2}+\tfrac{1}{2}b^2 & b \end{vmatrix}  = \tfrac{1}{12} ( b^4 + 2b^2 + 1) >0 \text{ for all $b \in \mathbb{R}$},
\]
which tells us the angle from $-\rho_3$ to $-\rho_2$, measured counter-clockwise, lies in the range $(0,\pi)$. Overall, we see that there is a half plane
\[
\{ re^{i\phi \pi} : r \in \mathbb{R}_{>0}, \phi \in (r,r+1]\}
\]
for some $r \in \mathbb{R}$ that contains $\rho_1, \rho_0, -\rho_3, -\rho_2$, and that in this half-plane, we have
\[
\phi (-\rho_2)>  \phi (-\rho_3) > \phi (\rho_0) > \phi (\rho_1),
\]
which is the configuration of $\rho_i$ for DT stability up to a rotation.

(ii) \textbf{The case $b \in \left( -\frac{1}{\sqrt{3}}, 0\right)$}. This was shown in Proposition \ref{prop:PTislimitstab}.

(iii) \textbf{The case $b=0$}. In this case, the vectors $\rho_0, -\rho_2$ lie on the negative real axis while $\rho_1, -\rho_3$ lie on the positive imaginary axis, which coincide with the configuration for $\rho_i$ for the large volume limit in  \cite[Section 6]{BMT1} or \cite{Toda1}.

(iv) \textbf{The case $b \in \left(0,\tfrac{1}{\sqrt{3}}\right)$}. In this case, 
\[
\mathfrak{Re} (-\rho_2) = \frac{1}{2}(b+1)(b-1)
\]
is negative while $\mathfrak{Im} (-\rho_2)=b$ is positive, and so $-\rho_2$ lies in the second quadrant in the complex plane.  On the other hand, $\rho_1$ lies in the first quadrant, and from our computations of $\mathfrak{Re} (-\rho_3), \mathfrak{Im} (-\rho_3)$ in Case (i), we know $-\rho_3$ lies in the first quadrant as well.  Since
\[
  \begin{vmatrix} -\rho_3 \\ \rho_1 \end{vmatrix} = \tfrac{1}{3}b(b^2+1)>0 \text{ for all $b \in \mathbb{R}$},
\]
the angle from $-\rho_3$ to $\rho_1$, measured counter-clockwise, lies in $(0,\pi)$.  Overall, all of $\rho_0, -\rho_2, \rho_1, -\rho_3$ lie on the upper-half complex plane in case (iv), and if we regard all of their phases as being in $(0, \pi)$ then we have
\[
  \phi (\rho_0) > \phi (-\rho_2) > \phi (\rho_1) > \phi (-\rho_3).
\]
This is the configuration of $\rho_i$ for \textit{$\sigma_5$-stability} in \cite[Section 2]{Lo4}, which is dual to PT stability up to a $\mathrm{GL}^+\!(2,\mathbb{R})$-action.

(v) \textbf{The case $b \in \left( \tfrac{1}{\sqrt{3}}, \infty\right)$}.  In this case, $-\rho_2$ and $\rho_1$ always lie on the upper-half plane.  Since $\begin{vmatrix} \rho_1 \\-\rho_2 \end{vmatrix}>0$ as in Case (i), however, we know that the angle from $\rho_1$ to $-\rho_2$ lies in $(0,\pi)$. 

On the other hand, from the computation in Case (iv), we know the angle from $\rho_1$ to $-\rho_3$ (measured counter-clockwise) lies in $(0,\pi)$; from the computation in Case (i), we know the angle from $-\rho_3$ to $-\rho_2$ (measured counter-clockwise) also lies in $(0,\pi)$.  Also, 
\[
\begin{vmatrix} \rho_3 \\ \rho_0 \end{vmatrix} = \begin{vmatrix} -\tfrac{1}{2}b + \tfrac{1}{6}b^3 & -\frac{1}{6} + \tfrac{1}{2}b^2 \\ -1 & 0 \end{vmatrix} = \tfrac{1}{6}(\sqrt{3}b+1)(\sqrt{3}b-1),
\]
which is positive in Case (v).

Putting all these together, we have that each of $\rho_i/\rho_{i+1}$ lies on the upper-half complex plane for $i=0, 1, 2$, while there is some half-plane
\[
\mathbb{H}_a = \{ re^{i\phi \pi} : r \in \mathbb{R}_{>0}, \phi \in (a,a+1]\} 
\]
such that $\rho_0, \rho_3$ lie in $\mathbb{H}_a$ while $\rho_1, \rho_2$ lie in $-\mathbb{H}_a$.  Overall, the $\rho_i$ satisfy the configuration for the polynomial stability that is dual to DT stability.
\end{proof}


\subsection{PT-stability and $\sigma_{\infty, B}$-stability}\label{sec:PTstabvsgenLVL}

In this subsection, we show that when the $B$-field is sufficiently negative, $\sigma_{\infty, B}$-stability coincides with PT-stability.  The idea is to prove an effective version of Proposition \ref{prop:PTislimitstab}, using the generalization of the large volume limit defined in Section \ref{sec:LVLgeneralization}.

As a result, when $3f_2>g_2$ in the definition of $\sigma_{\infty, B}$ (see Theorem \ref{thm:PTislimtstab-eff} below), we have 
\[
\text{PT- stability} \Leftrightarrow \text{$\sigma_{\infty, B}$-stability}.
\]
On the other hand, when $f=2(s+\tfrac{1}{6})m^2, g=3m^2$ in the definition of $\sigma_{\infty, B}$ (see Theorem \ref{thm:AG53-39-1}) where $s>0$ is a parameter that  appears in the definition of vertical stability, we have 
\[
\text{$\sigma_{\infty, B}$-stability for $B \ll 0$} \Leftrightarrow \text{vertical stability for $B\ll 0$}.
\]
Overall, whenever  $f=2(s+\tfrac{1}{6})m^2, g=3m^2$ and $s>\tfrac{1}{3}$, we have the equivalences
\[
\text{PT- stability}\,\, \Leftrightarrow \,\, \text{$\sigma_{\infty, B}$-stability for $B \ll 0$} \,\, \Leftrightarrow \,\, \text{vertical stability for $B\ll 0$}.
\]

We begin with a boundedness result that will be needed in the proof of the main theorem in this subsection.

\begin{lemma}\label{lem:AG53-47-1}
Let $X$ be a smooth projective threefold and $\omega$ an ample class on $X$. Fix a Chern character $v$ on $X$ with $v_0 \neq 0$.  Then there exist constants $b_{-1}, b_0 \in \mathbb{R}$ such that, for every $\sigma_{PT}$-semistable object $E$ in $\Ac^p$ with $\ch(E)=v$, we have 
\[
  \wh{\mu}_{\mathrm{max}}( \mathcal{H}^{-1}(E)^{\ast\ast}/\mathcal{H}^{-1}(E)) < b_{-1}
\]
and $\mathrm{length}(\mathcal{H}^0(E))<b_0$.
\end{lemma}

\begin{proof}
For any $\sigma_{PT}$-semistable object $E \in \Ac^p$ of nonzero rank, we know $\mathcal{H}^{-1}(E)$ is $\mu_\omega$-semistable and $\mathcal{H}^0(E)$ is supported in dimension 0 by \cite[Lemma 3.3]{Lo1}.  Hence
\[
\ch_3 (\mathcal{H}^{-1}(E)) = -\ch_3(E) + \ch_3(\mathcal{H}^0(E)) \geq -\ch_3(E),
\]
i.e.\ $\ch_3(\mathcal{H}^{-1}(E))$ is bounded from below.  Thus the set
\[
\{ \mathcal{H}^{-1}(E)  : E \in \Ac^p \text{ is $\sigma_{PT}$-semistable}, \ch(E) = v \}
\]
is bounded by \cite[Theorem 4.8]{Maru1}.  The lemma then follows from Proposition \ref{prop:AG53-46-1}.
\end{proof}

Now we prove the equivalence between PT-stability and $\sigma_{\infty,B}$-stability for sufficiently negative $B$.

\begin{theorem}\label{thm:PTislimtstab-eff}
  Let $X$ be a smooth projective threefold,  $\omega$ an ample $\mathbb{Q}$-divisor  on $X$, and  $b \in  \left( -\tfrac{1}{\sqrt{3}},0\right)$ a constant.  Let $\sigma_{PT}$ denote the polynomial stability condition $(Z_{m\omega, bm\omega}, \Ap)$ in the indeterminate $m$ as in Proposition \ref{prop:PTislimitstab}.  For any $n \in \mathbb{R}_{>0}$, let $\sigma_{\infty,n}$ denote the polynomial stability condition $\sigma_{\infty, B}$ with $B=nb\omega$.  Suppose $3f_2-g_2>0$ in the definition of the central charge $Z$ of $\sigma_{\infty,n}$. Fix a Chern character $v$ on $X$.
  Then for every nonzero object $E \in \dbx$ with $\ch(E)=v$, the following are equivalent:
  \begin{itemize}
      \item[(i)] $E$ is $\sigma_{PT}$-semistable (respectively\ $\sigma_{PT}$-stable).
      \item[(ii)] There exists some $n_0>0$  depending only on $v$ such that, for all  $n>n_0$, $E$ is  $\sigma_{\infty,n}$-semistable (respectively\ $\sigma_{\infty,n}$-stable).
  \end{itemize}
\end{theorem}

In the proof below, whenever we have a 2 by 2 matrix $A$ over $\mathbb{R}$ and a complex number $z=x+iy$ (where $x, y \in \mathbb{R}$), we will write $Az$ to denote the complex number whose real and imaginary parts form the column vector $A\left(\begin{smallmatrix} x \\ y \end{smallmatrix}\right)$.  (That is, we identify $\mathbb{C}$ with $\mathbb{R}^2$ via $z=x+iy \mapsto \left(\begin{smallmatrix} x \\ y \end{smallmatrix}\right)$.)

\begin{proof}
We will divide the proof into two parts, according to whether  $\ch_0$  is nonzero.  We will write $\phi_{PT}, \phi_{\infty,n}$ to denote the phase functions associated to $\sigma_{PT}, \sigma_{\infty,n}$, respectively.

\smallskip

\textbf{Part I}: $\ch_0(E)\neq 0$.

\textbf{(i) $\Rightarrow$ (ii):}  Suppose $E \in \Ac^p$ has Chern character $v$.  In order to prove that $E$ is $\sigma_{\infty,n}$-semistable, let us consider any $\Ac^p$-short exact sequence of the form
\begin{equation}\label{eq:AG52-26-1}
0 \to M \overset{k}{\to} E\to N \to 0
\end{equation}
where $M, N \neq 0$.  The $\sigma_{PT}$-semistability of $E$ implies that $M$ must be supported in dimension 1 or 3.  

\textbf{Case 1.} Suppose $\dim M=1$.  Then $M=\mathcal{H}^0(M) \in \coh^{\leq 1}(X)$.  Consider the canonical short exact sequence in $\Ac^p$
\begin{equation}\label{eq:AG52-26-2}
0 \to \mathcal{H}^{-1}(E)[1] \overset{i}{\to} E \overset{j}{\to} \mathcal{H}^0(E) \to 0.
\end{equation}

\textbf{Case 1(a).} Suppose $jk=0$.  Then $k$ factors through $i$, i.e.\ we have a commutative diagram 
\[
\xymatrix{
  \mathcal{H}^{-1}(E)[1] \ar[r]^(.6)i & E \\
  M \ar[u]^{\ol{k}} \ar[ur]_k
}
\]
for some $\ol{k}$.  Since $i, k$ are both injections in $\Ap$, it follows that $\ol{k}$ is also an injection in $\Ap$.  Applying $\Hom (M,-)$ to the $\Ap$-short exact sequence
\begin{equation}\label{eq:AG52-26-3}
0 \to \mathcal{H}^{-1}(E)^{\ast\ast}/\mathcal{H}^{-1}(E) \overset{\alpha}{\to} \mathcal{H}^{-1}(E)[1] \to \mathcal{H}^{-1}(E)^{\ast\ast}[1] \to 0
\end{equation}
and noting that $\Hom (M, \mathcal{H}^{-1}(E)^{\ast\ast}[1])=0$ by  \cite[Lemma 4.20]{CL}, we have the isomorphism
\[
\ol{\alpha}: \Hom (M, \mathcal{H}^{-1}(E)^{\ast\ast}/\mathcal{H}^{-1}(E)) \overset{\thicksim}{\to} \Hom (M, \mathcal{H}^{-1}(E)[1])
\]
induced by $\alpha$.  This means that $\ol{k}=\alpha{\wt{k}}$ for some $\wt{k} \in  \Hom (M, \mathcal{H}^{-1}(E)^{\ast\ast}/\mathcal{H}^{-1}(E))$.  Since $\ol{k}$ and $\alpha$ are both injections in $\Ap$, it follows that $\wt{k}$ is also an injection in $\Ap$;  since $\coh^{\leq 1}(X)$ is a Serre subcategory of $\Ap$, we can identify $M$ with a subsheaf of $\mathcal{H}^{-1}(E)^{\ast\ast}/\mathcal{H}^{-1}(E)$.  Since $M\neq 0$, we have $\mathcal{H}^{-1}(E)^{\ast\ast}/\mathcal{H}^{-1}(E)\neq 0$.  Note that  $\mathcal{H}^{-1}(E)^{\ast\ast}/\mathcal{H}^{-1}(E)$ is a pure 1-dimensional sheaf by Lemma \ref{lem:tech-2}, and so $M$ is also a pure 1-dimensional sheaf.

Now set 
\[
\mu_{32,B} = \frac{\ch_3^B}{\omega \ch_2^B}
\]
with $B=bn\omega$, which is a slope function with the HN property on $\coh^{\leq 1}(X)$.  
Let $G$ denote the left-most $\mu_{32,B}$-HN factor of $\mathcal{H}^{-1}(E)^{\ast\ast}/\mathcal{H}^{-1}(E)$.  Then we have $\mu_{32,B}(M) \leq \mu_{32,B}(G)$, which is equivalent to $\phi_{\infty, n}(M) \preceq \phi_{\infty,n}(G)$ since both $M, G$ are supported in dimension 1.  Therefore, to prove that $M$ does not destabilise $E$, it suffices to show $\phi_{\infty,n}(G) \preceq \phi_{\infty,n}(E)$.

Consider the truncation 
\[
  Z_{01} = f \tfrac{\omega^2}{2}\ch_1^B + i(-gm) \tfrac{\omega^3}{6}\ch_0^B
\]
of the central charge $Z$ that corresponds to the two highest-degree terms in $Z$.   For any object $A \in \Ap$ supported in dimension 3, let us define a polynomial  phase function $\phi_{01}(A)$ of $Z_{01}(A)$ via
\[
  Z_{01}(A)(m) \in \{ re^{i \pi \phi_{01} (A)(m)} : r \in \mathbb{R}_{>0}, \phi_{01} (A)(m) \in (\tfrac{1}{4}, \tfrac{5}{4} )\} \text{\qquad for $m \gg 0$}.
\]
We will now prove  $\phi_{\infty,n}(G) \prec \phi_{\infty,n}(E)$ by proving  $\phi_{\infty,n}(G) \prec \phi_{01}(E)$.

Observe that 
\begin{align*}
Z(G) &= -(\ch_3(G)-B\ch_2(G))+im\omega \ch_2(G) \\
&= -\ch_3(G)+nb\omega \ch_2(G) + im\omega \ch_2(G)
\end{align*}
while
\begin{align*}
  Z_{01}(E) &= f \tfrac{\omega^2}{2}(\ch_1(E)-B\ch_0(E)) + i(-gm) \tfrac{\omega^3}{6}\ch_0(E) \\
  &= m^2 \left( f_2\tfrac{\omega^2}{2}\ch_1(E) - f_2bn\tfrac{\omega^3}{2}\ch_0(E) + i(-g_2)m\tfrac{\omega^3}{6} \ch_0(E)\right).
\end{align*}
If we define the matrix
\[
 J = \begin{pmatrix} 1 & -nb/m \\ 0 & 1/m\end{pmatrix} = \begin{pmatrix} 1 & -nb \\ 0 & 1  \end{pmatrix}\begin{pmatrix}
    1 & 0 \\ 0 & 1/m\end{pmatrix}
\]
which lies in $\mathrm{GL}^+\!(2,\mathbb{R})$ for any $m \in \mathbb{R}_{>0}$, then 
\begin{align*}
  JZ(G) &= -\ch_3(G) + i\omega \ch_2(G) \\
  \tfrac{1}{m^2}JZ_{01}(E) &=  f_2\tfrac{\omega^2}{2}\ch_1(E) + n \tfrac{\omega^3}{6}b (-3f_2+g_2)\ch_0(E) + i(-g_2)\tfrac{\omega^3}{6}\ch_0(E).
\end{align*}
Our assumption $3f_2-g_2>0$ now implies that the term $  \tfrac{\omega^3}{6}b (-3f_2+g_2)\ch_0(E)$ is negative (note that $\ch_0(E)<0$).  Therefore, if we consider both $JZ(G)$ and $\tfrac{1}{m^2}JZ_{01}(E)$ as having phases in $(0,1)$, then there exists $n_1>0$ such that, for all $n>n_1$, the phase of $JZ(G)$ is strictly less than the phase of  $\tfrac{1}{m^2}JZ_{01}(E)$, in which case we have $\phi_{\infty,n}(G)\prec \phi_{01}(E)$.  Moreover, the choice of $n_1$ only depends on $\ch(E)$ by Lemma \ref{lem:AG53-47-1}, and we are done in this case.

\textbf{Case 1(b).} Suppose $jk \neq 0$.  Since $M$ is a subobject of $E$ in the abelian category $\Ac^p$, we can write $M$ as an extension in $\Ac^p$ 
\[
0\to M' \to M \to M'' \to 0
\]
 where $M'$ is a subobject of $\mathcal{H}^{-1}(E)[1]$ and $M''$ a subobject of $\mathcal{H}^0(E)$ in $\Ac^p$.  Since we are assuming $\dimension M = 1$, and $\coh^{\leq 1}(X)$ is a Serre subcategory of $\Ac^p$, both $M'$ and $M''$ must lie in $\coh^{\leq 1}(X)$.  Note that $M'$ is a subobject of $E$ in $\Ac^p$, and the vanishing $\Hom (\coh^{\leq 0}(X), E)=0$ from the $\sigma_{PT}$-semistability of $E$ implies that $M'$ is a pure 1-dimensional sheaf.

By comparing the phases of $JZ(M')$ and $JZ(G)$, we see that there exists an $n_0>0$ depending only on $\ch(E)$ such that, for all $n>n_0$, we have $\phi_{\infty,n}(M')\prec \phi_{\infty,n}(G)$.  Since $\omega$ is a $\mathbb{Q}$-divisor, this means that $JZ(M')=-\ch_3(M')+i\omega \ch_2(M')$ lies in a discrete lattice in $\mathbb{C}$ with phase bounded from above by the phase of $JZ(G)$, which is strictly less than $1$ because $\dimension G=1$.  Then, since $M''$ is a subsheaf of the 0-dimensional sheaf $\mathcal{H}^0(E)$, there exists some $\epsilon >0$, depending only on $\wh{\mu}(G)=\ch_3(G)/\omega \ch_2(G)$  and the length of $\mathcal{H}^0(E)$, and hence depending only on $\ch(E)$ overall by Lemma \ref{lem:AG53-47-1}, such that (here we write $\phi(z)$ to denote the phase of a complex number $z$):
\begin{itemize}
    \item the phase of $JZ(M) = JZ(M') + JZ(M'')$ is bounded from above by  $\phi (JZ(G))+\epsilon$; 
    \item $\phi (JZ(G))+\epsilon < 1$.
\end{itemize}
Since  $\phi (JZ(G))+\epsilon$ depends only on $\ch(E)$, the same argument as in the last part of Case 1(a) means that there exists $n_0>0$ depending only on $\ch(E)$ such that, whenever $n>n_0$, we have
\[
\phi (JZ(G))+\epsilon < \phi (\tfrac{1}{m^2}JZ_{01}(E)).
\]
As a result, for $n>n_0$, we have $\phi_{\infty,n}(M) \prec \phi_{\infty,n}(E)$, finishing  Case 1(b).

\textbf{Case 2.} Suppose $\dim M=3$.  If $\ch_0(\mathcal{H}^{-1}(M))=\ch_0(\mathcal{H}^{-1}(E))$, then $\ch_0(\mathcal{H}^{-1}(N))=0$ and it follows that $\mathrm{dim}\, N=2$ or $0$.  In either case, we have $\phi_{\infty,n}(E) \to \tfrac{1}{2}$ while $\phi_{\infty,n}(N) \to 1$ for any fixed $n>0$, meaning $E$ is not destabilised with respect to $\sigma_{\infty,n}$.  Also, if $\ch_0(\mathcal{H}^{-1}(M))< \ch_0(\mathcal{H}^{-1}(E))$ and $\mu_{\omega,B}(\mathcal{H}^{-1}(M)) < \mu_{\omega,B}(\mathcal{H}^{-1}(E))$, then we have $\phi_{\infty,n}(M) \prec \phi_{\infty,n}(E)$ as well.  Therefore, let us assume $0<\ch_0(\mathcal{H}^{-1}(M)) < \ch_0(\mathcal{H}^{-1}(E))$ and  $\mu_{\omega,B}(\mathcal{H}^{-1}(M)) = \mu_{\omega,B}(\mathcal{H}^{-1}(E))$ from now on.  With the $\sigma_{PT}$-semistability of $E$, this implies that  $\frac{\omega\ch_2(M)}{\ch_0(M)} \leq \frac{\omega\ch_2(E)}{\ch_0(E)}$.

Since
\[
\frac{\omega\ch_2^B}{\ch_0} = \frac{\omega\ch_2}{\ch_0} - bn\mu_\omega + \tfrac{1}{2}b^2n^2\omega^2,
\]
it follows that the inequality $\frac{\omega\ch_2^B(M)}{\ch_0(M)} \leq \frac{\omega\ch_2^B(E)}{\ch_0(E)}$ (which does not depend on $n$) holds for any $n>0$.  Note that  $\omega^2\ch_1^B(E)=\omega^2(\ch_1(E)-bn\omega\ch_0(E))$, and so there exists $n_0>0$ depending only on $\ch(E)$ (and $\omega$) such that, whenever $n>n_0$, we have $\omega^2\ch_1^B(E)<0$.  As a result, if $\frac{\omega\ch_2^B(M)}{\ch_0(M)} < \frac{\omega\ch_2^B(E)}{\ch_0(E)}$, then for any $n>n_0$ we have $\phi_{\infty,n} (M) \prec \phi_{\infty,n}(E)$.  So let us assume  $\frac{\omega\ch_2^B(M)}{\ch_0(M)} = \frac{\omega\ch_2^B(E)}{\ch_0(E)}$ from now on, which is equivalent to $\frac{\omega\ch_2(M)}{\ch_0(M)} = \frac{\omega\ch_2(E)}{\ch_0(E)}$.  

The $\sigma_{PT}$-semistability of $E$ now implies  $-\frac{\ch_3(M)}{\ch_0(M)} \leq -\frac{\ch_3(E)}{\ch_0(E)}$.  Since
\[
\frac{\ch_3^B}{\ch_0} = \frac{\ch_3}{\ch_0} - bn\frac{\omega \ch_2}{\ch_0} + \frac{1}{2}b^2n^2\mu_\omega - \frac{1}{6}b^3n^3\omega^3,
\]
the same argument as in the previous paragraph shows that $-\frac{\ch_3^B(M)}{\ch_0(M)} \leq -\frac{\ch_3^B(E)}{\ch_0(E)}$, and so $\phi_{\infty,n}(M) \preceq \phi_{\infty,n}(E)$ for any $n>n_0$.

\medskip

\textbf{(ii) $\Rightarrow$ (i):} Suppose statement (ii) holds.  Let us  take an $\Ap$-short exact sequence 
\[
0 \to M \to E \to N \to 0
\]
where $M, N \neq 0$.  
In trying to prove $E$ is $\sigma_{PT}$-semistable, it suffices to consider the cases where $\dim M=0, 2$ or $3$.  If $\dim M=0$ or $2$, then $\phi_{\infty, n}(M) \to 1$ while $\phi_{\infty, n}(E) \to \tfrac{1}{2}$, contradicting the $\sigma_{\infty,n}$-semistability of $E$ for $n \gg 0$.  So let us assume $\dim M=3$ from here on.  We now proceed in a manner similar to the proof for the other implication.

Since $E$ is $\sigma_{\infty,n}$-semistable for $n \gg 0$, it follows that $\mathcal{H}^{-1}(E)$ is a torsion-free  $\mu_\omega$-semistable sheaf.  If $\mu_\omega (\mathcal{H}^{-1}(M)) < \mu_\omega (\mathcal{H}^{-1}(E))$, we would have $\phi_{PT}(M) \prec \phi_{PT}(E)$.  Therefore, let us assume $\mu_\omega (\mathcal{H}^{-1}(M)) = \mu_\omega (\mathcal{H}^{-1}(E))$ from now on. As before, we have $\omega^2\ch_1^B(E)<0$ for $n \gg 0$; together with the $\sigma_{\infty,n}$-semistability of $E$, it follows that  $\frac{\ch_2^B(M)}{\ch_0(M)} \leq \frac{\ch_2^B(E)}{\ch_0(E)}$, which in turn implies $\frac{\ch_2(M)}{\ch_0(M)} \leq \frac{\ch_2(E)}{\ch_0(E)}$.

If we have the strict inequality $\frac{\ch_2(M)}{\ch_0(M)} < \frac{\ch_2(E)}{\ch_0(E)}$, it would follow that $\phi_{PT}(M) \prec \phi_{PT}(E)$, so let us assume $\frac{\ch_2(M)}{\ch_0(M)} = \frac{\ch_2(E)}{\ch_0(E)}$.  The $\sigma_{\infty,n}$-semistability of $E$ for $n \gg 0$ then implies $\frac{\ch_3^B(M)}{\ch_0(M)} \leq \frac{\ch_3^B(E)}{\ch_0(E)}$ for $n \gg 0$, which is equivalent to $\frac{\ch_3(M)}{\ch_0(M)} \leq \frac{\ch_3(E)}{\ch_0(E)}$.  We now have  $\phi_{PT}(M) \preceq \phi_{PT}(E)$, proving that $E$ is $\sigma_{PT}$-semistable.  This also completes the proof of Part I.

\bigskip

\textbf{Part II}: $\ch_0(E)=0$.  

If $\dim E\leq 1$, then every $\Ap$-short exact sequence of the form $0 \to M \to E \to N \to 0$ is also a short exact sequence in $\coh^{\leq 1}(X)$.  Also, for  objects in $\coh^{\leq 1}(X)$ we have $\frac{\ch_3^B}{\omega\ch_2^B} = \frac{\ch_3}{\omega\ch_2} - bn$.  It is then clear that $E$ is $\sigma_{PT}$-semistable if and only if it is $\sigma_{\infty,n}$-semistable for all $n>0$.  From now on, let us assume $\dim E=2$.

For any integers $3 \geq i > j \geq 1$, we will write
\[
\mu_{ij,\omega,B} = \frac{\omega^{3-i}\ch_i^B}{\omega^{3-j}\ch_j^B} \text{\quad and \quad} \mu_{ij,\omega} = \frac{\omega^{3-i}\ch_i}{\omega^{3-j}\ch_j}. 
\]
Consider any $\Ap$-short exact sequence $0\to M \to E \to N \to 0$ with $M, N \neq 0$. 

\textbf{(i) $\Rightarrow$ (ii):} Suppose $E$ is $\sigma_{PT}$-semistable. 

\textbf{Case 1.} Suppose $\dim M=2$.  Then we have $\mu_{21,\omega}(M) \leq \mu_{21,\omega}(E)$ from the $\sigma_{PT}$-semistability of $E$.  Since $\mu_{21,\omega,B}=\mu_{21,\omega}-bn$ for objects supported in dimension 2,   it follows that $\mu_{21,\omega,B}(M) \leq \mu_{21,\omega,B}(E)$; if strict inequality holds here, then we obtain $\phi_{\infty,n}(M) \prec \phi_{\infty,n}(E)$ for all $n>0$.  So let us assume $\mu_{21,\omega,B}(M) = \mu_{21,\omega,B}(E)$, or equivalently $\mu_{21,\omega}(M) = \mu_{21,\omega}(E)$.

Now the $\sigma_{PT}$-semistability of $E$ gives $\mu_{31,\omega}(M) \leq \mu_{31,\omega}(E)$.  Since 
\[
\mu_{31,\omega,B} = \mu_{31,\omega} - bn\mu_{21,\omega}+\frac{1}{2}b^2n^2
\]
for objects supported in dimension 2, we see that $\mu_{31,\omega,B}(M) \leq \mu_{31,\omega,B}(E)$.   On the other hand, observe that since $\dimension E=2$, we have
\[
\omega \ch_2^B (E) = \omega \ch_2(E) - nb\omega^2 \ch_1(E).
\]
Since $\omega^2\ch_1(E)<0$, it follows that there exists some $n_0>0$ depending only on $\ch(E)$ such that, for all $n>n_0$, we have $\omega\ch_2^B(E)<0$.  The inequalities $\mu_{31,\omega,B}(M) \leq \mu_{31,\omega,B}(E)$ and $\omega \ch_2^B(E)<0$ together now imply $\phi_{\infty,n}(M) \preceq \phi_{\infty,n}(E)$. 

\textbf{Case 2.} Suppose $\dim M = 1$.  In this case, we have $\phi_{\infty,n}(M)\to \tfrac{1}{2}$ while $\phi_{\infty,n}(E) \to 1$ for all $n>0$, and so $\phi_{\infty,n}(M) \prec \phi_{\infty,n}(E)$ for all $n>0$.

\textbf{Case 3.} Suppose $\dim M =0$.  In this case, $\phi_{\infty,n}(M)=1$ for all $n>0$.  Meanwhile, we have 
\[
\mathfrak{Im}\big(Z(E)\big) = m(  \omega \ch_2(E)- nb\omega^2 \ch_1(E)  ).
\]
Since $\omega^2 \ch_1(E)<0$, there exists $n_0>0$ depending only on $\ch (E)$ such that, for all $n>n_0$, we have $\mathfrak{Im}\big(Z(E)\big)<0$ for all $m>0$.  Then for all $n>n_0$ we have $\phi_{\infty,n}(M) \prec \phi_{\infty,n}(E)$.

\medskip

\textbf{(ii) $\Rightarrow$ (i):} Suppose there exists some $n_0>0$ such that $E$ is $\sigma_{\infty,n}$-semistable for all $n>n_0$.

When $\dim M=2$, we can show that $\phi_{PT}(M) \preceq \phi_{PT}(E)$ by considering $\mu_{21,\omega,B}, \mu_{21,\omega}$ (and taking $n \gg 0$ if necessary so that $\omega \ch_2^B(E)<0$) and then $\mu_{31,\omega,B}, \mu_{31,\omega}$ of $M$ and $E$ as in Case 1 above.

When $\dim M=1$ or $0$, we have $\phi_{PT}(M) \prec \phi_{PT}(E)$ from the definition of $\sigma_{PT}$-stability.

The arguments above can be modified easily to show the equivalence of \textit{$\sigma_{PT}$-stable} and \textit{$\sigma_{\infty.n}$-stable for $n>n_0$}.
\end{proof}

\subsection{Stable triples as PT stable objects}

Let $X$ be a smooth projective threefold.  Given a  triple $(B, P, \eta)$ where $B, P$ are coherent sheaves and $\eta \in \Ext^1(P,B[1])$, we write $A_\eta$ to denote the 2-term complex sitting at degrees $-1, 0$ corresponding to the extension class $\eta$.  In particular, when $B \in \coh^{\geq 2}(X)$ and $P \in \coh^{\leq 1}(X)$ we have $A_\eta \in \Ac^p$.  This is consistent with our notation in Definition \ref{defn:stable-triple}, where $B$ is specifically a 2-Gieseker semistable sheaf and $P$ a 0-dimensional sheaf.

We now prove that, under the coprime assumption on degree and rank, the following are equivalent:
\[
  A_\eta \text{ is a PT-stable object} \,\, \Leftrightarrow \,\, (B,P,\eta) \text{ is a stable triple}.
\]

\begin{proposition}\label{prop:PTequivSTnSST}
Let $X$ be a smooth projective threefold, and $\omega$ an ample class on $X$.  Fix $v_i \in A^i(X)$ for $i=0, 1$ such that $v_0, \omega^2 v_1$ are coprime.  
\begin{itemize}
\item[(a)] Suppose $E$ is a PT-stable object in $\Ac^p$ with $(\ch_0(E), \ch_1(E))=(v_0, v_1)$.  Then $E\cong A_\eta$ in $\dbx$ for some  stable triple $(B, P, \eta)$.
\item[(b)] Suppose $(B, P, \eta)$ is a  stable triple such that $(\ch_0(B), \ch_1(B))=(v_0, v_1)$.  Then $A_\eta$ is a PT-stable object in $\Ac^p$.
\end{itemize}
Statements (a) and (b) still hold if we replace each occurrence of stable triple with \textit{semistable triple}, \textit{$\mu_\omega$-stable triple}, or \textit{$\mu_\omega$-semistable triple}.
\end{proposition}

\begin{proof}
Given Remark \ref{rem:fourtripesequiv}, it suffices to prove the version of statements (a) and (b) for $\mu_\omega$-stable triples.

(a) Suppose $E$ is a PT-stable object as described.  Then $\mathcal{H}^{-1}(E)^{\ast\ast}/\mathcal{H}^{-1}(E)$ is pure 1-dimensional by Lemma \ref{lem:tech-2}. From the characterisation of PT-stable objects in Proposition \ref{prop:tech-0} and Lemma \ref{eq:STcond3equivs}, it follows easily that $E \cong A_\eta$ for the stable triple $(B,P,\eta)$ where $B=\mathcal{H}^{-1}(E),\ P=\mathcal{H}^0(E)$, and $\eta$ corresponds to the extension class from the canonical exact triangle
\[
\mathcal{H}^{-1}(E) \to E \to \mathcal{H}^0(E) \to \mathcal{H}^{-1}(E)[2].
\]

(b) Suppose $(B,P,\eta)$ is a stable triple as described. Then $B$ is $\mu_\omega$-stable, and $A_\eta$ is a PT-stable object by Lemma \ref{eq:STcond3equivs} and the characterisation of PT-stable objects in Proposition \ref{prop:tech-0}.
\end{proof}

We now  recover Lemma \ref{lem:rk1ST-PTpairsequiv}:

\begin{corollary}\label{cor:AG53-92-null-1}
Giving a rank-one stable triple with $\ch_1=0$  is equivalent to giving a PT stable pair in $\dbx$.
\end{corollary}

\begin{proof}
From Proposition \ref{prop:PTequivSTnSST}, we know that giving a rank-one stable triple is equivalent to giving a rank-one PT stable object.  Since a rank-one PT stable object with $\ch_1=0$ is equivalent to a PT stable pair by \cite[Proposition 6.1.1]{BayerPBSC}, the claim follows.
\end{proof}

We  can also recover Example \ref{eg:reflx-sh-gives-ST}:

\begin{lemma}\label{lem:AG53-92-null-2}
Let $F$ be a $\mu$-stable reflexive sheaf.  Suppose $E := F^\vee [1]$, $B=\mathcal{H}^{-1}(E)$, $P = \mathcal{H}^0(E)$, and $\eta$ denotes the element in $\Ext^1 (P, B[1])$ corresponding to the canonical exact triangle
\[
B[1] \to E \to P,
\]
then $(B, P, \eta)$ is a stable triple.
\end{lemma}

\begin{proof}
Observe that $B=F^\ast$ is $\mu$-stable, hence 2-Gieseker stable. Also, taking derived dual gives
\[
\Hom (\coh^{\leq 0}(X), E) \cong \Hom (F, (\coh^{\leq 0}(X))^\vee [1])=0
\]
since $(\coh^{\leq 0}(X))^\vee [1] \in \coh (X)[-2]$.  Applying the proof of Lemma \ref{lem:tech-2} to $E$ then gives that  $B^{\ast\ast}/B$ is pure 1-dimensional.  Finally, the last paragraph in the  proof of Proposition \ref{prop:PTequivSTnSST}(a) shows that condition (iii-b) in Lemma \ref{eq:STcond3equivs} holds.  Since condition (iii-a) is equivalent to condition (iii-b) by Lemma \ref{eq:STcond3equivs}, it follows that $(B, P, \eta)$ is a stable triple in the sense of Definition \ref{defn:stable-triple}.
\end{proof}

In Proposition \ref{prop:PTequivSTnSST}, we saw that under the coprime assumption on rank and degree, PT-stable objects correspond to stable triples.  Even without the coprime assumption, we have the following comparisons.

\begin{proposition}\label{pro:AG53-92-1}
Suppose $s>\tfrac{1}{3}$ and $(B,P,\eta)$ is a stable triple.  Then there exists $\beta_0<0$ such that, for all $\overline{\beta} <\beta_0$, the object $A_\eta$ is vertical stable.
\end{proposition}

\begin{proof}
By Theorem \ref{thm:AG53-39-1} and Theorem \ref{thm:PTislimtstab-eff} (also see the discussion at the start of Section \ref{sec:PTstabvsgenLVL}), it suffices to show that $A_\eta$ is a PT-stable object.  Suppose $E_0$ is a nonzero proper subobject of $A_\eta$ in $\Ac^p$.  We want to show that $\phi (E_0) \prec \phi (A_\eta)$ where $\phi$ denotes the phase function with respect to PT-stability.

If $E_0$ is supported in dimension 1, then $\phi (E_0) \prec \phi (A_\eta)$ since $\phi (\rho_1) < \phi (-\rho_3)$ for PT-stability.  By Lemma \ref{eq:STcond3equivs}, we have $\Hom (\coh^{\leq 0}(X),A_\eta)=0$, and so $E_0$ cannot be 0-dimensional.  Since $\mathcal{H}^{-1}(E_0)$ is a subsheaf of $\mathcal{H}^{-1}(A_\eta)=B$, $\mathcal{H}^{-1}(E_0)$ (and hence $E_0$) cannot be 2-dimensional, either.   Therefore, it suffices to consider the case where $E_0$ is 3-dimensional, in which case $\mathcal{H}^{-1}(E_0)$ is a torsion-free sheaf.

Since $(B,P,\eta)$ is a stable triple, the sheaf $B$ is 2-Gieseker stable.  This implies $p_{3,1}(\mathcal{H}^{-1}(E_0)) \prec p_{3,1}(\mathcal{H}^{-1}(A_\eta))$ (where $p_{3,1}$ denotes the reduced Hilbert polynomial with constant term omitted), which in turn implies $\phi (E_0) \prec \phi (A_\eta)$ unless $\mathcal{H}^{-1}(E_0) = \mathcal{H}^{-1}(A_\eta)$.

Suppose $\mathcal{H}^{-1}(E_0)=\mathcal{H}^{-1}(A_\eta)$.  Then $\mathcal{H}^{-1}(A_\eta/E_0)$ is a subsheaf of a sheaf in $\coh^{\leq 1}(X)$ and so must be 0.  This forces $A_\eta/E_0$ to lie in $\coh^{\leq 1}(X)$.  On the other hand, $A_\eta/E_0$ is a quotient sheaf of $\mathcal{H}^0(A_\eta)=P$ which is 0-dimensional, and so $A_\eta/E_0$ itself is a 0-dimensional sheaf.  This gives $\phi (E_0) \prec \phi (A_\eta) \prec \phi (A_\eta/E_0)$ and we are done.
\end{proof}

\begin{proposition}\label{pro:AG53-92-2}
Suppose $s>\tfrac{1}{3}$, and and $\ch$ is a Chern character such that $\ch_0 \neq 0$.  Then there exists $\beta_0<0$ such that, for every vertical semistable object $A$ at $\beta$ for some $\beta < \beta_0$ and $\ch(A)=\ch$, we have $A\cong A_\eta$ for some semistable triple $(B,P,\eta)$.
\end{proposition}

\begin{proof}
As in the proof of Proposition \ref{pro:AG53-92-1}, by  Theorem \ref{thm:AG53-39-1} and Theorem \ref{thm:PTislimtstab-eff}, we know there exists $\beta_0<0$ such that, for every vertical semistable object $A$ at $\beta$ where $\beta< \beta_0$ and $\ch(A)=\ch$, the object $A$ is PT-semistable.  Then by Proposition \ref{prop:tech-0}, the sheaf $\mathcal{H}^{-1}(A)$ is torsion-free and semistable in $\coh_{3,1}(X)$, while $\mathcal{H}^0(A) \in\coh^{\leq 0}(X)$ and $\Hom (\coh^{\leq 0}(X),A)=0$.  Also, $\mathcal{H}^{-1}(A)^{\ast\ast}/\mathcal{H}^{-1}(A)$ is pure 1-dimensional by Lemma \ref{lem:tech-2}.  Now by Lemma \ref{eq:STcond3equivs}, the triple $(\mathcal{H}^{-1}(A),\mathcal{H}^0(A),\eta)$, where $\eta$ corresponds to the extension from the canonical exact triangle $\mathcal{H}^{-1}(A) \to A \to \mathcal{H}^0(A) \to \mathcal{H}^{-1}(A)[1]$, is a semistable triple.
\end{proof}


\section{Stability along the $\Theta$ curve} \label{sec:stab-theta}

Fixing a class $v\in K_{\rm num}(X)$, consider the following curve in the $(\beta,\alpha)$-plane
$$ \Theta_v := \{ (\beta,\alpha) ~|~ \nu_{\beta,\alpha}(v)=0 \} ;$$
this is a hyperbola with foci along the $\beta$-axis and symmetric around the vertical line $\beta=\mu(v)$.

Let $(\alpha_v,\beta_v)$ be the coordinates of the point where the largest $\nu$-wall for the class $v$ to the left of vertical line $\beta=\mu(v)$ intersects $\Theta$.

\begin{proposition}\label{prop:moduli in theta}
Take $(\obeta,\oalpha)\in\Theta_v$ with $\obeta<\beta_v$; assume that $v_0>0$. Then every $\lambda_{\obeta,\oalpha,s}$-semistable object $A$ with $\ch(A)=v$ is S-equivalent to an object of the form $B[1]\oplus P$, where $B$ is a 2-Gieseker semistable sheaf and $P$ is a 0-dimensional sheaf. In addition, an object $A$ with $\ch(A)=v$ is $\lambda_{\obeta,\oalpha,s}$-stable if and only if $A=B[1]$ for some 2-Gieseker stable sheaf $B$ with $\hd(B)=1$.
\end{proposition}
\begin{proof}
Given $(\obeta,\oalpha)\in\Theta_v$, Proposition 3.2 in \cite{JM19} says that an object $A\in D(X)$ is $\lambda_{\obeta,\oalpha,s}$-semistable if and only if $B:=\calh_\beta^{-1}(A)$ is $\nu_{\obeta,\oalpha}$-semistable and $P:=\calh_\beta^{0}(A)$ is a 0-dimensional sheaf. When, in addition, $\obeta<\beta_v$ (or, equivalently, $\oalpha>\alpha_v$), \cite[Proposition 5.2 (i)]{JM19} implies that $B$ is a 2-Gieseker semistable sheaf. It follows that $A\in D(X)$ is $\lambda_{\obeta,\oalpha,s}$-semistable if and only if it fits into a triangle of the form $B[1]\to A \to P$; since both $B[1]$ and $P$ belong to $\cala^{\obeta,\oalpha}$ (because $B\in\calf_{\obeta,\oalpha}$ and $P\in\calt_{\obeta,\oalpha}$), then we get an exact sequence $0\to B[1] \to A\to P\to 0$ in $\cala^{\obeta,\oalpha}$. But $\lambda_{\obeta,\oalpha,s}(B[1])=\lambda_{\obeta,\oalpha,s}(P)=+\infty$, so $A$ is S-equivalent to $[B[1]\oplus P]$, as desired.

Let $B$ is a 2-Gieseker stable sheaf; according to \cite[Proposition 5.2 (i)]{JM19}, $B$ belongs to $\coh^\obeta(X)$ and is $\nu_{\obeta,\oalpha}$-stable with $\nu_{\obeta,\oalpha}(B)=0$. Let $F\into B[1]\onto G$ be a destabilizing sequence in $\cala^{\obeta,\oalpha}$, so that $\lambda_{\obeta,\oalpha,s}(F)=+\infty$; taking cohomology with respect to $\coh^\obeta(X)$, we obtain the exact sequence in $\coh^\obeta(X)$
$$ 0 \to \calh^{-1}_\obeta(F) \to B \to \calh^{-1}_\obeta(G) \to \calh^{0}_\obeta(F) \to 0 . $$
Let $K$ be the cokernel of the first morphism in the sequence; since $K$ is a quotient of $B$, we have that $\nu_{\obeta,\oalpha}(K)>0$; by contrast, $\nu_{\obeta,\oalpha}(K)\le 0$ because $K$ is a sub-object of $\calh^{-1}_\obeta(G)\in\calf_{\obeta,\oalpha}$. We have therefore two possibilities:
\begin{itemize}
\item either $K=0$, in which case $\calh^{-1}_\obeta(F) \simeq B$, and $\calh^{-1}_\obeta(G) \simeq \calh^{0}_\obeta(F)$, which implies that both of the latter objects are trivial; it follows that $F=B[1]$.
\item or $K=B$ and $\calh^{-1}_\obeta(F)=0$ so $F=\calh^{0}_\obeta(F)$; but $\nu_{\obeta,\oalpha}(F)=0$ (since $\lambda_{\obeta,\oalpha,s}(F)=+\infty$), which leads to a contradiction because $\calh^{0}_\obeta(F)\in\calt_{\obeta,\oalpha}$.
\end{itemize}
We conclude that $F=B[1]$, thus $B[1]$ is $\labs$-stable.

Conversely, let $B$ strictly 2-Gieseker semistable sheaf and assume that $\ch(B[1])=v$. Let
$$ 0\rightarrow C\rightarrow B\rightarrow D\rightarrow 0 $$
be a destabilizing sequence in $\coh(X)$, so that $C$ is 2-Gieseker stable and $D$ is 2-Gieseker semistable. Then as a polynomial on $\beta$
$$
\frac{\ch_2^{\beta}(C)}{\ch_0(C)}=\frac{\ch_2^{\beta}(B)}{\ch_0(B)}.
$$
Thus, having $\lambda_{\beta,\alpha,s}(B[1])=+\infty$ implies that
$$
\alpha^2=2\frac{\ch_2^{\beta}(B)}{\ch_0(B)}=2\frac{\ch_2^{\beta}(C)}{\ch_0(C)},
$$
and so $\lambda_{\beta,\alpha,s}(C[1])=+\infty$. This says that the sequence
$$
0\rightarrow C[1]\rightarrow B[1]\rightarrow D[1]\rightarrow 0
$$
is exact in $\mathcal{A}^{\obeta,\oalpha}$. Therefore, $B[1]$ is also strictly $\lambda_{\obeta,\oalpha,s}$-semistable.

Finally, if $B^{**}/B$ does not have pure dimension 1, let $Z_B$ be its maximal 0-dimensional subsheaf (compare with the notation in display \eqref{e' sqc}); then
$$ 0 \longrightarrow Z_B \longrightarrow B[1] \longrightarrow B'[1] \longrightarrow 0 $$
is a short exact sequence in $\mathcal{A}^{\obeta,\oalpha}$, so $B[1]$ is not $\lambda_{\obeta,\oalpha,s}$-stable.
\end{proof}

Next, we show that stable triples with Chern character $v$ with $v_0>0$ are Bridgeland stable objects right above the curve $\Theta_v$.

\begin{proposition}\label{prop:st-epsilon}
Fix $s>1/3$. Let $(B,P,\eta)$ be a stable triple with numerical Chern character $v$, and let $(\beta_v,\alpha_v)\in\Theta_v$ be the intersection point of the last tilt wall for $v$. Then $A_\eta$ is $\lambda_{(\beta,\alpha)}$-stable in a chamber immediately above the portion of $\Theta_v$ above $(\beta_v,\alpha_v)$.
\end{proposition}

\begin{proof}
Fix $\obeta<\beta_v$ and pick $\alpha^*$ such that there are no walls for Chern character $v$ crossing the interval $(\obeta,\oalpha)$ and $(\obeta,\alpha^*)$, where $(\obeta,\oalpha)\in\Theta_v$. For a contradiction, assume that $A_\eta$ is not $\lambda_{\obeta,\alpha,s}$-stable for any $\alpha\in(\bar\alpha,\alpha^*)$. 

Let $F\into A_\eta$ be a destabilizing sub-object in $\cala^{\obeta,\alpha'}$ for some $\alpha'\in(\bar\alpha,\alpha^*)$, so that $\lambda_{\obeta,\alpha',s}(F)\ge\lambda_{\obeta,\alpha',s}(A_\eta)$. Consider the composition $\varphi:F\into A_\eta \onto P$ and let $K:=\ker(\varphi)$ and $I:=\im(\varphi)$ as objects in $\cala^{\obeta,\alpha'}$. Set $F_i:=\calh_{\obeta}^{-i}(F)$ and $K_i:=\calh^{-i}_{\obeta}(K)$.

Since $I$ is a sub-object of $P$, and $\coh^{\leq 0}(X)$ is a Serre subcategory of $\cala^{\obeta,\alpha'}$, we conclude that $I$ is also a 0-dimensional sheaf. If $K=0$ we obtain a contradiction with the complete non-triviality of $\eta$. Therefore, we can assume that $K\ne0$. Moreover, $\ch_{\le2}(K)=\ch_{\le2}(F)$ and, since $\calh_{\obeta}^{-1}(I)=0$, we get that $F_1\simeq K_1\into B$ as objects in $\coh^{\obeta}(X)$.

Now, there are two possibilities to be considered. First, assume that $F\in\cala^{\obeta,\alpha}$ for every $\alpha\in(\oalpha,\alpha']$; since there are no walls in this interval, $\lambda_{\obeta,\alpha,s}(F)\ge\lambda_{\obeta,\alpha,s}(A_\eta)$ for every $\alpha\in(\oalpha,\alpha']$. But $\lambda_{\obeta,\alpha,s}(A_\eta)\to\infty$ as $\alpha$ approaches $\oalpha$, thus $\lambda_{\obeta,\alpha,s}(F)\to\infty$ as $\alpha\to\oalpha$ as well. It follows that $\rho_{\obeta,\oalpha}(F)=0$; by \cite[Proposition 3.2]{JM19}, we get that $F_0$ is a 0-dimensional sheaf and $F_1$ is a $\nu_{\obeta,\oalpha}$-semistable object with $\rho_{\obeta,\oalpha}(F_1)=0$. We then have that 
$$ \nu_{\obeta,\oalpha}(F_1) = \nu_{\obeta,\oalpha}(B) = 0; $$
but $B$ is $\nu_{\obeta,\oalpha}$-stable, so $F_1=B$ and we obtain $\ch_{\leq 2}(F)=-\ch_{\leq 2}(F_1)=\ch_{\leq 2}(A_{\eta})$. Moreover, since $K$ is a subobject of $B[1]$ in $\mathcal{A}^{\obeta,\alpha'}$ then if $Q:=B[1]/K$ denotes the corresponding quotient, we get the long exact sequence of cohomologies in $\coh^{\obeta}(X)$
$$
0\rightarrow K_1\rightarrow B\rightarrow \mathcal{H}_{\obeta}^{-1}(Q)\rightarrow K_0\rightarrow 0,
$$
which implies that $K_0\cong \mathcal{H}_{\obeta}^{-1}(Q)$ since the first map is an isomorphism. However, this is impossible unless $K_0=0$, since $K_0\in \mathcal{T}_{\obeta,\alpha'}$ while $\mathcal{H}_{\obeta}^{-1}(Q)\in \mathcal{F}_{\obeta,\alpha'}$. Therefore,
$$ \lambda_{\obeta,\alpha',s}(F) - \lambda_{\obeta,\alpha',s}(A_\eta) = \dfrac{\ch_3(I)-\ch_3(P)}{\rho_{\obeta,\alpha'}(A_\eta)^2} \le 0 $$
contradicting the hypothesis that $F$ destabilizes $A_\eta$ at $\cala^{\obeta,\alpha'}$.

Otherwise, there is $\alpha_t\in(\oalpha,\alpha')$ such that $\rho_{\obeta,\alpha_t}(F)=0$ so $F\notin\cala^{\obeta,\alpha}$ when $\alpha\in(\oalpha,\alpha_t)$. According to \cite[Proposition 3.2]{JM19}, $F_0$ is a 0-dimensional sheaf, thus 
$$ \rho_{\obeta,\alpha_t}(F) = -\rho_{\obeta,\alpha_t}(F_1) + \rho_{\obeta,\alpha_t}(F_0) = 0 ~~ \Longrightarrow ~~ \rho_{\obeta,\alpha_t}(F_1) = 0. $$
Since $\nu_{\obeta,\alpha_t}(B)<0$, it follows that $\nu_{\obeta,\alpha_t}(F_1)>\nu_{\obeta,\alpha_t}(B)$, contradicting the fact that $B$ is $\nu_{\obeta,\alpha}$-stable when $\alpha>\oalpha$ unless $F_1=0$; but this would imply that $F=F_0$ is a 0-dimensional sheaf, again contradicting the complete non-triviality of $\eta$.
\end{proof}


\section{DT/PT wall-crossing} \label{sec:honest-wall}

Let $\mathcal{P}_{\beta,\alpha,s}$ denote the slicing of the stability condition $\sigma_{\beta,\alpha,s}$, so that the double tilted category $\mathcal{A}_{\beta,\alpha}=\mathcal{P}_{\beta,\alpha,s}(0,1]$. For this section, we will fix $s>1/3$ and consider instead the $(\beta,\alpha)$-half plane of stability conditions $\tau_{\beta,\alpha,s}=(\mathcal{B}_{\beta,\alpha},\widehat{Z}_{\beta,\alpha,s})$, where
$$
\mathcal{B}_{\beta,\alpha}:=\mathcal{P}_{\beta,\alpha,s}\left(\frac{1}{2},\frac{3}{2}\right],\ \ \ \widehat{Z}_{\beta,\alpha,s}=\ch_2^{\beta}-\frac{\alpha^2}{2}\ch_0+i\left(\ch_3^{\beta}-\left(s+\frac{1}{6}\right)\alpha^2\ch_1^{\beta}\right),
$$
i.e., $\tau_{\beta,\alpha,s}$ is obtained from $\sigma_{\beta,\alpha,s}$ by the $\mathbb{C}^*$-action. The advantage of considering these stability conditions is that, for a fixed Chern character $v$ with $v_0,v_1$ relatively prime, the curve $\Theta^-(v)$ becomes an honest wall that we can cross. Furthermore, we have 
$$
\mathcal{M}_{\tau_{\beta,\alpha,s}}(v)=\mathcal{M}_{\sigma_{\beta,\alpha,s}}(v),\  \text{for}\  (\beta,\alpha)\ \text{above}\ \Theta^-(v)\  \text{and}\  \beta\ll 0.
$$
Likewise, the shift functor induces an isomorphism
$$
\mathcal{M}_{\tau_{\beta,\alpha,s}}(v)\cong \mathcal{M}_{\sigma_{\beta,\alpha,s}}(-v),\  \text{for}\  (\beta,\alpha)\ \text{below}\ \Theta^-(v),\  \text{above} \ \Gamma_{s}^{-}(v)\  \text{and}\  \beta\ll 0.
$$
We will use this isomorphism to prove that for $(\beta,\alpha)$ below and near $\Theta^-(v)$, the moduli space $\mathcal{M}_{\sigma_{\beta,\alpha,s}}(-v)$ is isomorphic to the Gieseker moduli space $\mathcal{G}(-v)$. 

Let $\widehat{\lambda}_{\beta,\alpha,s}$ be the Bridgeland slope associated to the stability condition $\tau_{\beta,\alpha,s}$. Thus,
$\widehat{\lambda}_{\beta_0,\alpha_0,s}(E)=0$ for each object $E$ with $\ch(E)=v$ and every $(\beta_0,\alpha_0)\in \Theta(v)$. We have shown in Proposition \ref{prop:st-epsilon} that above and near the wall $\Theta(v)$ for $\beta\ll 0$, each stable triple is $\lambda_{\beta,\alpha,s}$-stable and so $\widehat{\lambda}_{\beta,\alpha,s}$-stable; such triple fits into a short exact sequence
$$
0\rightarrow E'[1]\rightarrow A'\rightarrow Z\rightarrow 0,
$$
where $E'$ is a 2-Gieseker stable sheaf of homological dimension 1 and $Z$ is a 0-dimensional sheaf. Since $Z$ is 0-dimensional then $E'[1]$ is a destabilizing object for $A'$ making the curve $\Theta(v)$ an honest wall. Crossing this wall, i.e., near and below $\Theta(v)$, extensions of the form
$$
0\rightarrow Z\rightarrow A\rightarrow E'[1]\rightarrow 0
$$
are $\widehat{\lambda}_{\beta,\alpha,s}$-stable. Taking the long exact sequence of cohomology sheaves we obtain that $A$ is a 2-term complex and, moreover, that $\mathcal{H}^0(A)$ is 0-dimensional (since it is a quotient of $Z$). However, since the exact sequence 
$$
0\rightarrow\mathcal{H}^{-1}(A)[1]\rightarrow A\rightarrow \mathcal{H}^0(A)\rightarrow 0
$$ 
can not destabilize $A$ and $\widetilde{\lambda}_{\beta,\alpha,s}(A)>0=\widehat{\lambda}_{\beta,\alpha,s}(\mathcal{H}^0(A))$, then $\mathcal{H}^0(A)=0$. Therefore, $A=E[1]$ for some 2-Gieseker semistable sheaf $E$ that is, in fact, Gieseker semistable. Indeed, suppose that $F\hookrightarrow E$ is a subsheaf with $\delta_{10}(F,E)=\delta_{20}(F,E)=0$ and $\delta_{30}(F,E)>0$, then setting
$$
U=\frac{\ch_2^{\beta}(F)}{\ch_0(F)}-\frac{\alpha^2}{2}=\frac{\ch_2^{\beta}(E)}{\ch_0(E)}-\frac{\alpha^2}{2},
$$
we have that $U>0$ for $\beta\ll 0$. Moreover,
$$
\frac{1}{\widehat{\lambda}_{\beta,\alpha,s}(F)}-\frac{1}{\widehat{\lambda}_{\beta,\alpha,s}(E)}=-\frac{1}{U}\delta_{30}(F,E)<0,
$$
contradicting the stability of $E$. Thus, together with the results of Proposition \ref{prop:moduli in theta} and Proposition \ref{prop:st-epsilon} we have proved the following:
\begin{theorem} \label{thm:wall-xing}
    Let $v$ be a Chern character vector with $v_0>0$ and such that $v_0$ and $v_1$ are relatively prime. Fix $s>1/3$ and let $(\beta_v,\alpha_v)\in\Theta_v$ be the last tilt wall for $v$. For $(\obeta,\oalpha)\in\Theta_v$ with $\obeta<\beta_v$ the Bridgeland wall-crossing, via restriction to the locus of stable triples, produces a diagram
    $$
    \begin{diagram}
     \node{\mathcal{G}(-v)}\arrow{se,b}{\gamma}\node{}\node{\calt(v)}\arrow{sw,b}{\tau}\\
     \node{}\node{\mathcal{M}_{\obeta,\oalpha,s}(v)}\node{}
    \end{diagram}
    $$
with morphisms given by
$$ \gamma(E) = \big[E'[1]\oplus Z_E\big]  ~~{\rm and}~~ \tau(A) = \big[\calh^{-1}(A)[1]\oplus \calh^{0}(A)\big]. $$
\end{theorem}


\subsection{First example: a collapsing wall}

Let $X$ be a Fano threefold with Picard rank 1, and set $P:=r\cdot P_{\ox}-n$ where $P_{\ox}$ is the Hilbert polynomial of $X$, and $n\ge r\ge1$ are positive integers. The Gieseker moduli space $\calg(P)$ is described in \cite{GJ}: every $E\in\calm_X(P)$ satisfies an exact sequence of the form
\begin{equation} \label{eq:q-trivial}
0 \lra E \lra \ox^{\oplus r} \lra Q \lra 0,
\end{equation}
where $Q$ is 0-dimensional sheaf of length $n$. Notice that $E^{**}\simeq \ox^{\oplus r}$ and every sheaf in $\calg_X(P)$ has homological dimension larger than 1. Moreover $\ch(E)=v_{r,n}:=(r\cdot\deg(X),0,0,-n)$.

Let us now consider the PT moduli space $\calt(-v_{r,n})$. Take $A\in\calp(-v_{r,n})$ and set $B:=\calh^{-1}(A)$ and $T:=\calh^{0}(A)$, so that
$\ch_{\le2}(B)=(r,0,0)$ and $\ch_3(B)=n+\ch_3(T)\ge0$. It follows from \cite[Lemma 2.5]{GJ} that $B$ also satisfies an exact sequence like the one in display \eqref{eq:q-trivial}, so in fact $\ch_3(B)<0$, providing a contradiction. 

In addition, once can similarly deduce that if $A\in\mathcal{M}_{\obeta,\oalpha}(-v_{r,n})$, then $A=[\ox^{\oplus r}[1]\oplus Q]$, so in fact $\mathcal{M}_{\obeta,\oalpha}(-v_{r,n})$ coincides with ${\rm Sym}^n(Q)$, the $n^{\rm th}$-symmetric product of $X$.

Therefore, we conclude that $\calt(-v_{r,n})$ is empty, and the Brigeland wall described in Theorem \ref{thm:wall-xing} is a collapsing wall.


\subsection{Second example: a fake wall}

Let $X$ be a Fano threefold of Picard rank 1 for which Bridgeland stability conditions exist. Fix $w=(w_0,w_1,w_2)\in\Z\times\Z\times\frac{1}{2}\Z$, and let $\tilde{w}=(w_0,w_1,w_2,\varepsilon_X(w))$, where $\varepsilon_X(w)$ is the constant defined in display \eqref{eq:epsilon(v)}.

Let us consider the PT moduli space $\calp(-\tilde{w})$. Take $A\in\calp(-\tilde{w})$ and set $B:=\calh^{-1}(A)$ and $T:=\calh^{0}(A)$, so that
$\ch_{\le2}(B)=w$ and 
$$ \ch_3(B)=\varepsilon_X(w)+\ch_3(T)\ge\varepsilon_X(w). $$
Since $B$ is a 2-Gieseker semistable sheaf, we get, by definition, that $\ch_3(B)\le\varepsilon_X(w)$. It follows that $\ch_3(B)=\varepsilon_X(w)$ and $\ch_3(T)=0$, so in fact $A=B[1]$ and so the object is stable on both sides of the wall.

Therefore, $\calg(\tilde{w})=\calt(\tilde{w})$ meaning that the wall described in Theorem \ref{thm:wall-xing} is not an actual wall, since it does not affect the moduli spaces.


\subsection{Third example: an honest wall}

Set $v=(2,-1,-1/2,-1/6)$, and take $A\in\calp(-v)$, setting $B:=\calh^{-1}(A)$ and $T:=\calh^{0}(A)$, as usual. We get that $\ch_{\le2}(B)=(2,-1,-1/2)$ and
$$ \ch_3(T)=\ch_3(B)+\frac{1}{6}. $$

Since $\gcd(\ch_0(B),\ch_1(B))=1$, we get that $B$ is a $\mu$-stable sheaf. Since $\ch_{\le2}(B)=(2,-1,-1/2)$, then forcibly $\ch_3(B)\le 5/6$, see \cite{S21}. If $B$ is torsion-free but not reflexive, have the exact sequence
$$ 0 \to B \to B^{**} \to Q_B \to 0 $$
with $B^{**}$ being a $\mu$-stable rank 2 reflexive sheaf with $\ch_1(B)=-1$. Note that $\ch_2(Q_B)=\ch_2(B^{**})+1/2\ge0$; on the other hand, Bogomolov inequality yields $\ch_2(B^{**})\le1/4$. Given that $\ch_2(B^{**})\in\frac{1}{2}\Z$ (because $\ch_1(B^{**})=-1$) and $\ch_2(Q_B)\in\Z$, the only possibility is to have $\ch_2(B^{**})=-1/2$, so $\dim Q_B=0$. However, by definition, $Q_B$ has pure dimension 1, so it follows that $Q_B=0$ and $B$ is a reflexive sheaf. 

From \cite[Lemma 9.4]{H-reflexive}, every $B$ is a $\mu$-stable rank 2 reflexive sheaf with $\ch_{\le2}(B)=(2,-1,-1/2)$ has $\ch_3(B)=5/6$, and it is given by an exact sequence of the form
$$ 0 \longrightarrow \op3(-2) \longrightarrow \op3(-1)^{\oplus 3} \longrightarrow B \longrightarrow 0, $$
so that $\sing(B)$ consists of a single point, say $p\in\p3$, which is given by the intersection of the three linear forms in the morphism $\op3(-2)\to\op3(-1)^{\oplus 3}$. 

It then follows that the sheaf $\ch_3(T)=1$, while Lemma \ref{lem:sup-sing} implies that $T\simeq\calo_p$. Since $A$ is not a shifted sheaf, we can conclude that $\Theta_v$ is an actual wall.

Let us further understand the maps provided by Theorem \ref{thm:wall-xing}.

The Gieseker moduli space $\calg(v)$ consists of sheaves $E$ fitting into the following short exact sequence
$$ 0 \lra E \lra F \lra \calo_q \lra 0 $$
where $F$ is a $\mu$-stable reflexive sheaf with $\ch(F)=(2,-1,-1/2,5/6)$ and $q\in\p3$ is an arbitrary point. Clearly, $F=E'=E^{**}$, $Z_E=\calo_q$ and $\hd(E')=1$; as explained above, the sheaf $F$ has only one singular point, which we denote by $f$.

It follows that every sheaf $E$ such that $\supp(Z_E)\ne\sing(F)$ (that is $q\ne f$) gets killed by the wall at $\Theta_v$.


\section{Finiteness of Bridgeland walls}\label{sec:y}

In this section, we show that there are only finitely many actual $\lambda$-walls for $\alpha>\alpha_0$, where $\alpha_0$ is any positive real number. This will imply that there is some $\beta_0$ such that no walls intersect the line $\beta=\beta_1$ for any $\beta_1<\beta_0$. We start by showing that a particular type of pseudo $\lambda$-wall cannot exist.
Throughout this section We fix a Chern character $v$ with $\ch_0(v)>0$ and $\Delta(v)\geq0$.
The key to this section is Theorem \cite[Lemma 4.17]{JM19} which tells us that any actual $\lambda$-wall is a piecewise path of pseudo-walls.
\begin{lemma}\label{l:nowall}
Suppose $\beta_0$ is less than the smallest $\beta$ coordinate of the intersection of the largest actual tilt wall associated with $v$ with the $\beta$-axis. Then if a pseudo $\lambda$-wall $\Upsilon_{u,v}$ passes above $\Theta^-_v$ has a portion of actual $\lambda$-wall including a point with coordinate $\beta_0$, then it cannot intersect $\Theta^-_v$ at a point with $\beta<\beta_0$.    
\end{lemma}
\begin{proof}
Suppose otherwise. Let $(\beta_0,\alpha_0)\in\Upsilon$. Then consider the destabilising sequence in $\cala^{\beta_0,\alpha_0}$ associated to a destabilized object $E[1]$, where $E\in \calb^{\beta_0}$, which must be tilt     stable for all $\alpha>0$, has Chern character $v$:
\[
0\to F\to E[1]\to G[1]\to 0.
\]
Taking cohomology in $\calb^{\beta_0}$ we have the long exact sequence
\[0\to F_1\to E\to G\to F_0\to 0,\]
and split it into two short exact sequences via $E\onto C\hookrightarrow G$ in $\calb^{\beta_0}$. 
 
By hypothesis, the pseudo $\lambda$-wall intersects $\Theta^-_v$ to the left of $\beta_0$ and so there is a pseudo tilt wall for $v$, lying above $(\beta_0,\alpha_0)$. Let the $\alpha$-coordinate where $\beta=\beta_0$ intersects this pseudo tilt wall be $\alpha_1>\alpha_0$. Since $E$ is tilt stable, we have 
\[\nu_{\beta_0,\alpha}(F_1)<\nu_{\beta_0,\alpha}(E)<\nu_{\beta_0,\alpha}(C)\]
for all $\alpha>0$. But $\nu_{\beta_0,\alpha_1}(E)=\nu_{\beta_0,\alpha_1}(G)$ and so $\nu_{\beta_0,\alpha_1}(C)>\nu_{\beta_0,\alpha_1}(G)>\nu_{\beta_0,\alpha_1}(F_0)$. 
It follows that $C$ provides an actual tilt wall for $G$ which intersects the line segment $\alpha_0<\alpha<\alpha_1$, $\beta=\beta_0$. Then this tilt wall remains actual for its whole length and, in particular, when it intersects $\Theta^-_v$. At this point, it must intersect an actual tilt wall for $E$ which is a contradiction.
\end{proof}
We also need a further lemma. This relies on the observation that if a wall has a discontinuity so that there are two or more separate destabilisers at a point of the plane then the two numerical walls must intersect so that they extend as non-actual walls on the destabilizing side of the wall.
\begin{lemma}\label{l:maxmin}
For $s=1/3$. An actual $\lambda$-wall $W_{u,v}$ corresponding to an asymptotically stable object in $R^-$ or $R^+$ has maxima precisely on $\Gamma_{v,1/3}$ and in $R^0$ has minima precisely on tilt walls.
\end{lemma}
\begin{proof}
 By \cite[Prop~6.10]{JM19} this holds for numerical walls. The result follows from the following observation. At a singularity of $\Upsilon_{u,v}$, the actual portion of the two intersecting numerical walls must have the continuation of the numerical walls away from the region where the object is stable. This means that the vertical direction of the actual wall cannot change at a discontinuity when two numerical walls cross. The reason for this is because, in $R^-$, destabilization occurs from left to right along horizontal lines or top-down vertically from $\Theta^-_v$ (and mirror symmetrically in $R^+$), while in $R^{0}$ destabilisation occurs vertically upwards from $\Theta^-_{v}$ by Proposition \ref{prop:st-epsilon}. This, in turn, follows from the asymptotics in $R^-$ and the characterisation of stable triples in Section \ref{sec:triples}.
\end{proof}
The upshot of this lemma is that the piecewise walls destabilising asymptotically stable objects roughly follow the same geometry as numerical walls.
\begin{theorem}\label{t:finite}
For any $s\geq1/3$, $\alpha_0>0$ and any Chern character $v$ with $\Delta(v)\geq0$ and $v_0>0$, there are only finitely many actual $\lambda$-walls passing through the upper $(\beta,\alpha)$ plane for which $\alpha>\alpha_0$ which destabilize Geiseker stable objects, PT stable objects or their duals.
\end{theorem}
\begin{proof}
It suffices by \cite[Theorem 6.12]{JM19} to prove the Theorem in the case $s=1/3$ as no new walls are created as $s$ increases.

Consider first actual walls which intersect $\Theta^-_v$. But by \cite[Theorem 4.21]{JM19}, an actual $\lambda$-wall only crosses $\Theta_v$ at an actual tilt wall. There are only finitely many tilt walls above $\alpha_0$ and in a compact region around such an intersection point there can only be finitely many $\lambda$-walls. So there are only finitely many actual $\lambda$-walls in all which cross $\Theta^-_v$ above $\alpha_0$. 

Now consider the set of $\lambda$-walls which stay inside $R^-_v$, and which do not intersect an actual tilt wall. It suffices to consider those which destabilize asymptotically stable objects. From Lemma \ref{l:maxmin}, the proof of \cite[Lemma 7.16]{JM19} shows that any such wall must intersect $\Gamma^-_{v,s}$ and by \cite[Theorem 7.8]{JM19} it can only intersect once as an actual wall. Hence all walls are trapped in the region between $\Gamma^-_{v,s}$ and $\Theta^-_v$. If these two curves do not intersect, then the walls must all pass through a compact region, and hence there are finitely many. 

Now consider actual $\lambda$-walls above $\Theta^-_v$. This time there are several cases to consider. An actual $\lambda$-wall might exit the region $R^0$ in a number of ways:
\begin{enumerate}
\item Through $\Theta_v$
\item Through $\alpha=0$ between the branches of $\Theta_v$
\item Vertically (as an unbounded wall)
\item Asymptotically to the left and right above $\Theta_v$.
\end{enumerate}
In cases (1) or (2) there can only be finitely many. In the first case, since such a wall would intersect an actual tilt wall and, as before, there can only be finitely many intersecting each of finitely many tilt walls. In the second case, the curves must pass through a compact region between the branches of $\Theta_v$ and so again there can only be finitely many. We claim that situation (4) cannot happen. Since the asymptotically stable objects are also the objects that are stable near $\Theta^-_v$, it suffices to consider walls that destabilize these objects. Such an actual wall must be piecewise and consist of infinitely many numerical walls. If one of these numerical walls has a maximum to the left of where it is actual, then we are in the situation of Lemma \ref{l:nowall} above since any such wall must, at some
point, extend beyond the $\beta_0$ of the lemma. So we can assume all numerical walls making up the actual wall must double back above the actual portion. The proof of Lemma \ref{l:maxmin} implies that the loops to the left cross inside the loops to the right. By looking at a vertical line as in the proof of Lemma \ref{l:nowall}, we see that there must also be an actual portion of the same numerical wall above as the loop doubles back. If there are infinitely many such loops then this vertical line would cut such numerical walls infinitely many times. But the objects stay in $\cohab$ since otherwise there would be a $\Theta$-wall for one of them and we could argue as in the proof of Lemma \ref{l:nowall} that there would need to be an actual tilt wall. But in $\cohab$ there are only finitely many Harder--Narasimhan factors of $E$ and so there can only be finitely many numerical walls along the vertical line. But Lemma \ref{l:maxmin} shows that the lower portion of such a wall cannot have a minimum and so exits through $\alpha=\alpha_0$.  

So we have shown that there are only finitely many bounded $\lambda$-walls above $\alpha_0$. Note that Lemma \ref{l:maxmin} shows that there can be no piecewise unbounded walls where both ends exit vertically. Hence all such walls exit through $\alpha=\alpha_0$. If we pick an $\alpha_1$ so that all bounded $\lambda$-walls are contained in $\alpha<\alpha_1$ then there are only finitely many unbounded $\lambda$-walls in the region $\alpha_1\leq\alpha\leq\alpha_1+1$ between the branches of $\Theta_v$.  This deals with case (3).
\end{proof}

Combining Theorem \ref{t:finite} with Theorem \ref{thm:AG53-39-1} and Theorem \ref{thm:PTislimtstab-eff}, and noting that the condition $s>1/3$ is equivalent to $3f_2-g_2>0$ in Theorem \ref{thm:PTislimtstab-eff}, we immediately deduce the following statement.
\begin{theorem}\label{t:PTisBridgeland}
Let $X$ be a smooth projective threefold of Picard rank $1$ for which $\stab(X)$ is non-empty. Then $E\in D^b(X)$ with $\ch_0(E)>0$ is PT-(semi)stable if and only if $E$ is Bridgeland (semi)stable in the chamber immediately above $\Theta^-_{\ch(E)}$ and the last actual tilt wall.
\end{theorem}
In other words, positive rank PT-stable objects are Bridgeland stable in some region of $\stab(X)$, which is shown in Figure \ref{fig:intro}. In particular, if we let $\rho_v(\beta,\alpha)=0$ denote the equation defining $\Theta^-_v$ in which the coefficient of $\alpha^2$ is positive. Then there is a $\beta_0$ such that the curvilinear unbounded triangular region
\[\{(\beta,\alpha)\mid \rho_v(\beta,\alpha)>0,\ \beta<\beta_0\}\]
is contained in the PT chamber.

\printbibliography
\end{document}